\documentclass[a4paper,10pt]{amsart}
\usepackage[T1]{fontenc}
\usepackage{lmodern}

\usepackage{latexsym}
\usepackage{amsmath}
\usepackage{amsfonts}
\usepackage{amssymb}
\usepackage{amsxtra}
\usepackage{amsthm}
\usepackage{thmtools}
\usepackage{enumitem}
\usepackage{mathtools}
\usepackage{xcolor}
\usepackage{tikz}
\usepackage{tikz-cd}
\usepackage[pdftitle={Hilbert Schemes of Points in the Plane and Quasi-Lisse Vertex Algebras with N=4 Symmetry}]{hyperref}

\theoremstyle{definition}
\newtheorem{definition}{Definition}[section]
\newtheorem{remark}[definition]{Remark}
\theoremstyle{plain}
\newtheorem*{maintheorem*}{Main Theorem}
\newtheorem*{theorem*}{Theorem}
\newtheorem{proposition}[definition]{Proposition}
\newtheorem{theorem}[definition]{Theorem}
\newtheorem{corollary}[definition]{Corollary}
\newtheorem{lemma}[definition]{Lemma}

\newcommand{\Id}{\operatorname{Id}}

\DeclareMathOperator{\Hilb}{Hilb}

\newcommand{\VPA}{\mathrm{VPA}}
\newcommand{\VA}{\hbar\mathrm{VA}}

\newcommand{\tldU}{\widetilde{U}}
\newcommand{\tldfrU}{\widetilde{U}}
\newcommand{\tldcalD}{\widetilde{\calD}}
\newcommand{\tldmu}{\widetilde{\mu}}

\newcommand{\Swt}{\operatorname{t-wt}}
\newcommand{\Abar}{R}

\DeclareMathOperator{\Com}{Com} 

\DeclareMathOperator{\GL}{GL}
\newcommand{\tldcalO}{\widetilde{\calO}}
\newcommand{\JetBundle}[1]{\tldcalO_{J_\infty #1}}
\newcommand{\jet}[2]{J_{#1}{#2}}
\newcommand{\Cl}{C\ell}
\newcommand{\tldA}{\widetilde{A}}
\newcommand{\tldC}{\widetilde{C}}
\newcommand{\tldE}{\widetilde{E}}

\DeclareMathOperator{\hatotimes}{\widehat{\otimes}}

\newcommand{\Koszul}{\mathrm{Kosz}}
\newcommand{\DR}{\mathrm{dR}}
\newcommand{\Alg}{\mathrm{Alg}}

\newcommand{\Scheme}{\mathrm{Scheme}}
\newcommand{\ch}{\mathrm{ch}}
\newcommand{\vertex}{\mathrm{vert}}
\newcommand{\TN}{T_{\calN=4}}
\newcommand{\TTr}{T_{\beta\gamma}}
\newcommand{\TSF}{T_{\sf{SF}}}

\newcommand{\dE}{{}^{\tau} \!E} 
\newcommand{\HSF}{{}^{\mathrm{HS}} \!F} 
\newcommand{\HSGr}{{}^{\mathrm{HS}} \!\Gr} 
\newcommand{\HSE}{{}^{\mathrm{HS}} \!E} 
\newcommand{\hF}{{}^{\hbar} \!F}
\newcommand{\hGr}{{}^{\hbar} \!\Gr}
\newcommand{\hE}{{}^{\hbar} \!E}

\newcommand{\tldG}{\widetilde{G}}
\newcommand{\hatsl}{\widehat{\frsl}}
\newcommand{\red}{\mathrm{red}}

\newcommand{\Wbeta}{\widetilde{\beta}}
\newcommand{\Wgamma}{\widetilde{\gamma}}
\newcommand{\Wb}{\widetilde{b}}
\newcommand{\Wc}{\widetilde{c}}

\newcommand{\SF}{\mathsf{SF}}

\newcommand{\C}{{\mathbb C}}

\newcommand{\Z}{{\mathbb Z}}

\newcommand{\bbP}{{\mathbb P}}

\newcommand{\bbV}{{\mathbb V}}

\newcommand{\bfe}{\mathbf{e}}

\newcommand{\calA}{\mathcal{A}}

\newcommand{\calD}{\mathcal{D}}

\newcommand{\calF}{\mathcal{F}}

\newcommand{\calH}{\mathcal{H}}
\newcommand{\calI}{\mathcal{I}}

\newcommand{\calM}{\mathcal{M}}
\newcommand{\calN}{\mathcal{N}}
\newcommand{\calO}{\mathcal{O}}

\newcommand{\calT}{\mathcal{T}}

\newcommand{\calV}{\mathcal{V}}

\newcommand{\frX}{\mathfrak{X}}

\newcommand{\frg}{\mathfrak{g}}

\newcommand{\frm}{\mathfrak{m}}

\newcommand{\frgl}{\mathfrak{gl}}
\newcommand{\frsl}{\mathfrak{sl}}

\newcommand{\sfV}{\mathsf{V}}
\newcommand{\sfW}{\mathsf{W}}

\DeclareMathOperator{\Tr}{Tr}
\DeclareMathOperator{\gEnd}{End}
\DeclareMathOperator{\lEnd}{\operatorname{\mathscr{E}\kern-.1pc\mathit{nd}}}
\DeclareMathOperator{\gHom}{Hom}
\DeclareMathOperator{\lHom}{\operatorname{\mathscr{H}\kern-.1pc\mathit{om}}}

\DeclareMathOperator{\Spec}{Spec}

\DeclareMathOperator{\Proj}{Proj}
\DeclareMathOperator{\Lie}{Lie}

\DeclareMathOperator{\Res}{res}
\DeclareMathOperator{\Dim}{dim}
\DeclareMathOperator{\Ker}{Ker}
\renewcommand{\Im}{\operatorname{Im}}

\DeclareMathOperator{\gr}{Gr}
\DeclareMathOperator{\Gr}{Gr}

\newcommand{\fin}{\mathrm{fin}}

\newcommand{\blkbar}{\raisebox{0.5ex}{\rule{2ex}{0.4pt}}}
\newcommand{\qquot}{/\!\!/}

\newcommand{\bfone}{\mathbf{1}}
\newcommand{\NO}{{:}} 

\newcommand{\scbul}{{\,\raise1pt\hbox{$\scriptscriptstyle\bullet$}\,}}

\newcommand{\on}[1]{\operatorname{#1}}
\newcommand{\+}{\oplus}

\newcommand{\W}{\mathsf{W}}

\newcommand{\ii}{\mathrm{i}}
\newcommand{\e}{\mathrm{e}}
\DeclareMathOperator{\SL}{SL}
\DeclareMathOperator{\rk}{rk}
\newcommand{\wt}{\operatorname{wt}}

\begin{document}

\title[Hilbert Schemes and Vertex Algebras with $\calN=4$ Symmetry]{Hilbert Schemes of Points in the Plane and Quasi-Lisse Vertex Algebras with $\calN=4$ Symmetry}
\author[Tomoyuki Arakawa, Toshiro Kuwabara, Sven Möller]{Tomoyuki Arakawa,\textsuperscript{\lowercase{a},\lowercase{b}} Toshiro Kuwabara\textsuperscript{\lowercase{c}} and Sven Möller\textsuperscript{\lowercase{b},\lowercase{d}}}
\thanks{\textsuperscript{a}{Ningbo University, Ningbo, China}}
\thanks{\textsuperscript{b}{Research Institute for Mathematical Sciences, Kyoto University, Kyoto, Japan}}
\thanks{\textsuperscript{c}{University of Tsukuba, Tsukuba, Japan}}
\thanks{\textsuperscript{d}{Universität Hamburg, Hamburg, Germany}}
\thanks{Email: \href{mailto:arakawa@kurims.kyoto-u.ac.jp}{\nolinkurl{arakawa@kurims.kyoto-u.ac.jp}}, \href{mailto:toshiro.kuwa@gmail.com}{\nolinkurl{toshiro.kuwa@gmail.com}}, \href{mailto:math@moeller-sven.de}{\nolinkurl{math@moeller-sven.de}}}

\begin{abstract}
To each complex reflection group $\Gamma$ one can attach a canonical symplectic singularity $\calM_{\Gamma}$ \cite{Bea00}. Motivated by the 4D/2D duality \cite{BeeLemLie15,BeeRas}, Bonetti, Meneghelli and Rastelli \cite{BMR19} conjectured the existence of a supersymmetric vertex operator superalgebra $\W_{\Gamma}$ whose associated variety is isomorphic to $\calM_{\Gamma}$. We prove this conjecture when the complex reflection group $\Gamma$ is the symmetric group $S_N$ by constructing a sheaf of $\hbar$-adic vertex operator superalgebras on the Hilbert scheme of $N$ points in the plane. For that case, we also show the free-field realisation of $\W_{\Gamma}$ in terms of $\rk(\Gamma)$ many $\beta\gamma bc$-systems proposed in \cite{BMR19}, and identify the character of $\W_{\Gamma}$ as a certain quasimodular form of mixed weight and multiple $q$-zeta value.

In physical terms, the vertex operator superalgebra $\W_{S_N}$ constructed in this article corresponds via the 4D/2D duality \cite{BeeLemLie15} to the four-di\-men\-sion\-al $\calN=4$ supersymmetric Yang-Mills theory with gauge group $\SL_N$.
\end{abstract}

\maketitle

\setcounter{tocdepth}{1}
\tableofcontents
\setcounter{tocdepth}{2}


\section{Introduction}
\label{sec:intro}
Let $\Gamma$ be a complex reflection group of rank $\rk(\Gamma)$. By definition, $\Gamma$ is a subgroup of $\GL(V_\Gamma)$ that is generated by reflections, where $V_\Gamma=\C^{\rk(\Gamma)}$. The group $\Gamma$ acts diagonally on the symplectic vector space $T^*V_\Gamma=V_\Gamma\+V_\Gamma^*$, preserving its symplectic form, and it is known that the orbit space
\[
\calM_\Gamma\coloneqq(V_\Gamma\+ V_\Gamma^*)/\Gamma=\Spec\C[V_\Gamma\+ V_\Gamma^*]^\Gamma
\]
has a symplectic singularity \cite{Bea00}.

Motivated by the 4D/2D duality \cite{BeeLemLie15,BeeRas} and the free-field realisation in \cite{Adamovic16}, Bonetti, Meneghelli and Rastelli \cite{BMR19} conjectured the existence of a vertex operator superalgebra $\W_\Gamma$ of central charge $c_\Gamma\coloneqq-3\sum_{i=1}^{\rk(\Gamma)}(2p_i-1)$ equipped with $\calN=2$ supersymmetry, such that, among other things,
\begin{equation}\label{eq:ass-variety-with-Higgs}
X_{\W_\Gamma}\cong \calM_\Gamma,
\end{equation}
where $p_1,\dots ,p_{\rk(\Gamma)}$ are the degrees of the fundamental invariants of $\C[V_\Gamma]^\Gamma$ and $X_V$ denotes the associated variety of a vertex superalgebra $V$ \cite{Arakawa12}.

When $\Gamma$ is even a Coxeter group, \cite{BMR19} further conjectured that the $\calN=2$ symmetry of $\W_\Gamma$ should be enhanced to the small $\calN=4$ supersymmetry. That is, $\W_\Gamma$ should be a conformal extension of the small $\calN=4$ superconformal algebra $\on{Vir}_{\calN=4}^{c_\Gamma}$ of level $c_\Gamma/6$, a vertex operator superalgebra of central charge $c_\Gamma$ \cite{Kac,FK02,KW04,KRW03,Ara05}.

Note that \eqref{eq:ass-variety-with-Higgs} implies that the vertex operator superalgebra $\W_\Gamma$ is quasi-lisse in the sense of \cite{AraKaw18} since a normal variety with symplectic singularities has only finitely many symplectic leaves \cite{Kaledin:2006kq}.

\medskip

We now mention the meaning of $\W_\Gamma$ in the context of the 4D/2D duality
\[
\bbV\colon\{\text{4D $\calN=2$ SCFTs}\}\longrightarrow \{\text{VOAs}\},
\]
which associates a vertex operator superalgebra $\bbV(\calT)$ with each four-dimensional $\calN=2$ superconformal field theory $\calT$ \cite{BeeLemLie15}. It is believed that $\bbV$ is a finite map and that the associated variety of $\bbV(\calT)$ satisfies
\begin{equation}\label{eq:Hiigs-branch}
X_{\bbV(\calT)}\cong \on{Higgs}(\calT),
\end{equation}
where $\on{Higgs}(\calT)$ is the Higgs branch of $\calT$, as conjectured in \cite{BeeRas} (see \cite{A.Higgs,Ara18} for mathematical reviews). Since the Higgs branch of $\calT$ is expected to have a symplectic singularity, this Higgs branch conjecture \eqref{eq:Hiigs-branch} would imply that the vertex operator superalgebras $\bbV(\calT)$ are quasi-lisse, and it is moreover believed that they are of CFT-type, simple and strongly finitely generated.

Since the seminal work \cite{BeeLemLie15}, there have been considerable efforts to study examples of vertex operator superalgebras appearing in the image of the 4D/2D correspondence: see
\cite{BeeLemLie15, Beem:2015yu, LemPee15, ChrRee15, CorSha16, LemLie16, NisTac16, BuiNis16, CorGaiSha16, LemLieMen17, Son17, CorGaiSha17, SonXieYan17, BuiLacNis17, ItoShu17, NeiYan17, AM15, AraKaw18, BeeRas, BonRas18, FrePeiYan18, FluSon18, ChoNis18, BuiLac1802, AraMor16b, Cre18, WanXie1907, BMR19, NisSasZhu19, AgaLeeSon19, BuiLac19, BeeMenRas19, OhYag19, Jeo19, DedGukNak20, AugCreKan20, EagLocSha2001, DedFlu20, BeeMenPee20, WatZhu20, BiaLem20, FodZhu20, PanPee20, Lem20, XieYanYau2103, XieYan21, XieYan21W, XieYan21b, Ded21, AdaCreGen21, Yan21, Li21, KiyNis21, BriKasMil21, AgaAndKan2106, BeeMen2109, KaiMar2110, FRENEI21, KLS21, PWZ22, KozShaYan22, KawRid22, DedWan22, BeePee22, BuiNis22, KawRidWoo22, BSR22, KaiMarRas22, ZhePanWan2211, AEkeren19, BeeNai23, BeeNai23free, ArgLotWea23, LiXieYan23, PanWan2310, GuoLiPan2311, AraMT, MirLoc, AraEkeMor23, LiLiYan23, Fur23, ShaXieYan23, BuiJia23, Li23, RasRay23, ArdMarRos23, NawPanZhe23, ChaLinWu23}.

Now, according to \cite{BMR19}, in the case where the Coxeter group $\Gamma$ is the Weyl group $W(\frg)$ of a simple Lie algebra $\frg$, a central property of the conjectured vertex operator superalgebra $\W_{W(\frg)}$ is that
\[
\W_{W(\frg)}\cong\bbV(\on{SYM}_{\frg}),
\]
where $\on{SYM}_{\frg}$ is the four-dimensional $\calN=4$ supersymmetric Yang-Mills theory with gauge Lie algebra $\frg$. Since
\[
\on{Higgs}(\on{SYM}_{\frg})\cong\calM_{W(\frg)},
\]
the isomorphism $X_{\W_{\Gamma}}\cong\calM_{\Gamma}$ in \eqref{eq:ass-variety-with-Higgs} is exactly the Higgs branch conjecture \eqref{eq:Hiigs-branch} of Beem and Rastelli for the four-dimensional theory $\calT=\on{SYM}_{\frg}$.

However, despite of their importance, the vertex operator superalgebras $\W_\Gamma$ have been studied very little since the paper \cite{BMR19} appeared. The main difficulty is that $\W_\Gamma$ is in general expected to be a W-algebra in the sense that it is not generated by a Lie algebra, and therefore one cannot expect to define $\W_\Gamma$ by generators and relations (operator product expansions) in a closed form.

\medskip

In this article, we prove the conjecture of Bonetti, Meneghelli and Rastelli \cite{BMR19} in the case where $\Gamma$ is the symmetric group $S_N=W(\frsl_N)$, corresponding to the four-dimensional $\calN=4$ supersymmetric Yang-Mills theory with gauge group $\SL_N$. That is, we construct a family of vertex operator superalgebras $\W_{S_N}$, $N\ge2$, with the desired properties. In this situation, the symplectic variety $\calM_{S_N}$ is essentially\footnote{To be precise, $\C^{2N}/S_N\cong\calM_{S_N}\times T^*\C$ as symplectic varieties.} the $N$-th symmetric power $\C^{2N}/S_N$ of the affine plane $\C^2$, and it is well-known that the singular variety $M_0\coloneqq\C^{2N}/S_N$ admits the conical symplectic resolution
\[
\Hilb^N(\C^2)\longrightarrow\C^{2N}/S_N,
\]
where $M\coloneqq\Hilb^N(\C^2)$ is the Hilbert scheme of $N$ points in the affine plane.

In the classical setting, Kashiwara and Rouquier \cite{KR08} (see also \cite{BFM06} for a result in positive characteristic) constructed a sheaf $\calA_{\hbar}$ of $\hbar$-adic algebras on the Hilbert scheme $M$ such that the algebra $[\calA_{\hbar}(M)]^{\C^\times}$ of its global sections (and taking the invariants in some sense under a certain action of the one-dimensional torus $\C^\times$) is isomorphic to the spherical rational Cherednik algebra \cite{EG02} associated with $S_N$, which is a natural quantisation of the coordinate ring of $M_0=\C^{2N}/S_N$. The construction of \cite{KR08} is now regarded as an example of a quantisation of a conical symplectic resolution studied in \cite{BLPW16a,BLPW16b}.

In \cite{AKM15}, the notion of a sheaf $\calA_{\hbar}$ of $\hbar$-adic algebras was upgraded to that of a sheaf of $\hbar$-adic vertex algebras \cite{Li04}. Therefore, it is natural to try to apply the method in \cite{AKM15} to $M=\Hilb^N(\C^2)$ and chiralise the construction of \cite{KR08}. However, this straightforward attempt does not work for the Hilbert scheme $M$ due to some obstruction in constructing sheaves of vertex algebras (cf.\ \cite{GMS04}).

In order to overcome this problem, we replace $M$ with structure sheaf $\calO_M$ by the supervariety whose underlying topological space is again $M$ but whose structure sheaf is a certain superalgebra analogue $\tldcalO_M$ of $\calO_M$. In terms of the Nakajima quiver variety description \cite{Nak94,Nak98,Nakajima} of the Hilbert scheme $M$, this corresponds to replacing the vector space $\C$ on the framing vertex of the Jordan quiver by the superspace $\C^{1|1}$ (see \autoref{sec:superHilb} for details):
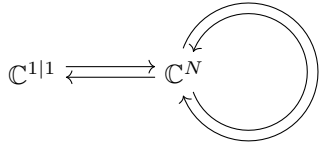
\[
\begin{tikzpicture}
\node (1) at (0,0) {$\C^{1|1}$};
\node (2) at (2,0) {$\C^N$};
\begin{scope}[transform canvas={yshift=2pt}]
\draw[->] (1) to (2);
\end{scope}
\begin{scope}[transform canvas={yshift=-2pt}]
\draw[->] (2) to (1);
\end{scope}
\draw[->] ([shift={(24pt,0)}]2.center) +(20:-22pt) arc (20:340:-22pt);
\draw[->] ([shift={(24pt,0)}]2.center) +(340:-26pt) arc (340:20:-26pt);
\end{tikzpicture}
\]
Then, the machinery of \cite{AKM15} becomes applicable and we are able to construct a sheaf $\tldcalD^\ch_{M,\hbar}$ of $\hbar$-adic vertex superalgebras on $M$ quantising the $\infty$-jet bundle $\JetBundle{M} = J_{\infty} \tldcalO_M$ associated with the sheaf $\tldcalO_M$.

The space of global sections $\tldcalD^\ch_{M, \hbar}(M)$ is naturally equipped with the structure of an $\hbar$-adic vertex operator superalgebra with a certain action of the one-dimensional torus $\C^\times$. To obtain a usual vertex operator algebra, we consider the (in some sense) $\C^\times$-invariant subalgebra $[\tldcalD^\ch_{M, \hbar}(M)]^{\C^\times}$ of $\tldcalD^\ch_{M, \hbar}(M)$ (for details, see \autoref{sec:global-VA}). Then the vertex operator superalgebra $[\tldcalD^\ch_{M, \hbar}(M)]^{\C^\times}$ is essentially the conjectured vertex operator superalgebra $\W_{S_N}$ (see \autoref{thm:W-prop}):
\begin{maintheorem*}
For $N\ge 2$, there exists a sheaf $\tldcalD^\ch_{M,\hbar}$ of $\hbar$-adic vertex operator superalgebras of central charge $c=-3N^2$ on the Hilbert scheme $M=\Hilb^N(\C^2)$ such that $\gr\tldcalD^\ch_{M,\hbar}\cong\pi_*\calO_{J_{\infty}M_{\calV}}[[\hbar]]$ and whose global sections $\sfV_{S_N}\coloneqq [\tldcalD^\ch_{M, \hbar}(M)]^{\C^\times}$ have the associated variety
\[
X_{\sfV_{S_N}}\cong\C^{2N}/S_N.
\]
Moreover, the vertex operator superalgebra $\sfV_{S_N}$ decomposes as a tensor product
\[
\sfV_{S_N}\cong\W_{S_N}\otimes\beta\gamma\otimes\SF,
\]
where $\W_{S_N}$ is a conformal extension of some quotient of the
small $\calN=4$ superconformal algebra $\on{Vir}_{\calN=4}^{c_{S_N}}$ of central charge $c_{S_N}=-3(N^2-1)$, and
\[
X_{\W_{S_N}}\cong\calM_{S_N}.
\]
The vertex operator superalgebras $\sfV_{S_N}$ and $\W_{S_N}$ are of CFT-type and quasi-lisse.
\end{maintheorem*}
Here, $\beta\gamma$ denotes the $\beta\gamma$-system vertex algebra and $\SF$ the symplectic fermion vertex superalgebra, and in the theorem they are endowed with conformal structures of central charges $c=-1$ and $c=-2$, respectively.

We suspect that $\sfV_{S_N}$, $\W_{S_N}$ and the quotient of the small $\calN=4$ superconformal algebra inside $\W_{S_N}$ are simple. This is true for $N=2$ (cf.\ \cite{GMS05}), where we identify $\W_{S_N}$ as the simple small $\calN=4$ superconformal algebra of level $k=-3/2$ \cite{Adamovic16}.

The global sections containing the tensor factor $\beta\gamma\otimes\SF$ is merely an artefact due to the Nakajima quiver variety construction of $M=\Hilb^N(\C^2)$ being associated with $\frgl_N$ rather than $\frsl_N$. This is mirrored by the corresponding decomposition $\C^{2N}/S_N\cong \calM_{S_N}\times T^* \C$ of the associated varieties.

\medskip

It is an immediate consequence of our sheaf construction that by considering the local sections over a Zariski open subset $U\subset M$ we obtain the vertex operator superalgebra embedding
\[
\sfV_{S_N}=[\tldcalD^\ch_{M, \hbar}(M)]^{\C^\times} \hookrightarrow [\tldcalD^\ch_{M, \hbar}(U)]^{\C^\times}.
\]
In particular, by choosing an appropriate open set $U$ we obtain the free-field realisation in terms of $\rk(S_N)=\rk(\frsl_N)=N-1$ many tensor copies of the $\beta\gamma b c$-system vertex superalgebra that is conjectured in \cite{BMR19} (generalising \cite{Adamovic16}):
\begin{theorem*}
For $N\ge2$, there is an embedding $\sfV_{S_N}\hookrightarrow(\beta\gamma b c)^{\otimes(N-1)}\otimes\beta\gamma\otimes\SF\hookrightarrow(\beta\gamma b c)^{\otimes N}$, which restricts to the free-field realisation
\[
\W_{S_N }\hookrightarrow(\beta\gamma b c)^{\otimes(N-1)}.
\]
\end{theorem*}
Note that the conformal structure of $\sfV_{S_N}$ fixes conformal structures of the $\beta\gamma b c$-systems of central charges $c=-3(2p_i-1)$ with $p_i=N,N-1,\dots,1$ and with $p_i=N,N-1,\dots,2$ in the case of $\W_{S_N}$. We recall that in the latter case the $p_i$ are the $\rk(S_N)=N-1$ many degrees of the fundamental invariants of $S_N$. The last $\beta\gamma b c$-system of central charge $c=-3$ is the obvious conformal extension of $\beta\gamma\otimes\SF$, which splits from the chosen local sections of $\tldcalD^\ch_{M,\hbar}$ and, in fact, implies the splitting result stated in the main theorem for the global sections $\sfV_{S_N}$.

As a corollary, we also obtain a free-field realisation of (some quotient of) the
small $\calN=4$ superconformal algebra $\on{Vir}_{\calN=4}^{c_{S_N}}$ of central charge $c_{S_N}=-3(N^2-1)$ in terms of $(\beta\gamma b c)^{\otimes(N-1)}$, generalising results for $N=2$ and $3$ in \cite{Adamovic16} (see also \cite{GMS05}) and \cite{BMR19}, respectively.

\medskip

Besides the Higgs branch, another important datum associated with a four-dimensional $\calN=2$ superconformal field theory $\calT$ is the Schur index $\calI_{\calT}(q)$. In the context of the 4D/2D duality, it is conjectured that this Schur index $\calI_{\calT}$, which should coincide with the supercharacter of the corresponding vertex operator superalgebra $\bbV(\calT)$, is a quasimodular form \cite{BeeRas}. This is true for $\calT=\on{SYM}_{\frsl_N}$ with $\bbV(\calT)\cong\W_{S_N}$:
\begin{theorem*}
For $N\ge2$, the supercharacter of $\W_{S_N}$ is
\[
\on{sch}_{\W_{S_N}}(q)=\frac{1}{\eta(q)^3}{\sum_{n=0}^\infty}(-1)^n\biggl(\binom{N+n}{N}+\binom{N+n-1}{N}\biggr)q^{(N+2n)^2/8}
\]
and coincides with the Schur index of the four-dimensional $\calN=4$ supersymmetric Yang-Mills theory $\on{SYM}_{\frsl_N}$ with gauge group $\SL_N$. This supercharacter is a sum of holomorphic quasimodular forms of weights $N-1,N-3,\dots\in\Z_{\ge0}$ for the full modular group $\SL_2(\Z)$ if $N$ is odd and for $\Gamma^0(2)$ with some character if $N$ is even.
\end{theorem*}
Here, $\eta(q)$ denotes the Dedekind eta function. The proof relies on the exact integration formulae in \cite{PP22} and results for the Schur indices for $\SL_N$ in \cite{BDF15}. For example, for $N=3$, the supercharacter equals
\[
\on{sch}_{\W_{S_3}}(q)=(1-E_2(q))/24,
\]
where $E_2(q)$ denotes the quasimodular Eisenstein series of weight~$2$. We also identify the characters $\on{sch}_{\W_{S_N}}(q)$ with certain multiple $q$-zeta values studied in \cite{Mac21,AR13}, which can be used to derive formulae for their expansion in terms of Eisenstein series for arbitrary $N$.

\medskip

Finally, we note that our results in this article give the first non-trivial examples of chiral quantisation of Nakajima quiver varieties. We plan to extend our results to other Nakajima quiver varieties in forthcoming work.


\subsection*{Outline}

In \autoref{sec:Hilb} we recall the construction of the $N$-point Hilbert scheme $M=\Hilb^N(\C^2)$ of the plane as a Nakajima quiver variety, i.e.\ by Hamiltonian reduction. Then, by adding fermionic coordinate functions, we construct a sheaf $\tldcalO_{M}$ of commutative superalgebras on $M$.

In \autoref{sec:sheaves-SVA} we review the notions of vertex superalgebras, vertex Poisson superalgebras and $\hbar$-adic vertex superalgebras. Then we introduce $\hbar$-adic versions $\calD^\ch(T^* \C^n)_{\hbar}$ of the $\beta\gamma$-system, $\Cl_{\hbar}(T^* \C^n)$ of the $bc$-system, $V^k(\frg)_{\hbar}$ of the affine vertex algebras and $\SF_{\hbar}$ of the symplectic fermion vertex superalgebra. We also construct the microlocalisation $\calD^\ch_{T^* \C^n, \hbar}$ of $\calD^\ch(T^*\C^n)_{\hbar}$, a certain sheaf of $\hbar$-adic vertex algebras on the affine space $T^* \C^n$ with global sections $\calD^\ch(T^* \C^n)_{\hbar}$.

In \autoref{sec:semi-infinite-red} we construct the central object of this text, a sheaf $\tldcalD^\ch_{M, \hbar}$ of $\hbar$-adic vertex superalgebras on the Hilbert scheme $M$. The construction is based on a vertex superalgebra analogue of quantum Hamiltonian reduction, which we call semi-infinite BRST cohomology. We then prove a vanishing (or no-ghost) theorem for this BRST cohomology.

In \autoref{sec:F-action} we study the $\hbar$-adic vertex superalgebra of global sections $\tldcalD^\ch_{M,\hbar}(M)$, which we then reduce to a vertex superalgebra $\sfV_{S_N}\coloneqq [\tldcalD^\ch_{M,\hbar}(M)]^{\C^\times}$ of invariants under a certain torus action. $\sfV_{S_N}$ is endowed with a natural conformal structure that makes it into a vertex operator superalgebra of CFT-type of central charge $c=-3N^2$. We show that the associated variety of $\sfV_{S_N}$ is the symplectic quotient variety $\C^{2N}/S_N$ so that $\sfV_{S_N}$ is quasi-lisse and that $\sfV_{S_N}$ contains a quotient of the small $\calN=4$ superconformal algebra $\on{Vir}_{\calN=4}^{c_{S_N}}$ of central charge $c_{S_N}=-3(N^2-1)$.

In \autoref{sec:Wakimoto}, by considering the local sections of $\tldcalD^\ch_{M,\hbar}$ over an appropriately chosen open subset $U\subset M$, we obtain a free-field realisation of $\sfV_{S_N}$. This also implies the factorisation $\sfV_{S_N}=\W_{S_N}\otimes \calD^\ch(T^*\C^1)\otimes\SF$. Then, $\W_{S_N}$ is a vertex operator superalgebra of CFT-type of central charge $c_{S_N}=-3(N^2-1)$ and a conformal extension of some quotient of $\on{Vir}_{\calN=4}^{c_{S_N}}$. Moreover, $\W_{S_N}$ has the associated variety $\calM_{S_N}$, is quasi-lisse and has a free-field realisation in terms of a $\beta\gamma bc$-system of rank $N-1=\rk(S_N)$. $\W_{S_N}$ is the vertex superalgebra for the reflection group $S_N$ conjectured by Bonetti, Meneghelli and Rastelli \cite{BMR19}.

In \autoref{sec:chars} we determine the supercharacter of the vertex operator superalgebra $\W_{S_N}$ and show that it is a quasimodular form of mixed weight. We also identify it with a certain multiple $q$-zeta value.

Finally, in \autoref{sec:n=2-case} we consider the special case of $N=2$. Then, the sheaf $\tldcalD^\ch_{M,\hbar}$ is essentially the $\hbar$-adic version of the twisted chiral de Rham algebra on $\bbP^1$ with parameter $\alpha=1/2$ introduced in \cite{GMS05}. Moreover, $\W_{S_2}$ is equal to the simple quotient of $\on{Vir}_{\calN=4}^{-9}$ and the free-field realisation in terms of the $\beta\gamma bc$-system coincides with the one given in \cite{Adamovic16}.


\subsection*{Acknowledgements}

We thank Dražen Adamović, Philip Argyres, Henrik Bachmann, Christopher Beem, Anirudh Deb, Martin Möller, Hiraku Nakajima, Leonardo Rastelli and Brandon Rayhaun for valuable discussions.

Tomoyuki Arakawa was supported by Grant-in-Aid KAKENHI JP21H04993 by the \emph{Japan Society for the Promotion of Science}.

Toshiro Kuwabara was supported by Grant-in-Aid KAKENHI JP21K03174 by the \emph{Japan Society for the Promotion of Science}.

Sven Möller acknowledges support through the Emmy Noether Programme by the \emph{Deutsche Forschungsgemeinschaft} (project number 460925688). Sven Möller was also supported by a Postdoctoral Fellowship for Research in Japan and Grant-in-Aid KAKENHI JP20F40018 by the \emph{Japan Society for the Promotion of Science}.


\subsection*{Notation}

Let $G$ be a group and $V$ a $G$-module. We denote the subspace of all $G$-invariant elements of $V$ by $V^G$. For a character $\theta\colon G \longrightarrow\C^\times$, we denote the subspace of all $G$-semi-invariants of weight $\theta$ by $V^{G,\theta}=\{v\in V\,|\,g\cdot v=\theta(g)x\}$.

All algebras and vector spaces will be over the base field $\C$. By a commutative (super)algebra we mean a commutative, associative and unital (super)algebra. For a commutative algebra $R$ and a vector space $V$, denote by $R\langle V\rangle=\bigoplus_{n=0}^{\infty}R\otimes_{\C} V^{\otimes n}$ the tensor algebra of $V$ over $R$. Let $S_{R}(V)=R\langle V\rangle/(a\otimes b-b\otimes a\,|\,a,b\in V)$ and $\Lambda_{R}(V)=R\langle V\rangle/(a\otimes b+b \otimes a\,|\,a,b\in V)$ be the symmetric and the exterior algebras of $V$ over $R$, respectively. Given a basis $\{a_1,\dots,a_r\}$ of $V$, we also write $R\langle a_1,\dots,a_r\rangle=R\langle V\rangle$, $R[a_1,\dots,a_r]=S_{R}(V)$ and $\Lambda_{R}(a_1,\dots,a_r)=\Lambda_{R}(V)$.

For a commutative algebra $A$, let $\Spec A$ be the affine scheme associated with $A$. For a commutative, graded algebra $A=\bigoplus_{n\in \Z_{\ge0}}A_n$, let $\Proj A$ be the projective scheme over $\Spec A_0$ associated with $A$. Throughout the paper, we only consider integral, separated and reduced schemes over $\C$ and call them algebraic varieties.

Let $M$ be an algebraic variety. For a sheaf $\calF$ on $M$ and an open subset $U\subset M$, we denote the set of sections of $\calF$ over $U$ by $\calF(U)$. We denote the structure sheaf of $M$ by $\calO_M$ and the coordinate ring of $M$ by $\C[M]=\calO_M(M)$.


\section{Hilbert Scheme of Points in the Plane}
\label{sec:Hilb}

In this section, we recall Nakajima's construction of the Hilbert scheme $M$ of points in the affine plane and construct a certain sheaf $\tldcalO_{M}$ of commutative superalgebras on it.


\subsection{Quiver Variety Construction}
\label{sec:quiver-const}

We review the construction of the Hilbert scheme $M\coloneqq\Hilb^N(\C^2)$ of $N$ points in the affine complex plane $\C^2$ as a Nakajima quiver variety \cite{Nak94,Nak98} for the Jordan quiver. We refer the reader to \cite{Nakajima} for details of the construction.

\medskip

Throughout the text, let $N\in\Z_{>0}$. Set $V=\gEnd(\C^N)\oplus\C^N$ and consider the corresponding symplectic vector space given by the cotangent bundle
\[
T^*V=V\oplus V^*=\gEnd(\C^N)^{\oplus2}\oplus(\C^N)^{\oplus2}.
\]
We write the standard coordinate functions of $\gEnd(\C^N)\subset V$ by $x_{ij}$ and the ones of $\gEnd(\C^N)\subset V^*$ by $y_{ij}$ for $i,j=1,\dots,N$. We further write the coordinate functions of the vector space $\C^N\subset V$ by $\gamma_i$ and the ones of $\C^N\subset V^*$ dual to $\gamma_i$ by $\beta_i$ for $i=1,\dots,N$. Then, the coordinate ring $\C[T^*V]$ of the symplectic vector space $T^*V$ is a Poisson algebra with Poisson bracket $\{y_{ij}, x_{kl}\}=\delta_{il}\delta_{jk}$, $\{\beta_i, \gamma_j\}=\delta_{ij}$ and zero otherwise.

Let $G\coloneqq\GL_N(\C)$ be the general linear group of rank $N$. We consider an action of $G$ on $\C[T^*V]$ given by
\[
g*X=gXg^{-1},\quad g*Y=gYg^{-1},\quad g*\gamma=g\gamma\quad\text{and}\quad g*\beta=\beta g^{-1},
\]
where $g=(g_{ij})_{i,j}$ is an element of $G$ and $X=(x_{ij})_{i,j}$, $Y=(y_{ij})_{i,j}$, $\gamma={}^t(\gamma_i)_{i}$ and $\beta=(\beta_i)_{i}$. The action induces a $G$-action on the symplectic vector space $T^*V$. This action is Hamiltonian, and we let $\mu\colon T^*V\longrightarrow \frg^*$ be a moment map associated with the action, where $\frg=\Lie(G)=\frgl_N=\gEnd(\C^N)$ is the Lie algebra of $G$. We define the linear map
\[
\mu^*\colon\frg\longrightarrow\C[T^*V],\qquad E_{ij}\mapsto\sum_{p=1}^{N} x_{ip}y_{pj}-\sum_{p=1}^{N}x_{pj}y_{ip}+\gamma_i\beta_j,
\]
where $E_{ij}\in\frg$ is the matrix whose $(i,j)$-th entry is $1$ and all other entries vanish. The map $\mu^*$ satisfies
\begin{align*}
\frac{d}{dt}(\exp(tA)*f)|_{t=0}&=\{\mu^*(A),f\},\\
\mu^*([A,B])&=\{\mu^*(A),\mu^*(B)\}
\end{align*}
for $A,B\in\frg$, $f\in\C[T^*V]$. The moment map $\mu\colon T^*V\longrightarrow\frg^*$ is the morphism of affine varieties associated with the comoment map $\mu^*$.

Let $\frX\subset T^*V$ be the subset defined by the following stability condition: a point $p\in T^*V$ belongs to $\frX$ if and only if the vector $\gamma(p)\in\C^N$ is a cyclic vector of $\C^N$ with respect to the action of the matrices $X(p),Y(p)\in\gEnd(\C^N)$. It is known that $\frX$ is a Zariski open subset, and $G$ acts freely on $\frX$.

Now, we consider the Hamiltonian reduction of $T^*V$ with respect to the $G$-action. The $G$-action is closed on the inverse image $\mu^{-1}(0)$ under the moment map~$\mu$. Set
\[
M_0=\mu^{-1}(0)\qquot G=\Spec\C[\mu^{-1}(0)]^G,
\]
the affine variety associated with the invariant subalgebra $\C[\mu^{-1}(0)]^G$, and
\[
M=(\mu^{-1}(0)\cap\frX)/G,
\]
the quotient manifold of closed $G$-orbits in $\mu^{-1}(0)\cap\frX$. We denote the projection $\mu^{-1}(0)\cap\frX\longrightarrow M$ by $\rho$. The inclusion map $\mu^{-1}(0)\cap\frX\longrightarrow\mu^{-1}(0)$ induces a map $\pi\colon M\longrightarrow M_0$. Note that the Poisson bracket on $\C[T^*V]$ induces a Poisson bracket on the algebra $\C[\mu^{-1}(0)]^G$ and on the structure sheaf $\calO_{M}$ of $M$. Moreover, the manifold $M$ is endowed with the structure of a holomorphic symplectic manifold by the Poisson bracket.

The symmetric group $S_N$ acts naturally on $\C^{2N}=(\C^2)^N$.
\begin{proposition}[\cite{Nakajima}]
$M_0\simeq\C^{2N}/S_N$ and $M\simeq\Hilb^N(\C^2)$, and the map $\pi\colon M\longrightarrow M_0$ gives a resolution of singularities. Moreover, the corresponding homomorphism $\pi^*\colon \C[M_0] = \C[\C^{2N}]^{S_N}\longrightarrow\calO_{M}(M)$ is a homomorphism of Poisson algebras, i.e.\ $\pi$ is a symplectic resolution.
\end{proposition}

The structure sheaf $\calO_M$ of $M$ can be realised also by Hamiltonian reduction of the structure sheaf $\calO_{\frX}$ of $\frX$ as follows: for an open subset $U$ of $M$, take an open subset $\tldU$ of $\frX$ such that $\rho(\tldU \cap \mu^{-1}(0))=U$. Then the sheaf associated with the presheaf given by
\begin{equation}\label{eq:Hamilton-str-sh}
U\mapsto\Bigl(\calO_{\frX}(\tldU)\bigm/\sum_{i,j=1}^{N}\calO_{\frX}(\tldU) \,\mu^*(E_{ij})\Bigr)^\frg
\end{equation}
coincides with the structure sheaf $\calO_M$ of $M$.


\subsection{A Sheaf of Poisson Superalgebras on the Hilbert Scheme}
\label{sec:superHilb}

By adding fermionic coordinate functions to $\calO_{\frX}$, we construct a sheaf $\tldcalO_{M}$ of commutative superalgebras on the Hilbert scheme $M$ by Hamiltonian reduction.

\medskip

Consider $\tldcalO_{\frX}\coloneqq\calO_{\frX}\otimes_{\C}\Lambda_{\C}(\psi_i,\phi_i\,|\,i=1,\dots,N)$, a sheaf of superalgebras on~$\frX$, and define a Poisson bracket on $\Lambda_{\C}(\psi_i,\phi_i\,|\,i=1,\dots,N)$ by setting $\{\psi_i,\phi_j\}=\delta_{ij}$ for $i,j=1,\dots,N$ and zero otherwise. Then $\tldcalO_{\frX}$ is a sheaf of Poisson superalgebras.

We define a homomorphism of Lie superalgebras
\begin{equation}\label{eq:tldmu}
\tldmu\colon\frg\longrightarrow\tldcalO_{\frX}(\frX),\qquad E_{ij}\mapsto\sum_{p=1}^{N}x_{ip}y_{pj}-\sum_{p=1}^{N}x_{pj}y_{ip}+\gamma_i\beta_j+\psi_i\phi_j.
\end{equation}
The comoment map $\tldmu$ induces an action of $\frg$ on $\tldcalO_{\frX}$ by $A\mapsto\{\tldmu(A),\blkbar\}$, and the corresponding $G$-action. As $\psi_i\phi_j$ is nilpotent for any $i,j=1,\dots,N$, note that the quotient sheaf $\tldcalO_{\frX}\big/\sum_{i,j}\tldcalO_{\frX}\,\tldmu(E_{ij})$ is supported on the closed subset $\mu^{-1}(0)\subset\frX$. Define
\[
\tldcalO_{M}=\Bigl(\rho_{*}\bigl(\tldcalO_{\frX}\bigm/\sum_{i,j=1}^{N} \tldcalO_{\frX}\,\tldmu(E_{ij})\bigr)\Bigr)^\frg,
\]
a sheaf of superalgebras on $M$. The Poisson bracket on $\tldcalO_{\frX}$ induces a Poisson bracket on $\tldcalO_{M}$.


\subsection{Local Coordinates over a Certain Open Subset}
\label{sec:big-cell}

In the following, we describe the local coordinates of the sheaf $\tldcalO_{M}$ over a certain affine open subset $U_{(N)}\subset M$ of the Hilbert scheme (cf.\ \cite{Hai98} and \autoref{sec:local-trivial}). This will be useful in \autoref{sec:Wakimoto}, when we describe a certain free-field realisation.

\medskip

Consider the Zariski open subset
\[
\tldU_{(N)}\coloneqq\{p\in T^*V\,|\,\gamma(p),\,X(p)\gamma(p),\,\dots,\,X(p)^{N-1} \gamma(p)\text{ span }\C^N\},
\]
of $\frX$, and let \[U_{(N)}\coloneqq\rho(\tldU_{(N)}\cap\mu^{-1}(0))\] be the corresponding Zariski open subset of $M$. We describe the local coordinates over the open subset $U_{(N)}$ explicitly.

Set $B_{(N)}=(\gamma,X \gamma,\dots,X^{N-1}\gamma)$, an $N$-by-$N$ matrix with entries in $\C[T^* V]$. Note that $B_{(N)}$ is an invertible matrix in $\calO_{\frX}(\tldU_{(N)})$. We then define the functions $[X^N:X^{i}]_{(N)}$ and $[Y:X^{i}]_{(N)}\in\tldcalO_{\frX}(\tldU_{(N)})$ for $i=0,\dots,N-1$ by
the equalities
\begin{align*}
{}^t([X^N:X^{i-1}]_{(N)})_{i=1,\dots,N}&=B_{(N)}^{-1}X^N\gamma,\\
{}^t([Y:X^{i-1}]_{(N)})_{i=1,\dots,N}&=B_{(N)}^{-1}Y\gamma
\end{align*}
and the fermionic functions $\{\psi:X^{i}\gamma\}_{(N)}\in\tldcalO_{\frX}(\tldU_{(N)})$ for $i=0,\dots,N-1$ by
\[
{}^t(\{\psi:X^{i-1}\gamma\}_{(N)})_{i=1,\dots,N}=B_{(N)}^{-1}\psi,
\]
where we write $\psi=(\psi_i)_{i=1,\dots,N}$. We also consider $\phi X^i\gamma$ for $i=0,\dots,N-1$ with $\phi={}^t(\phi_i)_{i=1,\dots,N}$. Then, all these functions are $G$-invariant and we shall see that
\begin{align}\label{eq:local-coord-X}
\tldcalO_{\frX}(\tldU_{(N)})&=\C[[X^N:X^i]_{(N)},[Y:X^i]_{(N)}\,|\,i=0, \dots,N-1]\\
&\quad\otimes\Lambda_{\C}(\{\psi:X^i\gamma\}_{(N)},\phi X^{i}\gamma\,|\, i=0,\dots,N-1)\nonumber\\
&\quad\otimes\C[(B_{(N)}^{\pm 1})_{ij}\,|\,i,j=1,\dots,N]\otimes\C[\tldmu(E_{ij})\,|\,i,j=1,\dots,N].\nonumber
\end{align}
By \eqref{eq:Hamilton-str-sh}, this implies the following description of the local coordinates over $U_{(N)}$ so that $U_{(N)}\simeq\C^{2N}$ is affine:
\begin{align}\label{eq:local-coord-Hilb}
\tldcalO_{M}(U_{(N)})&=\C[[X^N:X^i]_{(N)},[Y:X^i]_{(N)}\,|\,i=0,\dots,N-1]\\
&\quad\otimes\Lambda_{\C}(\{\psi:X^i\gamma\}_{(N)},\phi X^{i}\gamma\,|\,i=0,\dots,N-1).\nonumber
\end{align}

We now show \eqref{eq:local-coord-X}. Since $\calO_{\frX}(\tldU_{(N)})=\C[T^* V][(B_{(N)}^{-1})_{ij}\,|\,i,j=1,\dots,N]$, the right-hand side is contained in the left-hand side. To prove the opposite inclusion, it suffices to show that all the generators belong to the right-hand side. By the above definition, $\psi=B_{(N)}{}^t(\{\psi:X^{i-1}\}_{(N)})_{i=1,\dots,N}$ and $\phi=(\phi X^{i-1}\gamma)_{i=1,\dots,N}B_{(N)}^{-1}$ so that $\psi_i$ and $\phi_i$ belong to the right-hand side for any $i=1,\dots,N$. From the identity $X B_{(N)}=(X\gamma,X^2\gamma,\dots,X^N\gamma)$ it follows that $X=(X\gamma,X^2\gamma,\dots,X^N\gamma)B_{(N)}^{-1}$, and this implies that $x_{ij}$ belongs to the right-hand side for any $i,j=1,\dots,N$ since $X^N\gamma=B_{(N)}{}^t ([X^N:X^{i-1}]_{(N)})_{i=1,\dots,N}$.

In order to apply the same reasoning to $Y$, we need two lemmata. To this end, set $Z=(\tldmu(E_{ij}))_{i,j=1,\dots,N}=XY-YX+\gamma\beta+\psi\phi$.
\begin{lemma}\label{lemma:8}
For $i\in\Z_{\ge0}$, $\beta X^i\gamma$ lies in the right-hand side of \eqref{eq:local-coord-X}.
\end{lemma}
\begin{proof}
Note that
\begin{align*}
\beta X^i\gamma
&=\Tr(X^i\gamma\beta)=-\Tr(X^{i+1}Y)+\Tr(X^iYX)-\Tr(X^i\psi\phi)+\Tr(X^iZ)\\
&=-\Tr(X^i\psi\phi)+\Tr(X^i Z)=\phi X^i\psi+\Tr(X^iZ).
\end{align*}
Since $\psi_i$, $\phi_i$, $x_{ij}$ and $Z_{ij}=\tldmu(E_{ij})$ belong to the right-hand side for all $i,j$, the same holds for $\beta X^i\gamma$ for $i\in\Z_{\ge0}$.
\end{proof}

\begin{lemma}\label{lemma:9}
For $m,i=0,\dots,N-1$, let $[YX^m:X^i]_{(N)}\in\tldcalO_{\frX}(\tldU_{(N)})$ be defined by $([YX^m:X^{i-1}]_{(N)})_{i=1,\dots,N}=B_{(N)}^{-1}YX^m\gamma$. Then, $[YX^m:X^i]_{(N)}$ belongs to the right-hand side of \eqref{eq:local-coord-X}.
\end{lemma}
\begin{proof}
We proceed by induction over $m$. If $m=0$, $[YX^m:X^{i}]_{(N)}=[Y:X^i]_{(N)}$ belongs to the right-hand side of \eqref{eq:local-coord-X} for $i=0,\dots,N-1$. For the induction step, assume that $[YX^m:X^i]_{(N)}$ lies in the right-hand side for all $i=0,\dots,N-1$. Then, $YX^{m+1}\gamma=XYX^{m}\gamma+\gamma\beta X^{m}\gamma+\psi\phi X^{m}\gamma-ZX^m\gamma$. By the induction hypothesis, $XYX^m\gamma$ lies in the right-hand side. By \autoref{lemma:8}, $\gamma(\beta X^m\gamma)$ belongs to the right-hand side, and clearly so does $\psi(\phi X^m\gamma)$. Therefore, $[YX^{m+1}:X^i]_{(N)}$ also lies in the right-hand side of \eqref{eq:local-coord-X} for all $i=0,\dots,N-1$.
\end{proof}

By \autoref{lemma:9}, all entries of the matrix $YB_{(N)}=(Y\gamma,YX\gamma,\dots,YX^{N-1}\gamma)$ lie in the right-hand side of \eqref{eq:local-coord-X}. Thus, so does $y_{ij}$ for all $i,j=1,\dots,N$. Similarly, all entries of the row vector $\beta B_{(N)}=(\beta\gamma,\beta X\gamma,\dots,\beta X^{N-1}\gamma)$ lie in the right-hand side, and hence so do all $\beta_i$. This proves \eqref{eq:local-coord-X}.


\subsection{Local Trivialisation of the Hamiltonian Reduction}
\label{sec:local-trivial}

The Hilbert scheme $M=\Hilb^N(\C^2)$ has an affine open covering $M=\bigcup_{\lambda\vdash N}U_{\lambda}$ indexed by the partitions $\lambda$ of $N$ \cite{Hai98}, which includes the set $U_{(N)}$ from the preceding section. We now describe the local trivialisations of the sheaf $\tldcalO_{M}$ over the open sets $U_\lambda$.

\medskip

Let $\lambda=(\lambda_0\ge\lambda_1\ge\dots)$ be a partition of $N$, denoted by $\lambda\vdash N$. For a pair of non-negative integers $(i,j)\in\Z_{\ge 0}^2$, we write $(i,j)\in\lambda$ if $i<\lambda_j$.

Generalising the definition in \autoref{sec:big-cell}, we let $\tldU_{\lambda}\subset T^*V$ be the subset of those $p\in T^*V$ satisfying that $\{X(p)^iY(p)^j\gamma(p)\,|\,(i,j)\in\lambda\}$ spans the vector space $\C^N$. With $B_{\lambda}\coloneqq(X^iY^j\gamma)_{(i,j)\in\lambda}$, $\tldU_{\lambda}$ may be written as $\tldU_{\lambda}=\{p\in T^*V\,|\,\det B_{\lambda}(p)\ne0\}$. Thus, $\tldU_{\lambda}$ is a Zariski open subset of $\frX$, and moreover there is an affine open covering $\frX=\bigcup_{\lambda\vdash N}\tldU_{\lambda}$. It induces an affine open covering
\[
M=\bigcup_{\lambda\vdash N}U_{\lambda},\quad\text{with}\quad U_{\lambda}=(\tldU_{\lambda}\cap\mu^{-1}(0))/G
\]
the affine open subset of $M$ associated with $\tldU_{\lambda}$.

For $P=P(X,Y)\in\C\langle X,Y\rangle$ and $(i,j)\in\lambda$, define a bosonic section $[P:X^iY^j]_{\lambda}$ and a fermionic section $\{\psi:X^iY^j\gamma\}_{\lambda}$ in $\tldcalO_{\frX}(\tldU_{\lambda})$ by
\begin{align*}
{}^t([P:X^iY^j]_{\lambda})_{(i,j)\in\lambda}&=B_{\lambda}^{-1}P\gamma,\\
{}^t(\{\psi:X^iY^j\gamma\}_{\lambda})_{(i,j)\in\lambda}&=B_{\lambda}^{-1}\psi,
\end{align*}
respectively. Then, $[P:X^iY^j]_{\lambda}$ is $G$-invariant for any $P\in\C\langle X,Y\rangle$, $(i,j)\in\lambda$. Also, $\{\psi:X^iY^j\gamma\}_{\lambda}\in\tldcalO_{\frX}(\frX)$ is a $G$-invariant fermionic section for $(i,j)\in\lambda$. Define $A_{\lambda}$ to be the subalgebra of $\calO_{\frX}(\tldU_{\lambda})=\C[T^*V][(B_{\lambda}^{-1})_{ij}\,|\,i,j=1,\dots,N]\subset\tldcalO_{\frX}(\tldU_{\lambda})$ generated by $[P:X^iY^j]_{\lambda}$ for all $P\in \C\langle X,Y\rangle$ and $(i,j)\in\lambda$.

\begin{proposition}\label{prop:local-triv}
For $\lambda\vdash N$, there is the following trivialisation of the Hamiltonian reduction
\begin{align*}
\tldcalO_{\frX}(\tldU_{\lambda})&=A_{\lambda}\otimes\Lambda(\{\psi:X^iY^j \gamma\}_{\lambda},\phi X^iY^j\gamma\,|\,(i,j)\in\lambda)\\
&\quad\otimes\C[(B_{\lambda}^{\pm 1})_{ij}\,|\,i,j=1,\dots,N]\otimes\C[\tldmu(E_{ij})\,|\,i,j=1,\dots,N]\eqqcolon\tldA_{\lambda}.
\end{align*}
This localisation yields the isomorphism
\[
\tldcalO_M(U_{\lambda})\simeq A_{\lambda}\otimes\Lambda(\{\psi:X^iY^j\gamma\}_{\lambda},\phi X^iY^j\gamma\,|\,(i,j)\in\lambda).
\]
\end{proposition}

\begin{remark}
It is easy to see that $A_{\lambda}=\calO_M(U_{\lambda})$. For further details of the combinatorial nature of the affine open covering $M=\bigcup_{\lambda\vdash N}U_{\lambda}$, see \cite{Hai98} or Chapter~18 of \cite{MS05}.
\end{remark}

\begin{proof}[Proof of \autoref{prop:local-triv}]
We have to show that all generators of $\tldcalO_{\frX}(\tldU_{\lambda})$ belong to $\tldA_{\lambda}$. Evidently, all entries of the vectors $X^aY^b\gamma=B_{\lambda}{}^t([X^aY^b:X^iY^j]_{\lambda})_{(i,j)\in\lambda}$, $\psi=B_{\lambda}{}^t(\{\psi:X^iY^j\gamma\}_{\lambda})_{(i,j)\in\lambda}$ and $\phi=(\phi X^iY^j\gamma)_{(i,j)\in\lambda}B_{\lambda}^{-1}$ lie in $\tldA_{\lambda}$. Consider the matrix $X B_{\lambda}=(X^{k+1}Y^l\gamma)_{(k,l)\in\lambda}$. All its entries are of the form $[X^{k+1}Y^l:X^iY^j]_{\lambda}$ and lie in $A_{\lambda}$. Thus, $X=(X^{k+1}Y^l\gamma)_{(k,l)\in\lambda}B^{-1}_{\lambda}$ implies that $x_{ij}$ lies in $\tldA_{\lambda}$ for $i,j=1,\dots,N$. Similarly, $Y=(YX^kY^l\gamma)_{(k,l)\in\lambda}B_{\lambda}^{-1}$ implies that $y_{ij}$ lies in $\tldA_{\lambda}$ for $i,j=1,\dots,N$. Therefore, we obtain the identity of \autoref{prop:local-triv}.
\end{proof}


\section{Sheaves of \texorpdfstring{$\hbar$}{ℏ}-Adic Vertex Superalgebras}
\label{sec:sheaves-SVA}

In this section, we review the notion of ($\hbar$-adic) vertex superalgebras. We also recall jet schemes and arc spaces. Then we introduce $\hbar$-adic versions of some well-known examples of vertex superalgebras.


\subsection{Vertex Superalgebras and \texorpdfstring{$\hbar$}{ℏ}-Adic Vertex Superalgebras}
\label{sec:vertex-algebras}

In the following, we define vertex superalgebras, vertex Poisson superalgebras and $\hbar$-adic vertex superalgebras.

\medskip

A vertex superalgebra $V$ (see, e.g., \cite{Bor86,Li96b,Kac}) is defined as a tuple $(V, \bfone, \partial, Y(\blkbar, z))$, where $V$ is a $\Z/2\Z$-graded vector space, $\bfone \in V$ is a distinguished even element called the vacuum vector, $\partial$ is an even, linear map $\partial\colon V \longrightarrow V$ called the translation operator, and $Y(\blkbar, z)$ is the state-field correspondence, an even, linear map,
\begin{align*}
Y(\blkbar, z)\colon V &\longrightarrow \gEnd_{\C}(V)[[z, z^{-1}]],\\
a&\longmapsto Y(a, z) = \sum_{n \in \Z} a_{(-n-1)} z^{n},
\end{align*}
subject to the following axioms:
\begin{enumerate}[label=(\alph*)]
\item $Y(a, z)$ is a field, i.e.\ $a_{(n)} b = 0$ for any $a,b \in V$ if $n \gg 0$,
\item $Y(\bfone, z) = \Id_V$,
\item $Y(a, z) \bfone \in V[[z]]$ and $Y(a, z) \bfone |_{z=0} = a$ for any $a \in V$,
\item $[\partial, Y(a, z)] = Y(\partial a, z) = \partial_z Y(a, z)$ for any $a \in V$, and $\partial \bfone = 0$,
\item (Borcherds' identity) for any $a,b \in V$ and $l,m,n \in \Z$,
\begin{align}\label{eq:Borcherds}
&\sum_{j=0}^{\infty} \binom{m}{j} (a_{(n+j)} b)_{(l+m-j)}\\
&= \sum_{j=0}^{\infty} (-1)^j \binom{n}{j} \{a_{(m+n-j)} b_{(l+j)} - (-1)^{n+p(a)p(b)} b_{(l+n-j)} a_{(m+j)}\},\nonumber
\end{align}
where $p(a) \in \Z/2\Z$ is the parity of $a$.
\end{enumerate}
The $Y(a,z)$, $a\in V$, are called vertex operators and sometimes written as $a(z)$.

The vertex superalgebra $V$ is said to be commutative if $a_{(n)} = 0$ for any $a\in V$ and $n \in \Z_{\ge 0}$. Equivalently, all vertex operators (super)commute with each other. In this case, the $(-1)$-product endows $V$ with the structure of a commutative superalgebra (with an even derivation).

It is well-known that the commutation relations between the operators
$a_{(m)}$ and $b_{(n)}$ for $a,b \in V$ and $m,n \in \Z$ are encoded into the identity
\begin{equation}\label{eq:OPE}
Y(a, z) Y(b, w) = \sum_{n \in \Z} \frac{Y(a_{(n)} b, w)}{(z-w)^{n+1}} = \sum_{n \ge 0} \frac{Y(a_{(n)} b, w)}{(z-w)^{n+1}} + \NO Y(a, z) Y(b, w) \NO,
\end{equation}
called the operator product expansion (or OPE for short), where $\NO\blkbar\NO$ is the normally ordered product (see, e.g., \cite{FBZ}).
It is often written in the form
\[
Y(a, z) Y(b, w) \sim \sum_{n \ge 0} \frac{Y(a_{(n)} b, w)}{(z-w)^{n+1}},
\]
omitting the term $\NO Y(a, z) Y(b, w) \NO$, which is regular at $z=w$ and does not affect the commutators $[a_{(m)}, b_{(n)}]$.

\medskip

A conformal vertex superalgebra of central charge $c\in\C$ is a vertex superalgebra $V$ together with an even vector $T\in V$ satisfying
\begin{enumerate}[label=(\alph*)]
\item $[T_{(m+1)},T_{(n+1)}]=(m-n)T_{(m+n+1)}+\frac{m^3-m}{12}\delta_{m+n,0}\,c\Id_V$ for any $m,n\in\Z$,
\item $T_{(1)}$ acts semisimply on $V$ with eigenvalues in $\frac{1}{2}\Z$, called weights,
\item $\wt(a_{(n)})=\wt(a)-n-1$ for any $a\in V$, $n\in\Z$,
\item $T_{(0)}=\partial$.
\end{enumerate}
We write $V=\bigoplus_{n\in\frac{1}{2}\Z}V_n$ for the corresponding eigenspace decomposition of $V$. Note that $\bfone\in V_0$, $\partial\in V_1$ and $T\in V_2$. We point out that we do not require ``correct statistics'', i.e.\ that $\bigoplus_{n\in\Z}V_n$ and $\bigoplus_{n\in\frac{1}{2}+\Z}V_n$ be exactly the vectors of even and odd parity, respectively (cf.\ \cite{RSW22,HM23}).

A conformal vertex superalgebra $V$ is called vertex operator superalgebra if additionally $V_n=\{0\}$ for $n\ll 0$ and $\Dim_\C V_n<\infty$ for all $n\in\frac{1}{2}\Z$.

\medskip

A vertex Poisson superalgebra $V$ \cite{FBZ} is a tuple $(V, \bfone, \partial, Y_{-}(\blkbar, z), Y_{+}(\blkbar, z))$, where
\[
Y_{\pm}(\blkbar, z)\colon V \longrightarrow \gEnd_{\C}(V)[[z, z^{-1}]]
\]
are even, linear maps to fields
\[
Y_{-}(a, z) = \sum_{n \in \Z_{\ge 0}} a_{(-n-1)} z^n, \quad Y_{+}(a, z) = \sum_{n \in \Z_{< 0}} a_{(-n-1)} z^n,
\]
$a\in V$, such that $(V, \bfone, \partial, Y_{-}(\blkbar, z))$ is a commutative vertex superalgebra and $(V, \partial, Y_{+}(\blkbar, z))$ is a vertex Lie superalgebra, i.e.\ the operators $a_{(n)}$ for $n \in \Z_{\ge 0}$ satisfy the following axioms:
\begin{enumerate}[label=(\alph*)]
\item $a_{(n)} b = (-1)^{n+p(a)p(b)+1} \sum_{j \ge 0} (-1)^j \partial^j (b_{(n+j)} a) / j!$,
\item $[a_{(m)},b_{(n)}] = \sum_{j \ge 0} \binom{m}{j} (a_{(j)} b)_{(m+n-j)}$,
\item $[\partial, Y_{+}(a, z)] = Y_{+}(\partial a, z) = \partial_z Y_{+}(a, z)$
\end{enumerate}
for any $a,b \in V$ and $m,n \in \Z_{\ge 0}$. For a vertex Poisson superalgebra, in addition, we require that $Y_{-}(\blkbar, z)$ and $Y_{+}(\blkbar, z)$ be compatible in the sense that the $a_{(n)}$ for $n \in \Z_{\ge 0}$ are derivations with respect to the $(-1)$-product, i.e.\ that
\begin{enumerate}[label=(\alph*),resume]
\item $a_{(n)} (b_{(-1)} c) = (a_{(n)} b)_{(-1)} c + (-1)^{p(a)p(b)} b_{(-1)} (a_{(n)} c)$
\end{enumerate}
for all $a,b,c \in V$ and $n \in \Z_{\ge 0}$.

\medskip

Following \cite{Li04}, we introduce the notion of $\hbar$-adic vertex superalgebras. To this end, let $\hbar$ be an indeterminate that commutes with all other operators. An $\hbar$-adic vertex superalgebra $V$ is a tuple $(V, \bfone, \partial, Y(\blkbar, z))$, where $V$ is a flat $\C[[\hbar]]$-module complete in the $\hbar$-adic topology, the even vacuum vector $\bfone \in V$ and the parity-preserving $\C[[\hbar]]$-linear map $\partial\colon V \longrightarrow V$ satisfy the same axioms as for a vertex superalgebra and
\[
Y(\blkbar, z)\colon V \longrightarrow \gEnd_{\C[[\hbar]]}(V)[[z, z^{-1}]]
\]
is a $\C[[\hbar]]$-linear map such that the $(n)$-products are continuous with respect to the $\hbar$-adic topology and $(V / \hbar^k V, \bfone, \partial, Y(\blkbar, z))$ is a vertex superalgebra for each $k \in \Z_{>0}$. Note that an $\hbar$-adic vertex superalgebra is not necessarily a vertex superalgebra over $\C[[\hbar]]$ since $Y(a, z)$ is not a field on $V$. Indeed, for any $k \in \Z_{>0}$, $Y(a, z) = \sum_{n \in \Z} a_{(n)} z^{-n-1}$ satisfies $a_{(n)} b \equiv 0$ modulo $\hbar^k$ if $n \gg 0$ but not necessarily $a_{(n)} b = 0$.

Let $(V, \bfone, \partial, Y(\blkbar, z))$ be an $\hbar$-adic vertex superalgebra. Assume that $V / \hbar V$ is commutative. Then, $Y_{+}(\blkbar, z) \coloneqq \hbar^{-1} Y(\blkbar, z) \bmod \hbar$ satisfies the axioms of a vertex Lie superalgebra. Thus, $(V / \hbar V, \bfone, \partial, Y(\blkbar, z) \bmod \hbar, \hbar^{-1} Y(\blkbar, z) \bmod \hbar)$ is a vertex Poisson superalgebra.

For $\hbar$-adic vertex superalgebras $V$ and $W$, we denote the $\hbar$-adic completion of the tensor product $V \otimes_{\C[[\hbar]]} W$ by $V \hatotimes W$. This tensor product $V \hatotimes W$ is also an $\hbar$-adic vertex superalgebra and it includes $V$ and $W$ as $\hbar$-adic vertex subalgebras.

One can also introduce the notions of conformal $\hbar$-adic vertex superalgebras and $\hbar$-adic vertex operator superalgebras in analogy to the non-$\hbar$-adic versions above. For instance, the main properties of a conformal vector in an $\hbar$-adic vertex superalgebra are stated in \autoref{lemma:Virasoro}.


\subsection{Jet Bundles}
\label{sec:jet-bundles}

We briefly recall jet bundles, jet schemes and arc spaces. References are \cite{Ish07,EinMus,BouNicSeb00}. In the following, a differential algebra is a commutative $\C$-superalgebra $A$ with an even derivation $\partial$. Note that a differential algebra $A$ can be regarded as a commutative vertex superalgebra, where $Y(\blkbar, z)$ is given by $Y(a, z) = \e^{z \partial} a$ for $a \in A$.

For a finitely generated, unital, commutative superalgebra $R$, let $\jet{\infty}{R}$ be the unique differential algebra satisfying the following universal property: there is an algebra homomorphism $j \colon R \longrightarrow \jet{\infty}{R}$ such that for any algebra homomorphism $\varphi \colon R \longrightarrow A$ from $R$ to a differential algebra $A$, there is a unique differential algebra homomorphism $\widetilde{\varphi} \colon \jet{\infty}{R} \longrightarrow A$ satisfying $\widetilde{\varphi} \circ j = \varphi$. The differential algebra $\jet{\infty}{R}$ can also be characterised as the unique unital, commutative $\C$-superalgebra such that $\gHom_{\Alg}(\jet{\infty}{R}, A) \simeq \gHom_{\Alg}(R, A[[t]])$ for any unital, commutative $\C$-superalgebra $A$. For $m \in \Z_{\ge 0}$ and $R$ as above, let $\jet{m}{R}$ be the unique unital, commutative $\C$-superalgebra satisfying $\gHom_{\Alg}(\jet{m}{R}, A) \simeq \gHom_{\Alg}(R, A[t]/(t^{m+1}))$ for any unital, commutative $\C$-superalgebra $A$. There exists a natural homomorphism $\jet{m}{R} \longrightarrow \jet{n}{R}$ for $m \le n$, and $\jet{\infty}{R}$ is the injective limit associated with the direct system $\jet{0}{R} \longrightarrow \jet{1}{R} \longrightarrow \jet{2}{R} \longrightarrow \dots$

We may assume that the commutative superalgebra $R$ is of the form
\[
R = (\C[x_1, \dots, x_N] \otimes \Lambda(\psi_1, \dots, \psi_M)) / (f_1, \dots, f_r)
\]
with bosonic variables $x_1,\dots,x_N$, fermionic variables $\psi_1,\dots,\psi_M$ and $f_1,\dots,f_r \in \C[x_1, \dots, x_N] \otimes \Lambda(\psi_1, \dots, \psi_M)$. Then, $\jet{\infty}{R}$ is the differential algebra given by
\[
\jet{\infty}{R} = (\C[x_{i (-n)} \,|\, \substack{i = 1, \dots, N \\ n = 1, 2, \dots}] \otimes \Lambda(\psi_{i (-n)} \,|\, \substack{i = 1, \dots, M \\ n = 1, 2, \dots})) / ( \partial^n f_i \,|\, \substack{i = 1, \dots, r \\ n = 1, 2, \dots}),
\]
where $x_{i (-n)} = (\partial^{n-1} / (n-1)!) x_i$ and $\psi_{i (-n)} = (\partial^{n-1} / (n-1)!) \psi_i$. For $m \in \Z_{\ge 0}$, the superalgebra $\jet{m}{R}$ is the subalgebra of $\jet{\infty}{R}$ generated by $x_{i (-n)}$ and $\psi_{j (-n)}$ with $i=1,\dots,N$, $j=1,\dots,M$ and $n=0,1,\dots,m$.

Let $X$ be a topological space and $\tldcalO_X$ a sheaf of unital, commutative superalgebras on $X$. We define the sheaf of differential algebras $\JetBundle{X} \coloneqq \jet{\infty}{\tldcalO_{X}}$ associated with $\tldcalO_X$ as the sheaf $U \mapsto \jet{\infty}{(\tldcalO_{X}(U))}$ for open subsets $U \subset X$. Note that there is a natural inclusion $j \colon \tldcalO_X \hookrightarrow \jet{\infty}{\tldcalO_{X}}$. Similarly, for $m \in \Z_{\ge 0}$, let $\jet{m}{\tldcalO_{X}}$ be the sheaf $U \mapsto \jet{m}{(\tldcalO_{X}(U))}$ for open subsets $U \subset X$. We call the sheaf $\JetBundle{X} = \jet{\infty}{\tldcalO_{X}}$ the $\infty$-jet bundle associated with $\tldcalO_{X}$.

Let $X$ be a scheme of finite type with structure sheaf $\calO_X$. The arc space (or $\infty$-jet scheme) over $X$ is the unique scheme $\jet{\infty} X$ satisfying
\[
\gHom_{\Scheme}(\Spec A, \jet{\infty}{X}) \simeq \gHom_{\Scheme}(\Spec A[[t]], X)
\]
for any unital, commutative $\C$-superalgebra $A$. Similarly, for $m \in \Z_{\ge 0}$, the $m$-jet scheme over $X$ is the unique scheme $\jet{m}{X}$ such that
\[
\gHom_{\Scheme}(\Spec A, \jet{m}{X}) \simeq \gHom_{\Scheme}(\Spec A[t] / (t^{m+1}), X)
\]
for any unital, commutative $\C$-superalgebra $A$. By definition, for $m \in \Z_{\ge 0} \cup \{\infty\}$, there is a natural morphism $\pi_m \colon \jet{m}{X} \longrightarrow X$. For $m \in \Z_{\ge 0} \cup \{\infty\}$, the sheaf $\jet{m}{\calO_X}$ associated with $\calO_X$ is related to the structure sheaf $\calO_{\jet{m}{X}}$ of the $m$-jet scheme $\jet{m}{X}$ by $(\pi_m)_{*} \calO_{\jet{m}{X}} = \jet{m}{\calO_{X}}$.

Now assume that the commutative superalgebra $R$ is a Poisson superalgebra with Poisson bracket $\{\blkbar, \blkbar\}$. The Poisson bracket $\{\blkbar, \blkbar\}$ of $\tldcalO_X$ induces a natural vertex Poisson superalgebra structure on the differential algebra $\jet{\infty}{R}$ satisfying $f_{(n)} g = \delta_{n0} \{f, g\}$ for $f,g \in R$ and $n \in \Z_{\ge 0}$ (\cite{Arakawa12}, Proposition~2.3.1).


\subsection{The \texorpdfstring{$\beta\gamma$}{βγ}-System and Its Microlocalisation}
\label{sec:h-adic-betagamma}

In the following, we introduce the $\hbar$-adic analogue $\calD^\ch(T^*\C^{N})_{\hbar}$ of the $\beta\gamma$-system $\calD^\ch(T^*\C^{N})=(\beta\gamma)^{\otimes N}$ of rank~$N$, also called Weyl vertex algebra or bosonic ghost vertex algebra. We then construct the microlocalisation $\calD^\ch_{T^* \C^N, \hbar}$ of $\calD^\ch(T^*\C^{N})_{\hbar}$ as a sheaf of $\hbar$-adic vertex algebras on the affine space $T^* \C^{N}$, as discussed in Section~2.2 of \cite{AKM15}, with global sections $\calD^\ch_{T^* \C^N, \hbar}(T^* \C^N) = \calD^\ch(T^* \C^N)_{\hbar}$. Since we only deal with the affine space $T^* \C^N$, the construction looks simpler than the one in \cite{AKM15}, where we considered the cotangent bundle of a flag manifold.

\medskip

Denote by $x_1,\dots,x_N$ and $y_1,\dots,y_N$ be the standard coordinate functions on $T^*\C^N = \C^{N} \oplus (\C^N)^*$. They are Darboux coordinates with respect to the standard symplectic form. Recall that the $\beta\gamma$-system $\calD^{\ch}(T^* \C^N)=(\beta\gamma)^{\otimes N}$ is the vertex algebra strongly generated by $x_i$ and $y_i$ for $i=1,\dots,N$ being subject to the operator product expansions $x_i(z) y_j(w) \sim -\delta_{ij}/(z-w)$ and $x_i(z) x_j(w) \sim y_i(z) y_j(w) \sim 0$. The $\hbar$-adic $\beta\gamma$-system on $T^* \C^N$ is the $\hbar$-adic vertex algebra $\calD^\ch(T^* \C^{N})_{\hbar}$ isomorphic as a $\C[[\hbar]]$-module to
\[
\calD^\ch(T^* \C^{N})_{\hbar} = \C[[\hbar]][x_{1 (-n)}, \dots, x_{N (-n)}, y_{1 (-n)}, \dots, y_{N (-n)} \,|\, n \in \Z_{>0}]
\]
with operator product expansions
\begin{align*}
x_i(z) y_j(w) &\sim - \delta_{ij} \hbar /(z-w),\\
x_i(z) x_j(w) &\sim 0,\\
y_i(z) y_j(w) &\sim 0
\end{align*}
for $i,j=1,\dots,N$, where we write $x_i(z) = Y(x_{i (-1)}, z)$ and $y_i(z) = Y(y_{i (-1)}, z)$. Evidently, it is an $\hbar$-adic analogue of the $\beta\gamma$-system vertex algebra.

The notion of chiral differential operators (CDO), which is the chiral analogue of sheaves of differential operators on complex manifolds, was introduced in \cite{MSV99,BeiDri04} and developed in \cite{MSV99,MS99,GMS00,GMS04,GMS03}. Following their construction, we can localise the $\beta\gamma$-system $\calD^\ch(T^* \C^N)$ generated by $x_i$ and $y_i$ for $i=1,\dots,N$ as a sheaf of vertex algebras on the complex vector space $\C^N$. Moreover, as discussed in \cite{AKM15}, we can also localise the above $\hbar$-adic $\beta\gamma$-system $\calD^\ch(T^* \C^{N})_{\hbar}$ as a sheaf $\calD^\ch_{T^* \C^N, \hbar}$ of $\hbar$-adic vertex algebras on the cotangent bundle $T^* \C^N$.

We shall give the details of this construction below, which we can summarise as follows: for the $\hbar$-adic $\beta\gamma$-system, the operator product expansions (and hence the $(n)$-products) between the vertex operators are determined by the Wick formula and thus they turn out to be bidifferential operators in the variables $x_{i (-n)}$, $y_{i (-n)}$. Therefore, even for rational functions in $x_{i (-n)}$, $y_{i (-n)}$, the same bidifferential operators give well-defined operator product expansions (or $(n)$-products) between them. Hence, we obtain a sheaf of $\hbar$-adic vertex algebras $\calD^\ch_{T^* \C^N, \hbar}$ on $T^* \C^N$.

As we discussed in the previous section, the jet bundle $\calO_{J_{\infty} T^* \C^N}$ on the symplectic vector space $T^* \C^N$ is equipped with a vertex Poisson algebra structure. The $\hbar$-adic $\beta\gamma$-system $\calD^\ch_{T^*\C^N, \hbar}$ is a quantisation of $\calO_{J_{\infty} T^* \C^N}$. Indeed, the quotient $\calD^\ch_{T^*\C^N, \hbar} / \hbar \calD^\ch_{T^*\C^N, \hbar}$ is isomorphic to $\calO_{J_{\infty} T^* \C^N}$ as vertex Poisson algebras.

\medskip

We now construct the sheaf of $\hbar$-adic vertex algebras $\calD^\ch_{T^* \C^N, \hbar}$ over $T^* \C^N$ that localises $\calD^\ch(T^* \C^N)_{\hbar}$. We discuss the localisation in Zariski topology. The key is that the operator product expansion of the ($\hbar$-adic) $\beta\gamma$-system satisfies the Wick formula, which implies that it can be realised by certain bidifferential operators. This fact allows us to construct the localisation of $\calD^\ch_{T^* \C^N, \hbar}$ similarly to the construction of the ring of microlocal differential operators on a
holomorphic symplectic manifold (cf.\ \cite{KR08, Losev10}).

While the $(-n)$-products on the $\hbar$-adic $\beta\gamma$-system are not associative, the variables $x_{i (-n)}$ and $y_{i (-n)}$ for $i=1,\dots,N$ and $n \in \Z_{>0}$ mutually commute in $\calD^\ch(T^* \C^N)_{\hbar}$. Thus, any element of $\calD^\ch(T^* \C^N)_{\hbar}$ can be viewed as a polynomial in the variables $\{x_{i (-n)}, y_{i (-n)} \,|\, i=1, \dots, N; n \in \Z_{>0}\}$. Namely, we do not need to care about the order of the variables, but we do care about the placement of brackets. We identify the variables $x_{i (-1)}$ and $y_{i (-1)}$ with the coordinate functions $x_i$ and $y_i$, respectively, in $\C[T^* \C^N] = \C[x_1, \dots, x_N, y_1, \dots, y_N]$ for $i=1,\dots,N$. Recall the $\infty$-jet bundle $\calO_{J_{\infty} T^* \C^N}$ on $T^* \C^N$. For an open subset $U \subset T^* \C^N$,
\[
\calO_{J_{\infty}T^* \C^N}(U) = \calO_{T^* \C^N}(U) \otimes_{\C[T^* \C^N]} \C[x_{i (-n)}, y_{i (-n)} \,|\, i=1, \dots, N; n \in \Z_{>0}],
\]
and we set
\[
\calD^\ch_{T^* \C^N, \hbar}(U) = \calO_{J_{\infty}T^* \C^N}(U)[[\hbar]]
\]
as a $\C[[\hbar]]$-module. For example, if $U = U_f = \{ f \ne 0 \} \subset T^* \C^N$ is an affine open subset associated with a polynomial $f \in \C[x_{i (-1)}, y_{i (-1)} \,|\, i=1, \dots, N] = \C[T^* \C^N]$, then
\[
\calD^\ch_{T^* \C^N, \hbar}(U_f) =
\C[[\hbar]][x_{i (-n)}, y_{i (-n)}, f^{-1} \,|\, i=1, \dots, N; n \in \Z_{>0}]
\]
as a $\C[[\hbar]]$-module. Moreover, for open subsets $U' \subset U$, the restriction morphism $\Res^U_{U'}\colon \calD^\ch_{T^* \C^N, \hbar}(U) \longrightarrow \calD^\ch_{T^* \C^N, \hbar}(U')$ is induced by the restriction morphism $\calO_{T^* \C^N}(U) \longrightarrow \calO_{T^* \C^N}(U')$. Hence, we obtain a sheaf of $\C[[\hbar]]$-modules $\calD^\ch_{T^* \C^N, \hbar}$ on the affine space $T^* \C^N$. The global sections are $\calD^\ch_{T^* \C^N, \hbar}(T^* \C^N) = \calD^\ch(T^* \C^N)_{\hbar}$.

Then, in order to define an $\hbar$-adic vertex algebra structure on $\calD^\ch_{T^* \C^N, \hbar}$, we recall the following observation (see also \cite{AKM15}, Lemma~2.8.1.1). For a polynomial $f \in \C[x_{i (-n)}, y_{i (-n)} \,|\, i=1, \dots, N; n \in \Z_{>0}]$, we write $f(z) = Y(f \bfone, z)$ for the corresponding vertex operator. Then, the following lemma follows from the Wick formula (see \cite{Kac}, Theorem 3.3, \cite{FBZ}, Lemma~12.2.6).
\begin{lemma}\label{lemma:Wick-bidiff}
There exist bidifferential operators $P_{nk}$ on the polynomial algebra $\C[x_{i (-m)}, y_{i (-m)} \,|\, i=1, \dots, N; m \in \Z_{>0}]$ for $k \in \Z_{\ge 0}$, $n \in \Z$ such that the operator product expansion of $f(z) g(w)$ in $\calD^\ch(T^* \C^N)_{\hbar}$ satisfies
\[
f(z) g(w) = \sum_{k = 0}^{\infty} \sum_{n \in \Z} \frac{\hbar^{k}}{(z-w)^n} P_{nk}(f, g)
\]
for $f,g \in \calO_{J_{\infty} T^* \C^N}(T^* \C^N) = \C[x_{i (-m)}, y_{i (-m)} \,|\, i=1, \dots, N; m \in \Z_{>0}]$.
\end{lemma}

For an open subset $U \subset T^* \C^N$, we shall define the operator product expansion of $f(z) g(w)$ for $f,g \in \calO_{J_{\infty }T^* \C^N}(U)$. In order to do so by using \autoref{lemma:Wick-bidiff}, we regard the elements $f$ and $g$ as rational functions in the variables $x_{i (-n)}$ and $y_{i (-n)}$ for $i=1,\dots,N$, $n \in \Z_{>0}$. By definition, there is a projection
\[
\sigma_0\colon \calD^\ch_{T^* \C^N, \hbar}(U) \longrightarrow \calO_{J_{\infty} T^* \C^N}(U), \qquad f \mapsto f \mod \hbar,
\]
which we call the symbol map. Let
\begin{equation}\label{eq:lift-map}
\iota\colon \calO_{J_{\infty} T^* \C^N}(U) \longrightarrow \calD^\ch_{T^* \C^N, \hbar}(U), \qquad f \mapsto \iota(f) = f \bfone
\end{equation}
be a natural section of the symbol map $\sigma_0$. Note that $\iota$ does not send products to $(-1)$-products, i.e.\ $\iota(f)_{(-1)} \iota(g)$ is in general different from $\iota(f g)$. By abuse of notation, we also denote $\iota(f)$ simply by $f$. Recall that we write $f(z) = Y(\iota(f), z)$ for the vertex operator on $\calD^\ch_{T^* \C^N, \hbar}(U)$.

Note that $f$, $g$ are rational functions in $\{x_{i (-n)}, y_{i (-n)} \,|\, i=1, \dots, N; n \in \Z_{>0}\}$, and thus only finitely many variables may appear in them. This implies that the bidifferential operators $P_{nk}$ also act on $\calO_{J_{\infty}T^* \C^N}(U)$ and
$P_{nk}(f, g) \in \calO_{J_{\infty} T^* \C^N}(U)$ for any $k \in \Z_{\ge 0}$, $n \in \Z$. Hence, we define the operator product expansion by
\begin{equation}\label{eq:12}
f(z) g(w) =\sum_{k = 0}^{\infty} \sum_{n \in \Z} \frac{\hbar^{k}}{(z-w)^n} P_{nk}(f, g) \in \calD^\ch_{T^* \C^N, \hbar}(U)[[(z-w)^{-1}]]
\end{equation}
for $f,g \in \calO_{J_{\infty} T^* \C^N}(U)$. We extend the operator product expansion $\C[[\hbar]]$-linearly and continuously in the $\hbar$-adic topology for general $f,g \in \calD^\ch_{T^* \C^N, \hbar}(U)$. The translation operator $\partial$ can be defined on $\calD^\ch_{T^* \C^N, \hbar}(U)$ in the same way since it is also a differential operator. Note that the above definition of the operator product expansion \eqref{eq:12} induces $(n)$-products on $\calD^\ch_{T^* \C^N, \hbar}(U)$ through the formula \eqref{eq:OPE} $f(z) g(w) = \sum_{n \in \Z} Y(\iota(f)_{(n)} \iota(g), w) (z-w)^{-n-1}$, namely
\begin{equation}\label{eq:16}
\iota(f)_{(n)} \iota(g) = \sum_{k \ge 0} \hbar^{k} \iota(P_{n k}(f, g)).
\end{equation}
Then, we obtain an $\hbar$-adic vertex algebra $\calD^\ch_{T^* \C^N, \hbar}(U)$ with vacuum vector $\bfone$ by the following lemma:
\begin{lemma}\label{lemma:local-Borcherds}
For any open subset $U \subset T^* \C^N$, Borcherds' identity \eqref{eq:Borcherds} holds on $\calD^\ch_{T^* \C^N, \hbar}(U)$ with respect to the $(n)$-products defined above.
\end{lemma}
\begin{proof}
For a point $p \in U$, let $\widehat{\calO}_{T^* \C^N, p}$ be the formal completion of the stalk $\calO_{T^* \C^N, p}$ with respect to the unique maximal ideal $\frm_p$. Set
\[
\bigl(\calD^\ch_{T^* \C^N, \hbar}\bigr)^{\wedge}_p = \widehat{\calO}_{T^* \C^N, p} \otimes_{\C[T^* \C^N]} \C[[\hbar]][x_{i (-n)}, y_{i (-n)} \,|\, i=1, \dots, N; n \in \Z_{>0}].
\]
Then we define the $(n)$-products on $(\calD^\ch_{T^* \C^N, \hbar})^{\wedge}_p$ in the same way as the ones on $\calD^\ch_{T^* \C^N, \hbar}(U)$. By definition, the $(n)$-products are continuous in the $\frm_p$-adic topology on $(\calD^\ch_{T^* \C^N, \hbar})^{\wedge}_p \times (\calD^\ch_{T^* \C^N, \hbar})^{\wedge}_p$. Moreover, the natural embedding
\begin{equation}\label{eq:14}
\calD^\ch_{T^* \C^N, \hbar}(U) \hookrightarrow \bigl(\calD^\ch_{T^* \C^N, \hbar}\bigr)^{\wedge}_p
\end{equation}
given by Taylor expansion commutes with the $(n)$-products. We furthermore set $\tilde{x}_{i (-1)} = x_{i (-1)} - x_i(p)$, $\tilde{y}_{i (-1)} = y_{i (-1)} - y_i(p)$, $\tilde{x}_{i (-n)} = x_{i (-n)}$ and $\tilde{y}_{i (-n)} = y_{i (-n)}$ for $i=1,\dots,N$ and $n \in \Z_{\ge 2}$. For any polynomials in the variables $\{\tilde{x}_{i (-n)}, \tilde{y}_{i (-n)} \,|\, i=1, \dots, N; n \in \Z_{>0}\}$, Borcherds' identity holds on $(\calD^\ch_{T^* \C^N, \hbar})^{\wedge}_p$ since it coincides with the one on the $\hbar$-adic $\beta\gamma$-system $\calD^\ch(T^* \C^N)_{\hbar}$ by the definition of the $(n)$-products. Then, Borcherds' identity must hold on $(\calD^\ch_{T^* \C^N, \hbar})^{\wedge}_p$ for any formal power series in the variables $\{\tilde{x}_{i (-n)}, \tilde{y}_{i (-n)} \,|\, i=1, \dots, N; n \in \Z_{>0}\}$ since the $(n)$-products are continuous in the $\frm_p$-adic topology, and thus it holds on $\calD^\ch_{T^* \C^N, \hbar}(U)$ by the embedding~\eqref{eq:14}.
\end{proof}

Clearly, the restriction $\Res^{U}_{U'}\colon \calD^\ch_{T^* \C^N, \hbar}(U) \longrightarrow \calD^\ch_{T^* \C^N, \hbar}(U')$ is an $\hbar$-adic vertex algebra homomorphism for open subsets $U' \subset U$, and hence we obtain the sheaf of $\hbar$-adic vertex algebras $\calD^\ch_{T^* \C^N, \hbar}$ over the affine space $T^* \C^N$. Finally, the operator product expansion \eqref{eq:12} yields that $f(z) g(w) \equiv \NO f(z) g(w) \NO \bmod \hbar$ for any symbol $f,g \in \calO_{J_{\infty} T^* \C^N}(U)$ since $P_{n0}(f, g) = 0$ for $n \ge 0$, and thus there is natural isomorphism of vertex Poisson algebras
\[
\calD^\ch_{T^* \C^N, \hbar}(U)/\hbar \calD^\ch_{T^* \C^N, \hbar}(U) \simeq \calO_{J_{\infty}T^* \C^N}(U)
\]
for any open subset $U$. Therefore, $\calD^\ch_{T^* \C^N, \hbar}$ is a quantisation of $\calO_{J_{\infty} T^* \C^N}$.

For an invertible function $f \in \calO_{J_{\infty} T^* \C^N}(U)$, the element $\iota(1/f) \in \calD^\ch_{T^* \C^N, \hbar}(U)$ is in general not an inverse of $\iota(f)$ with respect to the $(-1)$-product (the normally ordered product). However, the following proposition ensures the existence of the inverse with respect to the $(-1)$-product:
\begin{proposition}\label{prop:inverse-wrt-NO}
Assume that $f_0 = \sigma_0(f)$ is invertible in the commutative algebra $\calO_{J_{\infty} T^* \C^N}(U)$. We write $1/f_0$ for the inverse element of $f_0$ in $\calO_{J_{\infty} T^* \C^N}(U)$ with respect to the usual multiplication. Then, there exists an element $g \in \calD^\ch_{T^* \C^N, \hbar}(U)$ such that $f_{(-1)} g = \bfone$ in $\calD^\ch_{T^* \C^N, \hbar}(U)$ and $\sigma_0(g) = 1/f_0$. In particular,
\[
\NO f(z) g(z) \NO = Y(f_{(-1)} g, z) = \Id_{\calD^\ch_{T^* \C^N, \hbar}(U)}.
\]
\end{proposition}
\begin{proof}
Write $f = \sum_{k \ge 0} \hbar^k \iota(f_k)$, $g = \sum_{k \ge 0} \hbar^k \iota(g_k)$, where $f_k$, $g_k \in \calO_{J_{\infty} T^* \C^N}(U)$. Then, by \eqref{eq:16},
\begin{align}\label{eq:15}
f_{(-1)} g &= \sum_{k \ge 0} \hbar^{k} \iota(P_{-1 k}(f, g)) = \iota(f_0 g_0) + \hbar \bigl( \iota(f_0 g_1 + f_1 g_0) + \iota(P_{-1, 1}(f_0, g_0)) \bigr) \\
&\quad + \hbar^2 \bigl(\iota(f_0 g_2 + f_1 g_1 + f_2 g_0) \nonumber\\
&\quad + \iota(P_{-1, 1}(f_0, g_1) + P_{-1, 1}(f_1, g_0)) + \iota(P_{-1, 2}(f_0, g_0)) \bigr) + \dots\nonumber
\end{align}
Setting $g_0 = 1/f_0$ and
\[
g_1 = - \frac{1}{f_0}\bigl( f_1 g_0 + P_{-1, 1}(f_0, g_0) \bigr) \in \calO_{J_{\infty} T^* \C^N}(U),
\]
it follows that $f_{(-1)} (\iota(g_0) + \hbar \iota(g_1)) - \bfone \equiv 0 \bmod \hbar^2$. By induction on $k=1,2,\dots$ we can determine $g_k$ such that
\[
f_{(-1)} (\iota(g_0) + \hbar^1 \iota(g_1) + \dots + \hbar^{k} \iota(g_{k})) - \bfone \equiv 0 \mod \hbar^{k+1}.
\]
Then, since $\calD^\ch_{T^* \C^N, \hbar}(U)$ is complete in the $\hbar$-adic topology, we obtain the element $g = \sum_{k \ge 0} \hbar^k \iota(g_k)$ with the desired properties.
\end{proof}


\subsection{The \texorpdfstring{$bc$}{bc}-System}
\label{sec:h-adic-bc}

We introduce the $\hbar$-adic analogue $\Cl_{\hbar}(T^*\C^N)$ of the $bc$-system $\Cl(T^*\C^N)=(bc)^{\otimes N}$ of rank $N$, which is also called Clifford vertex superalgebra or $bc$-ghost vertex superalgebra.

\medskip

Again, consider the symplectic vector space $T^* \C^N = \C^N \oplus (\C^N)^*$. We write $\Pi V$ for the odd vector space corresponding to an even vector space $V$. Let $\psi_1,\dots,\psi_N$, $\phi_1,\dots,\phi_N$ be the standard Poisson coordinate functions on the odd vector space $\Pi T^* \C^N$, i.e.\ $\{\psi_i, \phi_j \} = \delta_{ij}$ and $\{\psi_i, \psi_j\} = \{\phi_i, \phi_j\} = 0$ for $i,j=1,\dots,N$, where $\{\blkbar, \blkbar\}$ denotes the Poisson superbracket of the Poisson superalgebra $\C[\Pi T^* \C^N]$.

The $\hbar$-adic $bc$-system $\Cl_{\hbar}(T^* \C^N)$ on $\Pi T^* \C^N$ is the $\hbar$-adic vertex superalgebra generated by the fermions $\psi_1,\dots,\psi_N$, $\phi_1,\dots,\phi_N$ with operator product expansions
\begin{align*}
\psi_i(z) \phi_j(w) &\sim \delta_{ij} \hbar /(z-w), \\
\psi_i(z) \psi_j(w) &\sim 0, \\
\phi_i(z) \phi_j(w) &\sim 0
\end{align*}
for $i,j=1,\dots,N$, where we again write $\psi_i(z) = Y(\psi_i, z)$, $\phi_i(z) = Y(\phi_i, z)$.
As a $\C[[\hbar]]$-module,
\[
\Cl_{\hbar}(T^* \C^N) = \Lambda_{\C[[\hbar]]}(\psi_{1 (-n)}, \dots, \psi_{N (-n)}, \phi_{1 (-n)}, \dots, \phi_{N (-n)} \,|\, n \in \Z_{>0}).
\]

Let $\Lambda^\vertex(T^* \C^N) = \Cl_{\hbar}(T^* \C^N) / \hbar \Cl_{\hbar}(T^* \C^N)$. It follows from the definition that $\psi_i(z) \phi_j(w) \sim \psi_i(z) \psi_j(w) \sim \phi_i(z) \phi_j(w) \sim 0 \bmod \hbar$ for $i,j=1,\dots,N$, and thus $\Lambda^\vertex(T^* \C^N)$ is a commutative vertex superalgebra. Note that $\Lambda^\vertex(T^* \C^N)$ becomes a vertex Poisson superalgebra upon defining
\[
Y_{+}(\psi_i, z) = \hbar^{-1} \sum_{n \ge 0} \psi_{i (n)} z^{-n-1}, \quad Y_{+}(\phi_i, z) = \hbar^{-1} \sum_{n \ge 0} \phi_{i (n)} z^{-n-1}
\]
for $i=1,\dots,N$.

\medskip

To construct the BRST cohomology, we shall later also need the (ghost) $\hbar$-adic $bc$-system $\Cl_{\hbar}(T^* \frg)$ associated with the symplectic vector space $T^* \frg = \frg \oplus \frg^*$ for $\frg = \frgl_N$. In this case, we consider the standard coordinate functions $\Psi_{ij} \in \frg^* \subset \C[\frg]$ and $\Phi_{ij} \in \frg \subset \C[\frg^*]$ for $i,j=1,\dots,N$, which correspond to the matrices $E_{ij} \in \frg$ (see above). Thus,
\[
\Cl_{\hbar}(T^* \frg) = \Lambda_{\C[[\hbar]]}(\Psi_{ij (-n)}, \Phi_{ij (-n)} \,|\, \substack{i,j=1, \dots, N \\ n = 1, 2, \dots})
\]
as a $\C[[\hbar]]$-module, and the operator product expansions between the generators are given by
\begin{align*}
\Psi_{ij}(z) \Phi_{kl}(w) &\sim \delta_{ik} \delta_{jl} \hbar / (z-w), \\
\Psi_{ij}(z) \Psi_{kl}(w) &\sim 0, \\
\Phi_{ij}(z) \Phi_{kl}(w) &\sim 0.
\end{align*}
Note that the $\hbar$-adic $bc$-system $\Cl_{\hbar}(T^* \frg)$ is naturally $\Z$-graded by the charge (or ghost) grading given by $\deg(\Psi_{ij (-n)}) = 1$, $\deg(\Phi_{ij (-n)}) = -1$ for $i,j=1,\dots,N$ and $n \in \Z_{>0}$. Moreover, the homogeneous subspace $\Cl^n_{\hbar}(T^* \frg)$ of charge $n \in \Z$ can be decomposed as a $\C[[\hbar]]$-module as
\begin{equation}\label{eq:bc-double-grade}
\Cl^n_{\hbar}(T^* \frg) = \prod_{p+q=n} \Cl^{p,q}_{\hbar}(T^* \frg)
\end{equation}
with
\[
\Cl^{p,q}_{\hbar}(T^* \frg) = \Lambda^{p}_{\C[[\hbar]]}(\Psi_{ij (-n)} \,|\, \substack{i,j=1, \dots, N \\ n = 1, 2, \dots}) \hatotimes \Lambda^{-q}_{\C[[\hbar]]}(\Phi_{ij (-n)} \,|\, \substack{i,j=1, \dots, N \\ n = 1, 2, \dots}), \nonumber
\]
where $\Lambda^{p}_{\C[[\hbar]]}(\blkbar)$ is the space of $p$-th exterior powers. Note that $\Cl^{p,q}_{\hbar}(T^* \frg) = 0$ unless $p \ge 0$ and $q \le 0$.


\subsection{Affine Vertex Algebras and Small \texorpdfstring{$\calN=4$}{N=4} Superconformal Algebra}
\label{sec:h-adic-affine-VA}

We also introduce the $\hbar$-adic affine vertex algebras $V^k(\frg)_{\hbar}$ and the $\hbar$-adic small $\calN=4$ superconformal algebra $\on{Vir}_{\calN=4,\hbar}^c$. Let $\frg$ be a finite-dimensional Lie algebra with invariant bilinear form $\kappa\colon \frg \times \frg \longrightarrow \C$. Consider the universal affine vertex algebra $V^{\kappa}(\frg)$ \cite{FZ92}, and define $V^{\kappa}(\frg)[\hbar] = V^{\kappa}(\frg) \otimes \C[\hbar]$. For an element $A \in \frg$, we write $\widehat{A} = \hbar A_{(-1)} \bfone \in V^{\kappa}(\frg)[\hbar]$. Then,
\begin{equation}\label{eq:OPE-affine-VA}
\widehat{A}_{(0)} \widehat{B} = \hbar \, \widehat{[A, B]}, \quad \widehat{A}_{(1)} \widehat{B} = \hbar^2 \,\kappa(A, B) \bfone, \quad \widehat{A}_{(n)} \widehat{B} = 0 \text{ if } n > 1
\end{equation}
for $A$, $B \in \frg$. The universal $\hbar$-adic affine vertex algebra $V^{\kappa}(\frg)_{\hbar}$ associated with $\frg$ and $\kappa$ is the $\hbar$-adic completion of the vertex subalgebra generated by the elements in $\{\widehat{A} \,|\, A \in \frg\}$. Note that $V^{\kappa}(\frg)_{\hbar} / (\hbar)$ is a vertex Poisson algebra, and it satisfies $V^{\kappa}(\frg)_{\hbar} / (\hbar) \simeq S(\frg \otimes \C[t^{-1}] t^{-1})$. By \eqref{eq:OPE-affine-VA}, the vertex Poisson structure $Y_{-}(\blkbar, z)$ on it does not depend on $\kappa$.

In the present article, we consider the case where $\frg$ is the general linear Lie algebra $\frgl_N(\C)$. In this case, the invariant bilinear form is given explicitly by $\kappa(A, B) = \kappa_k(A, B) = k \Tr(AB) - (k/N) \Tr(A) \Tr(B)$ for some level $k \in \C$. We also use the notation $V^k(\frg)_{\hbar}$ for $V^{\kappa_k}(\frg)_{\hbar}$.

\medskip

Similarly, we obtain an $\hbar$-adic analogue of the small $\calN=4$ superconformal algebra (see, e.g., \cite{Kac}). The universal small $\calN=4$ superconformal algebra $\on{Vir}_{\calN=4}^c$ of central charge~$c$ (or equivalently of level $k=c/6$) is a vertex operator superalgebra strongly generated by bosonic elements $J^0$, $J^{\pm}$ and the conformal vector $\TN$ and the fermionic ones $G^{\pm}$ and $\tldG^{\pm}$ satisfying the operator product expansions
\begin{align*}
J^0(z) J^{\pm}(w) &\sim \frac{\pm 2}{z-w} J^{\pm}(w), \quad J^{0}(z) J^{0}(w) \sim \frac{c/3}{(z-w)^2}, \\
J^{+}(z) J^{-}(w) &\sim \frac{-c/6}{(z-w)^2} + \frac{-1}{z-w} J^{0}(w) \\
J^{0}(z) B^{\pm}(w) &\sim \frac{\pm 1}{z-w} B^{\pm}(w), \quad J^{\pm}(z) B^{\mp}(w) \sim \frac{\mp 1}{z-w} B^{\pm}(w),
\end{align*}
as well as
\begin{align*}
G^{a}(z) \tldG^{b}(w) &\sim \frac{(c/3) \epsilon^{ab}}{(z-w)^3} + \frac{2 J^{ab}(w)}{(z-w)^2} + \frac{\epsilon^{ab} \TN(w) + \partial J^{ab}(w)}{z-w}, \\
\TN(z) \TN(w) &\sim \frac{c/2}{(z-w)^4} + \frac{2}{(z-w)^2} \TN(w) + \frac{1}{z-w} \partial \TN(w), \\
\TN(z) A(w) &\sim \frac{1}{(z-w)^2} A(w) + \frac{1}{z-w} \partial A(w),\\
\TN(z) B^{\pm}(w) &\sim \frac{3 / 2}{(z-w)^2} B^{\pm}(w) + \frac{1}{z-w} \partial B^{\pm}(w),
\end{align*}
with the operator product expansions of other pairs among the elements being trivial, for $A = J^0, J^{\pm}$, $B = G,\tldG$ and $a,b = \pm$, where $\epsilon^{+-} = 1$, $\epsilon^{-+} = -1$, $\epsilon^{\pm \pm} = 0$ and $J^{+-} = J^{-+} = J^0/2$, $J^{\pm\pm} = J^{\pm}$. Note that the vertex superalgebra $\on{Vir}_{\calN=4}^c$ contains the universal affine vertex algebra $V^k(\frsl_2)$ of level $k=c/6$, which is also the level of $\on{Vir}_{\calN=4}^c$ by definition.

Consider the vertex subalgebra of $\on{Vir}_{\calN=4}^c[\hbar]$ generated by the elements $\hbar J^0$, $\hbar J^{\pm}$, $\hbar^2 \TN$, $\hbar^{3/2} G^{\pm}$ and $\hbar^{3/2} \tldG^{\pm}$, and define the universal $\hbar$-adic small $\calN=4$ superconformal algebra $\on{Vir}_{\calN=4,\hbar}^c$ of central charge $c$ as its $\hbar$-adic completion.


\subsection{Symplectic Fermion Vertex Superalgebra}
\label{sec:symplec-ferm}

Finally, we introduce the $\hbar$-adic version $\SF_{\hbar}$ of the symplectic fermion vertex superalgebra \cite{Kau00}. Let $V = \C \Lambda_1 \oplus \C \Lambda_2$ be a symplectic vector space with symplectic bilinear form given by $(\Lambda_1, \Lambda_2) = -1$ and $(\Lambda_i, \Lambda_i) = 0$ for $i=1,2$. The vertex superalgebra of symplectic fermions $\SF$ associated with $V$ is the vertex superalgebra generated by fermionic elements $\Lambda_1$ and $\Lambda_2$ with operator product expansion $u(z) v(w) \sim (u, v)/(z-w)^2$ for $u,v \in V$. The $\hbar$-adic analogue of $\SF$ is denoted by $\SF_{\hbar}$. It is an $\hbar$-adic vertex superalgebra generated by $\Lambda_1$ and $\Lambda_2$ with operator product expansion given by
\[
u(z) v(w) \sim \hbar^2 (u, v)/(z-w)^2
\]
for $u,v \in V$.


\section{Semi-Infinite BRST Reduction}
\label{sec:semi-infinite-red}

In this section, for $N\ge1$ we construct a sheaf $\tldcalD^\ch_{M, \hbar}$ of $\hbar$-adic vertex superalgebras on the Hilbert scheme $M=\Hilb^N(\C^2)$. This is the central object of this text. The construction is based on a vertex algebra analogue of quantum Hamiltonian reduction, namely the semi-infinite BRST cohomology (\cite{Fei84,KosSte87,GorMalSch01}, see also \cite{AKM15,Kuwabara21,AraMT}). We also prove a vanishing (or no-ghost) theorem for this BRST cohomology.


\subsection{Construction of the BRST Cohomology}
\label{sec:brst-construction}

In the following, we define a BRST cohomology sheaf $\calH_{\VA}^{\infty/2 + \bullet}(\frg, \tldcalD^\ch_{\frX, \hbar})$ on $M$, which we later call $\tldcalD^\ch_{M, \hbar}$, after having proved a vanishing theorem for the cohomological spaces.

\medskip

Recall the symplectic vector space $T^* V = \gEnd(\C^N)^{\oplus 2} \oplus (\C^N)^{\oplus 2}$ and its standard symplectic coordinates $\{x_{ij}, y_{ij}\}_{i,j=1, \dots, N} \cup \{\gamma_i, \beta_{i}\}_{i=1, \dots, N}$ from \autoref{sec:quiver-const}. Let $\calD^\ch_{T^* V, \hbar}$ be the microlocalisation of the $\hbar$-adic $\beta\gamma$-system, a sheaf of vertex algebras on $T^* V$, and let $\Cl_{\hbar}(T^* \C^N)$ be the $\hbar$-adic $bc$-system. Then, define the sheaf $\tldcalD^\ch_{T^*V, \hbar} = \calD^\ch_{T^*V, \hbar} \hatotimes \Cl_{\hbar}(T^* \C^N)$ of $\hbar$-adic vertex superalgebras on $T^*V$. Restricting onto $\frX\subset T^*V$, we obtain a sheaf of $\hbar$-adic vertex superalgebras $\tldcalD^\ch_{\frX, \hbar} \coloneqq \tldcalD^\ch_{T^*V, \hbar} |_{\frX}$.
As a $\C[[\hbar]]$-module,
\begin{align*}
\tldcalD^\ch_{\frX, \hbar}(\tldU) &= \calO_{\frX}(\tldU) \otimes_{\C[\frX]} \C[[\hbar]][x_{ij (-n)}, y_{ij (-n)}, \gamma_{i (-n)}, \beta_{i (-n)} \,|\, \substack{i,j = 1, \dots, N \\ n = 1, 2, \dots}] \\
&\quad \hatotimes \Lambda_{\C[[\hbar]]}(\psi_{i (-n)}, \phi_{i (-n)} \,|\, \substack{i = 1, \dots, N \\ n = 1, 2, \dots})
\end{align*}
for any open subset $\tldU \subset \frX$.

To construct the BRST reduction of $\tldcalD^\ch_{\frX, \hbar}$, we introduce a chiralisation $\tldmu_\ch$ of the comoment map $\tldmu\colon \frg \longrightarrow \tldcalO_{\frX}(\frX)$, which we call chiral comoment map. To this end, consider the $\hbar$-adic affine vertex algebra $V^{-2N}(\frg)_{\hbar}$ associated with the general linear Lie algebra $\frg = \frgl_N$ of level $-2N$. Note that the corresponding invariant bilinear form $\kappa_{-2N}(A, B) = - 2N \Tr(AB) + 2 \Tr(A) \Tr(B)$ is the negative of the Killing form of $\frg = \frgl_N$. Define the $\C[\partial]$-linear map $\tldmu_\ch\colon V^{-2N}(\frg)_{\hbar}\longrightarrow\tldcalD^\ch_{\frX, \hbar}(\frX)$,
\[
\tldmu_\ch(E_{ij}) = \sum_{p=1}^{N} (x_{ip (-1)} y_{pj} - x_{pj (-1)} y_{ip}) + \gamma_{i (-1)} \beta_j + \psi_{i (-1)} \phi_j.
\]
By direct verification, we obtain the following lemma.
\begin{lemma}\label{lemma:chiral-comoment-hom}
The chiral comoment map $\tldmu_\ch\colon V^{-2N}(\frg)_{\hbar}\longrightarrow\tldcalD^\ch_{\frX, \hbar}(\frX)$ is a homomorphism of $\hbar$-adic vertex superalgebras.
\end{lemma}

Now, recall the ghost $\hbar$-adic $bc$-system $\Cl_{\hbar}(T^* \frg)$ with generators $\Psi_{ij}$, $\Phi_{ij}$ for $i,j=1,\dots,N$. It is $\Z$-graded by the charge (or ghost) grading defined in \autoref{sec:h-adic-bc}, i.e.\ $\Cl_{\hbar}(T^* \frg) = \prod_{\bullet \in \Z} \Cl^{\bullet}_{\hbar}(T^* \frg)$. Let $c_{ijkl}^{pq}$ denote the structure constants of the Lie algebra $\frg = \frgl_N$, i.e.\ $[E_{ij}, E_{kl}] = \sum_{pq} c_{ijkl}^{pq} E_{pq}$ for $i,j,k,l=1,\dots,N$. Then, via
\[
J(E_{ij})\coloneqq - \sum_{klpq} c_{ijkl}^{pq} \Psi_{kl} \Phi_{pq}
\]
we define a $\C[\partial]$-linear map $J\colon V^{2N}(\frg)_{\hbar}\longrightarrow\Cl_{\hbar}^0(T^* \frg)$. Again, a straightforward computation shows:
\begin{lemma}\label{lemma:ghost-hom}
The map $J\colon V^{2N}(\frg)_{\hbar}\longrightarrow\Cl_{\hbar}^0(T^* \frg)$ is a homomorphism of $\hbar$-adic vertex superalgebras.
\end{lemma}

We then define the sheaf
\[
\tldC_{\VA} = \prod_{\bullet \in \Z} \tldC_{\VA}^{\bullet}, \qquad \tldC_{\VA}^{\bullet} = \tldcalD^\ch_{\frX, \hbar} \hatotimes \Cl^{\bullet}_{\hbar}(T^* \frg),
\]
of $\Z$-graded $\hbar$-adic vertex superalgebras on $\frX$. On this sheaf we consider the odd global section
\begin{align*}
Q &\coloneqq \sum_{i,j} \bigl(\tldmu_\ch(E_{ij})+\frac{1}{2}J(E_{ij})\bigr)_{(-1)}\Psi_{ij}\\
&=\sum_{i,j} \tldmu_\ch(E_{ij})_{(-1)} \Psi_{ij} - \frac{1}{2} \sum_{ijklpq} c_{ijkl}^{pq} \Psi_{ij} \Psi_{kl} \Phi_{pq} \in \tldC_{\VA}^1(\frX)
\end{align*}
of charge $+1$. Note that the image of $Q_{(0)}$ lies in $\hbar \tldC_{\VA}$, and we define a derivation $d_{\VA} = (1/\hbar) Q_{(0)}$ on $\tldC_{\VA}$, which is homogeneous of charge $+1$. Crucially, a straightforward calculation shows that $Q_{(0)} Q = 0$, which implies
\[
(d_{\VA})^2 = (Q_{(0)})^2 = (1/2) (Q_{(0)} Q)_{(0)} = 0.
\]
Therefore:
\begin{proposition}\label{prop:total-BRST-cpx}
For any $\tldU \subset \frX$ open, $(\tldC_{\VA}(\tldU), d_{\VA})$ is a cochain complex.
\end{proposition}

We define the global sections
\[
\tldE_{ij} = Q_{(0)} \Phi_{ij} = \tldmu_\ch(E_{ij}) + J(E_{ij}) \in \tldC_{\VA}^{0}(\frX)
\]
for $i,j=1,\dots,N$, which induce a $\C[\partial]$-linear map $V^0(\frg)_{\hbar}\longrightarrow\tldC_{\VA}^0(\frX)$. Then \autoref{lemma:chiral-comoment-hom} and \autoref{lemma:ghost-hom} imply:
\begin{lemma}\label{lemma:sln-rep}
The map $E_{ij}\mapsto\tldE_{ij}$ induces a homomorphism $V^0(\frg)_{\hbar}\longrightarrow\tldC_{\VA}^0(\frX)$ of $\hbar$-adic vertex superalgebras.
\end{lemma}
In particular, the operators $\{ (1/ \hbar) \tldE_{ij (0)} \,|\, i, j = 1, \dots, N \}$ form the general linear Lie algebra $\frg = \frgl_N$, and make $\tldC_{\VA}(\frX)$ into a $\frg$-module.

\medskip

It shall be advantageous to consider the BRST cohomology only on a relative subcomplex
\cite{FreGarZuc86}. To this end, we consider the subspace of $\tldC_{\VA}$ given by
\[
C_{\VA} \coloneqq \{ c \in \tldC_{\VA} \,|\, \tldE_{ij (0)} c = \Phi_{ij (0)} c = 0 \text{ for all } i, j = 1, \dots, N \}.
\]
As $Q_{(0)}$ is a derivation,
\begin{align*}
\tldE_{ij (0)} Q_{(0)} c &= Q_{(0)} \tldE_{ij (0)} c - (Q_{(0)} \tldE_{ij})_{(0)} c = (Q_{(0)}^2 \Phi_{ij})_{(0)} c = 0,\\
\Phi_{ij (0)} Q_{(0)} c &= - Q_{(0)} \Phi_{ij (0)} c + (Q_{(0)} \Phi_{ij})_{(0)} c = \tldE_{ij (0)} c = 0
\end{align*}
for any section $c \in C_{\VA}$ and $i,j=1,\dots,N$. Therefore, $(C_{\VA}, d_{\VA})$ is a sheaf of subcomplexes of $(\tldC_{\VA}, d_{\VA})$.
For an open subset $\tldU \subset \frX$, we write the cohomology
\[
H^{\infty/2+\bullet}_{\VA}(\frg, \tldcalD^\ch_{\frX, \hbar}(\tldU)) \coloneqq H^{\bullet}(C_{\VA}(\tldU), d_{\VA})
\]
and call it the (relative) BRST cohomology of $\tldcalD^\ch_{\frX, \hbar}(\tldU)$ with respect to $\frg$.

\begin{lemma}
The BRST cohomology $H_{\VA}^{\infty/2+\bullet}(\frg, \tldcalD^\ch_{\frX, \hbar}(\tldU))$ is equipped with the structure of a ($\Z$-graded) $\hbar$-adic vertex superalgebra.
\end{lemma}
\begin{proof}
Again using that the coboundary operator $Q_{(0)}$ is a derivation, we see that $Q_{(0)} (a_{(n)} b) = (-1)^{p(a)} a_{(n)} Q_{(0)} b + (Q_{(0)} a)_{(n)} b$ for $a,b \in C_{\VA}$ and $n \in \Z$. This implies that $a_{(n)} b \in \Ker d_{\VA}$ for $a,b\in\Ker d_{\VA}$. Also, $a_{(n)} (Q_{(0)} b) = Q_{(0)} (a_{(n)} b)$ for $a \in \Ker d_{\VA}$ and $b \in C_{\VA}$. Thus, $(\Ker d_{\VA})_{(n)} (\Im d_{\VA}) \subset \Im d_{\VA}$.
\end{proof}

Now, using the BRST cohomology, we construct a sheaf of $\hbar$-adic vertex superalgebras on the Hilbert scheme $M$. For an open subset $U \subset M$, take an open subset $\tldU \subset \frX$ such that $\tldU$ is closed under the $G$-action and $U = (\tldU \cap \mu^{-1}(0))/G$. Crucially, by \autoref{lemma:support-BRST-cohom} below, the BRST cohomology $H_{\VA}^{\infty/2 + \bullet}(\frg, \tldcalD^\ch_{\frX, \hbar}(\tldU))$ does not depend on the choice of the open subset $\tldU$ satisfying these conditions. Therefore, it defines a presheaf $U \mapsto H_{\VA}^{\infty/2 + \bullet}(\frg, \tldcalD^\ch_{\frX, \hbar}(\tldU))$ over $M$, which we use to define the BRST cohomology sheaf:
\begin{definition}\label{def:BRST-sheaf}
The sheaf on the Hilbert scheme $M=\Hilb^N(\C^2)$ associated with the presheaf $U \mapsto H_{\VA}^{\infty/2 + \bullet}(\frg, \tldcalD^\ch_{\frX, \hbar}(\tldU))$ is denoted by $\calH_{\VA}^{\infty/2 + \bullet}(\frg, \tldcalD^\ch_{\frX, \hbar})$.
\end{definition}

\medskip

We comment that the BRST cochain complex can be decomposed into a double cochain complex. Recall the decomposition \eqref{eq:bc-double-grade} of the $\hbar$-adic $bc$-system $\Cl^{p,q}_{\hbar}(T^* \frg)$. Set $\tldC_{\VA}^{p,q} = \tldcalD^\ch_{\frX, \hbar} \hatotimes \Cl^{p,q}_{\hbar}(T^* \frg)$ for $p,q \in \Z$. Then, $\tldC_{\VA}^{n} = \prod_{p + q = n} \tldC_{\VA}^{p, q}$ for any $n \in \Z$. Define
\begin{align}\label{eq:d-plus}
d^{+}_{\VA} &\coloneqq \hbar^{-1} \sum_{i,j} \sum_{n \ge 0} \Psi_{ij (-n-1)} \tldE_{ij (n)} \\
&\quad + \hbar^{-1} \sum_{ijklpq} \sum_{n \ge 0} \sum_{m=0}^{n} c_{ijkl}^{pq} \Psi_{ij (-n+m-1)} \Psi_{kl (-m)} \Phi_{pq (n)} \nonumber \\
&= \sum_{i,j} \sum_{n \ge 0} \Psi_{ij (-n-1)} (\hbar^{-1} \tldE_{ij (n)}) \nonumber \\
&\quad + \sum_{ijklpq} \sum_{n \ge 0} \sum_{m=0}^{n} c_{ijkl}^{pq} \Psi_{ij (-n+m-1)} \Psi_{kl (-m)} \frac{\partial}{\partial \Psi_{pq (-n-1)}} \nonumber
\end{align}
and
\begin{align}\label{eq:d-minus}
d^{-}_{\VA} &\coloneqq \hbar^{-1} \sum_{i,j} \sum_{n \ge 0} \tldmu_\ch(E_{ij})_{(-n-1)} \Psi_{ij (n)} \\
&\quad + \hbar^{-1} \sum_{ijklpq} \sum_{n \ge 0} \sum_{m=0}^{n} c_{ijkl}^{pq} \Phi_{pq (-n-1)} \Psi_{ij (n-m)} \Psi_{kl (m)} \nonumber \\
&= \sum_{i,j} \sum_{n \ge 0} \tldmu_\ch(E_{ij})_{(-n-1)} \frac{\partial}{\partial \Phi_{ij (-n-1)}} \nonumber \\
&\quad + \hbar \sum_{ijklpq} \sum_{n \ge 0} \sum_{m=0}^{n} c_{ijkl}^{pq} \Phi_{pq (-n-1)} \frac{\partial}{\partial \Phi_{ij (-n+m-1)}} \frac{\partial}{\partial \Phi_{kl (-m-1)}}. \nonumber
\end{align}
Then, $d^{+}_{\VA}$ maps from $\tldC_{\VA}^{p, q}$ to $\tldC_{\VA}^{p+1, q}$, $d^{-}_{\VA}$ maps from $\tldC_{\VA}^{p, q}$ to $\tldC_{\VA}^{p, q+1}$, and $d_{\VA} = d^+_{\VA} + d^-_{\VA}$ and $d^+_{\VA} \circ d^-_{\VA} = - d^-_{\VA} \circ d^+_{\VA}$ hold. Thus, we obtain a double complex $(\tldC_{\VA}, d^+_{\VA}, d^-_{\VA})$ whose total complex is the cochain complex $(\tldC_{\VA}, d_{\VA})$. Setting $C_{\VA}^{p,q} = C_{\VA} \cap \tldC^{p,q}_{\VA}$, we obtain a double complex $(C_{\VA}^{p,q}, d_{\VA}^{+}, d_{\VA}^{-})$ whose total complex is the BRST complex $(C_{\VA}, d_{\VA})$.


\subsection{Poisson BRST Reduction}
\label{sec:Poisson-brst}

By considering the associated graded, we make the sheaf $(C_{\VA}, d_{\VA})$ of cochain complexes of $\hbar$-adic vertex superalgebras into a sheaf $(C_{\VPA}, d^{+}_{\VPA}, d^{-}_{\VPA})$ of double complexes of vertex Poisson superalgebras.

\medskip

The cochain complex $\tldC_{\VA} = \tldcalD^\ch_{\frX, \hbar} \hatotimes \Cl_{\hbar}(T^* \frg)$ is naturally equipped with a filtration $\hF_{\bullet} \tldC_{\VA}$ by powers of $\hbar$: $\hF_p \tldC_{\VA} \coloneqq \hbar^p \tldC_{\VA}$ for $p \in \Z_{\ge 0}$. For each $p \in \Z_{\ge 0}$, the associated graded space is
\[
\hGr_{p} \tldC_{\VA} = \hF_{p} \tldC_{\VA} / \hF_{p+1} \tldC_{\VA} \simeq \JetBundle{\frX} \otimes \Lambda^\vertex(T^* \frg)
\]
as commutative vertex superalgebra. Recall that there is a natural vertex Poisson superalgebra structure on $\hGr_{p} \tldC_{\VA}$ with the Poisson structure $Y_{+}(a, z)$ for $a \in \hGr_{p} \tldC_{\VA}$ given by $Y_{+}(a, z) = \hbar^{-1} Y(a, z) \bmod \hbar$. By abuse of notation, we write $a_{(n)}$ for the modes of $Y_{+}(a, z) = \sum_{n \ge 0} a_{(n)} z^{-n-1}$, i.e.\ the operator $a_{(n)}$ on $\hGr_{p} \tldC_{\VA}$ is the one induced from $\hbar^{-1} a_{(n)} \bmod \hbar$ on $\tldC_{\VA}$. Then, the isomorphism $\hGr_{p} \tldC_{\VA} \simeq \JetBundle{\frX} \otimes \Lambda^\vertex(T^* \frg)$ is one of vertex Poisson superalgebras.

By restriction, we obtain the filtered complex $\hF_p C_{\VA} = C_{\VA} \cap \hF_p \tldC_{\VA}$. For $p \in \Z_{\ge 0}$, the associated graded space is the vertex Poisson superalgebra
\[
\hGr_{p} C_{\VA} = \{ c \in \JetBundle{\frX} \otimes \Lambda^\vertex(T^* \frg) \,|\, \tldE_{ij (0)} c = \Phi_{ij (0)} c = 0 \text{ for } i,j = 1, \dots, N \}.
\]
Note that $\hGr_{p} \tldC_{\VA}$ and $\hGr_{p} C_{\VA}$ are independent of $p \in \Z_{\ge 0}$.

Then, the coboundary operators $d_{\VA}^{+}$ and $d_{\VA}^{-}$ induce coboundary operators on $\hGr_{p} \tldC_{\VA}$ and $\hGr_p C_{\VA}$. They can be explicitly described as
\begin{align}\label{eq:d-VPA-plus}
d^{+}_{\VPA} &\coloneqq \sum_{i,j} \sum_{n \ge 0} \Psi_{ij (-n-1)} \tldE_{ij (n)} \\
&\quad + \sum_{ijklpq} \sum_{n \ge 0} \sum_{m=0}^{n} c_{ijkl}^{pq} \Psi_{ij (-n+m-1)} \Psi_{kl (-m)} \Phi_{pq (n)} \nonumber \\
&= \sum_{i,j} \sum_{n \ge 0} \Psi_{ij (-n-1)} \tldE_{ij (n)} \nonumber \\
&\quad + \sum_{ijklpq} \sum_{n \ge 0} \sum_{m=0}^{n} c_{ijkl}^{pq} \Psi_{ij (-n+m-1)} \Psi_{kl (-m)} \frac{\partial}{\partial \Psi_{pq (-n-1)}} \nonumber
\end{align}
and
\begin{equation}\label{eq:d-VPA-minus}
d^{-}_{\VPA} \coloneqq \sum_{i,j} \sum_{n \ge 0} \tldmu_{\infty}(E_{ij})_{(-n-1)} \Psi_{ij (n)} = \sum_{i,j} \sum_{n \ge 0} \tldmu_{\infty}(E_{ij})_{(-n-1)} \frac{\partial}{\partial \Phi_{ij (-n-1)}},
\end{equation}
where $\tldmu_{\infty} \colon S(\frg \otimes \C[t^{-1}] t^{-1}) \longrightarrow \JetBundle{\frX}(\frX)$ is a homomorphism of vertex Poisson superalgebras induced from the Lie algebra homomorphism $\tldmu\colon \frg \longrightarrow \tldcalO_{\frX}(\frX)$. Note that
\[
\JetBundle{\frX}(\frX) = \C[x_{ij (-n)}, y_{ij (-n)}, \gamma_{i (-n)}, \beta_{i (-n)}, \psi_{i (-n)}, \phi_{i (-n)} \,|\, \substack{i, j = 1, \dots, N \\ n = 1, 2, \dots}\,]
\]
as a commutative superalgebra. Then, the homomorphism $\tldmu_{\infty}$ is given by
\[
\tldmu_{\infty}(E_{ij}) = \sum_{p=1}^{N} (x_{ip (-1)} y_{pj} - x_{pj (-1)} y_{ip}) + \gamma_{i (-1)} \beta_{j} + \psi_{i (-1)} \phi_{j}
\]
for the matrices $E_{ij} \in \frg$ for $i,j=1,\dots,N$.

Setting $\tldC_{\VPA} = \hGr_{p} \tldC_{\VA} = \JetBundle{\frX} \otimes \Lambda^\vertex(T^* \frg)$ and $C_{\VPA} = \hGr_p C_{\VA}$, we obtain double complexes $(\tldC_{\VPA}, d^{+}_{\VPA}, d^{-}_{\VPA})$ and $(C_{\VPA}, d^{+}_{\VPA}, d^{-}_{\VPA})$, respectively, each defined on a vertex Poisson superalgebra.


\subsection{A Koszul Complex Associated with the BRST Complex}
\label{sec:koszul-cpx}

In the following, using the vertex Poisson algebra complex of the previous section, we associate a Koszul complex (see, e.g., Section 16 in \cite{Mat89}) with the BRST complex.

\medskip

For $p \in \Z$, we consider the vertex Poisson algebra complex $(\tldC^{p, \bullet}_{\VPA}, d^{-}_{\VPA})$ from above. By the explicit description \eqref{eq:d-VPA-minus} of the coboundary operator $d^{-}_{\VPA}$, for an open subset $\tldU \subset \frX$, the complex $(\tldC^{p, \bullet}_{\VPA}(\tldU), d^{-}_{\VPA})$ coincides with the Koszul complex $K_{- \bullet}(\{\partial^m \tldmu_{\infty}(E_{ij}) \,|\, \substack{i,j = 1, \dots, N \\ m = 1, 2, \dots}\}, \JetBundle{\frX}(\tldU))$ associated with the sequence of sections $\{\partial^{m} \tldmu_{\infty}(E_{ij}) \,|\, \substack{i,j = 1, \dots, N \\ m = 1, 2, \dots}\}$ on $\JetBundle{\frX}(\tldU)$.

\begin{lemma}\label{lemma:regular-seq}
For any open subset $\tldU \subset \frX$, the sequence $\{\partial^m \tldmu_{\infty}(E_{ij}) \,|\, \substack{i,j = 1, \dots, N \\ m = 1, 2, \dots}\}$ is a regular sequence on $\JetBundle{\frX}(\tldU)$.
\end{lemma}
\begin{proof}
Since the moment map $\mu \colon T^* V \longrightarrow \frg^*$ is a flat morphism, the associated algebra homomorphism $\mu^*$ maps $S(\frg)$ injectively into $\C[T^* V] = \calO_{\frX}(\frX)$. Thus, $\{ \mu^*(E_{ij}) \,|\, i, j = 1, \dots, N \}$ is a regular sequence on $\C[T^* V]$, and $\mu^{-1}(0)$ is a complete intersection. By Proposition~1.4 in \cite{Mustata01}, $J_n \mu^{-1}(0)$ is also a locally complete intersection for any $n \in \Z_{\ge 0}$. By comparing the dimensions of the manifolds $\Dim_{\C} J_{n} \frX = \Dim_{\C} J_{n} M + 2 \Dim_{\C} J_{n} G$, we see that $\Dim_{\C} J_n \frX = \Dim_{\C} J_n \mu^{-1}(0) + \#\{ \partial^m \mu^*(E_{ij}) \,|\, \substack{i,j = 1, \dots, N \\ m = 0, 1, \dots, n}\}$. This implies that $\{ \partial^m \mu^*(E_{ij}) \,|\, \substack{i,j = 1, \dots, N \\ m = 0, 1, \dots, n}\}$ is again a regular sequence on $\C[J_n T^* V]$ for any $n \in \Z_{\ge 0}$. Therefore, $\{ \partial^m \mu^*(E_{ij}) \,|\, \substack{i,j = 1, \dots, N \\ m = 0, 1, \dots}\}$ is a regular sequence on $\C[J_{\infty} T^* V]$.

We arrange the sequence as $r_s \coloneqq \partial^{m_s} \mu^*(E_{i_s j_s})$ for $s = 1, 2, \dots$ and define $n_s = \partial^{m_s}\tldmu_{\infty}(E_{i_s j_s}) - r_{s} = \partial^{m_s} (\psi_{i_s (-1)} \phi_{j_s})$. Now, assume for the sake of contradiction that the sequence $\{ r_s + n_s \,|\, s = 1, 2, \dots \}$ is not regular on $\JetBundle{\frX}(\frX) = \C[J_{\infty} T^* V] \otimes \Lambda(\psi_{i (-n)}, \phi_{i (-n)} \,|\, \substack{i = 1, \dots, N \\ n = 1, 2, \dots})$. Then, there is an $s \in \Z_{>0}$ such that $r_s + n_s$ is a zero-divisor on $\JetBundle{\frX}(\frX) / (r_1 + n_1, \dots, r_{s-1} + n_{s-1})$. This implies that there exist $a_1,\dots,a_s \in \C[J_{\infty} T^* V]$ and $b_1,\dots,b_s \in \Lambda(\psi_{i (-n)}, \phi_{i (-n)} \,|\, \substack{i = 1, \dots, N \\ n = 1, 2, \dots})$ such that $\sum_{i=1}^{s} (a_s + b_s)(r_s + n_s) = 0$ in $\JetBundle{\frX}(\frX)$. Taking the component $\C[J_{\infty} T^* V] \otimes 1 \subset \JetBundle{\frX}(\frX)$, the equation implies that $\sum_{i=1}^{s} a_s r_s = 0$ in $\C[J_{\infty} T^* V]$, contradicting the fact that $\{ r_s \,|\, s = 1, 2, \dots \}$ is a regular sequence on $\C[J_{\infty} T^* V]$. Thus $\{ \partial^{n_s} \tldmu_{\infty}(E_{i_s j_s}) = r_s + n_s \,|\, s = 1, 2, \dots \}$ is a regular sequence on $\JetBundle{\frX}(\frX)$.
\end{proof}

For an open subset $\tldU \subset \frX$, let $\tau_{\ge \bullet} \tldC_{\VPA}(\tldfrU)$ be the column filtration associated with the double complex $(\tldC_{\VPA}(\tldU), d^{+}_{\VPA}, d^{-}_{\VPA})$, i.e.\
\[
\tau_{\ge p} \tldC_{\VPA}(\tldU) = \prod_{k \ge p, q \le 0} \tldC^{k,q}_{\VPA}(\tldU)
\]
for $p \in \Z$. We consider the spectral sequence $\dE_r^{p,q}(\tldfrU)$ associated with this column filtration. Then the zeroth term of the spectral sequence coincides with the tensor product of the Koszul complex and the exterior algebra
\[
\dE_0^{p,q}(\tldU) = K_{-q}(\{\partial^m \tldmu_{\infty}(E_{ij}) \,|\, \substack{i,j = 1, \dots, N \\ m = 1, 2, \dots}\}, \JetBundle{\frX}(\tldU)) \otimes \Lambda^p(\Psi_{ij (-n)} \,|\, \substack{i, j = 1, \dots, N \\ n = 1, 2, \dots})
\]
as discussed above. Thus, the first term is the tensor product of the Koszul homology and the exterior algebra
\[
\dE^{p,q}_1 = H_{-q}^{\Koszul}(\{\partial^m \tldmu_{\infty}(E_{ij}) \,|\, \substack{i,j = 1, \dots, N \\ m = 1, 2, \dots}\}, \JetBundle{\frX}(\tldU)) \otimes \Lambda^p(\Psi_{ij (-n)} \,|\, \substack{i, j = 1, \dots, N \\ n = 1, 2, \dots}).
\]
Then, by \autoref{lemma:regular-seq} above,
\begin{equation}\label{eq:10}
\dE^{p,q}_{1} =
\begin{cases}
\JetBundle{\frX}(\tldU) / \sum_{i,j,m} \JetBundle{\frX}(\tldU) \partial^m \tldmu_{\infty}(E_{ij}) \otimes \Lambda^{p}(\Psi_{ij (-n)} \,|\, \substack{i, j = 1, \dots, N \\ n = 1, 2, \dots}), & \!\!\! q = 0, \\
0, & \!\!\! q \ne 0.
\end{cases}
\end{equation}
The coboundary operator $d_1 = d^{+}_{\VPA}\colon \dE^{p,q}_{1} \longrightarrow \dE^{p+1, q}_{1}$ is given by \eqref{eq:d-VPA-plus}, and thus $(\dE^{p,q}_{1}, d_1 = d^{+}_{\VPA})$ coincides with the Chevalley complex \cite{CarEil56,Wei94} of the Lie algebra $\frg[t]$ with coefficients in the Koszul homology.

\begin{lemma}\label{lemma:conv-spec-double-cpx}
For any open subset $\tldU \subset \frX$, the spectral sequence $\dE_r^{p,q}(\tldU)$ converges to the total cohomology, i.e.\
$\dE_r^{p,q}(\tldU) \Longrightarrow H^{p+q}(\tldC_{\VPA}(\tldU), d_{\VPA})$.
\end{lemma}
\begin{proof}
Note that $\tldC^{p,q}_{\VPA} = 0$ unless $p \ge 0$ and $q \le 0$. To prove the asserted convergence, we construct subcomplexes that are bounded both from above and below. For $m \in \Z_{\ge 0}$, set $(\tldC_{\VPA})_m(\tldU) = \partial^m R$, where $R \coloneqq \calO_{\frX}(\tldU) \otimes \Lambda(T^* \C^N) \otimes \Lambda(T^* \frg)$ is regarded as a subalgebra of $\jet{\infty}{R}$ by the canonical embedding $j \colon R \hookrightarrow \jet{\infty}{R}$. By \eqref{eq:d-VPA-plus} and \eqref{eq:d-VPA-minus}, $d_{\VPA}$ preserves the subspace $(\tldC_{\VPA})_0(\tldU)$. Since $d_{\VPA}$ commutes with the translation operator $\partial$, $d_{\VPA}$ also preserves $(\tldC_{\VPA})_m(\tldU)$ for any $m \in \Z_{\ge 0}$. Therefore, $((\tldC_{\VPA})_m(\tldU), d^+_{\VPA}, d^-_{\VPA})$ is a double subcomplex of $(\tldC_{\VPA}(\tldU), d^+_{\VPA}, d^-_{\VPA})$. Consider the spectral sequence $(\dE_r^{p,q})_m(\tldU)$ associated with the double complex $((\tldC_{\VPA})_m(\tldU), d^+_{\VPA}, d^-_{\VPA})$. Since $(\tldC_{\VPA})_m(\tldU)$ is bounded, the spectral sequence $(\dE_r^{p,q})_m(\tldU)$ converges. Hence, so does the spectral sequence $\dE_r^{p,q}(\tldU)$.
\end{proof}

As a consequence, we obtain the following lemma.
\begin{lemma}\label{lemma:vanish-neg-VPA}
Let $\tldU \subset \frX$ be an open subset.
\begin{enumerate}
\item If $\tldU \cap \mu^{-1}(0) = \emptyset$, then $H^{n}(\tldC_{\VPA}^{\bullet}(\tldU), d_{\VPA}) = 0$ for all $n \in \Z$.
\item The cohomology $H^{n}(\tldC_{\VPA}^{\bullet}(\tldU), d_{\VPA})$ vanishes if $n < 0$.
\end{enumerate}
\end{lemma}
\begin{proof}
Since $\JetBundle{\frX}(\tldU) / \sum_{i,j,m} \JetBundle{\frX}(\tldU) \partial^m \tldmu_{\infty}(E_{ij}) = 0$ if $\tldU \cap \mu^{-1}(0) = \emptyset$ in \eqref{eq:10}, it follows that $\dE^{p,q}_{1} = 0$ for any $p,q \in \Z$. Thus, $H^{p+q}(\tldC_{\VPA}^{\bullet}(\tldU), d_{\VPA}) = 0$ for any $p,q \in \Z$ by \autoref{lemma:conv-spec-double-cpx}, proving the first assertion.

By \eqref{eq:10}, $\dE^{p,q}_{1} = 0$ unless $p \in \Z_{\ge 0}$ and $q = 0$. Thus, by \autoref{lemma:conv-spec-double-cpx}, the cohomology $H^{p+q}(\tldC_{\VPA}^{\bullet}(\tldU), d_{\VPA})$ vanishes unless $p+q \ge 0$.
\end{proof}


\subsection{Cohomology of the Poisson BRST Reduction}
\label{sec:cohom-Poisson-BRST}

In the following, we compute the cohomology of the sheaf $(C_{\VPA}, d_{\VPA})$ of vertex Poisson algebra complexes from above.

\medskip

Recall that the first term of the spectral sequence $(\dE^{p,q}_{1}, d_1 = d^{+}_{\VPA})$ coincides with the Chevalley complex of the Lie algebra $\frg[t]$ with coefficients in the Koszul homology. Therefore, the second term $\dE_{2}^{p,q}$ is isomorphic to the Lie algebra cohomology, and thus
\begin{align}\label{eq:13}
\dE_2^{p,q}(\tldU) &\simeq H^{p}(H^{q}(\tldC_{\VPA}(\tldU), d^{-}_{\VPA}), d^{+}_{\VPA}) \\
&\simeq H^{p}(\frg, H^{\Koszul}_{-q}(\{\partial^{m} \tldmu_{\infty}(E_{ij}) \}_{i,j,m}, \tldcalO_{\frX}(\tldU))) \nonumber \\
&\simeq
\begin{cases}
H^{p}(\frg, \JetBundle{\frX}(\tldU) / \sum_{i,j,m} \JetBundle{\frX}(\tldU) \partial^m \tldmu_{\infty}(E_{ij})), & q = 0, \\
0, & q \ne 0. \nonumber
\end{cases}
\end{align}
Recall the affine open covering $\frX = \bigcup_{\lambda \vdash N} \tldU_{\lambda}$ in terms of partitions $\lambda$ of $N$ introduced in \autoref{sec:local-trivial}. By \autoref{prop:local-triv},
\begin{align}\label{eq:18}
\tldcalO_{\frX}(\tldU_{\lambda}) &= A_{\lambda} \otimes \Lambda(\{\psi : X^i Y^j \gamma \}_{\lambda}, \phi X^i Y^j \gamma \,|\, (i,j) \in \lambda) \\
&\quad \otimes \C[(B_{\lambda}^{\pm})_{ij} \,|\, i, j = 1, \dots, N] \otimes \C[\tldmu(E_{ij}) \,|\, i, j = 1, \dots, N], \nonumber
\end{align}
where $A_{\lambda}$ is the subalgebra of $\tldcalO_{\frX}(\tldU_{\lambda})$ generated by the $G$-invariant bosonic sections defined in \autoref{sec:local-trivial}. Note that the Lie group $G$ acts trivially on the subalgebra $A_{\lambda} \otimes \Lambda(\{\psi : X^i Y^j \gamma \}_{\lambda}, \phi X^i Y^j \gamma \,|\, (i,j) \in \lambda)$, while $\C[(B_{\lambda}^{\pm})_{ij} \,|\, i, j = 1, \dots, N] \otimes \C[\tldmu(E_{ij}) \,|\, i, j = 1, \dots, N]$ is as $G$-module isomorphic to the coordinate ring $\C[T^* G]$ of the cotangent bundle $T^* G \simeq G \times \frg^*$. The $\infty$-jet bundle $\JetBundle{\frX}(\tldU_{\lambda}) = \jet{\infty}{\tldcalO_{\frX}(\tldU_{\lambda})}$ is thus the tensor product $\JetBundle{\frX}(\tldU_{\lambda}) \simeq \JetBundle{M}(U_{\lambda}) \otimes \C[\jet{\infty}{T^* G}]$ of the $\jet{\infty}{G}$-invariant subalgebra $\JetBundle{M}(U_{\lambda}) \!=\! \jet{\infty}{(A_{\lambda} \otimes \Lambda(\{\psi : X^i Y^j\}_{\lambda}, (\phi X^i Y^j \gamma) \,|\, (i, j) \in \lambda))}$ and the subalgebra $\C[\jet{\infty}{T^* G}] \simeq \C[\jet{\infty}{G}] \otimes \C[\jet{\infty}{\frg^{*}}]$. Therefore, by \eqref{eq:13}, $\dE_2^{p,0} \simeq \JetBundle{M}(U_{\lambda}) \otimes H^p(\frg[t], \C[\jet{\infty}{G}])$.

The Lie algebra cohomology $H^p(\frg[t], \C[\jet{\infty}{G}])$ is naturally isomorphic to the de Rham cohomology $H^p_{\DR}(G, \C)$. Hence,
\begin{equation}\label{eq:17}
\dE^{p,q}_2(\tldU_{\lambda}) \simeq
\begin{cases}
\JetBundle{M}(U_{\lambda}) \otimes H^p_{\DR}(G, \C), & q = 0, \\
0, & q \ne 0.
\end{cases}
\end{equation}
Then, the spectral sequence $\dE^{p,q}_r$ collapses at the second term and it converges by \autoref{lemma:conv-spec-double-cpx}, and thus we obtain:
\begin{proposition}\label{prop:VPA-BRST-cohom}
For a partition $\lambda \vdash N$ and $n\in\Z$, the vertex Poisson algebra cohomology satisfies
\[
H^n(\tldC_{\VPA}(\tldU_{\lambda}), d_{\VPA}) \simeq \JetBundle{M}(U_{\lambda}) \otimes H^n_{\DR}(G, \C).
\]
The superalgebra
$\JetBundle{M}(U_{\lambda}) = \jet{\infty}{(A_{\lambda} \otimes \Lambda(\{\psi : X^i Y^j\}_{\lambda}, (\phi X^i Y^j \gamma) \,|\, (i, j) \in \lambda))}$ is naturally isomorphic to the commutative superalgebra
\[
\jet{\infty}{\calO_M(U_{\lambda})} \otimes \Lambda\bigl(\{\psi : X^i Y^j\}_{\lambda, (-n)}, (\phi X^i Y^j \gamma)_{(-n)} \,\bigm|\, \substack{ (i, j) \in \lambda \\ n = 1, 2, \dots }\bigr).
\]
\end{proposition}

Now, we determine the cohomology of the subcomplex $(C_{\VPA}, d_{\VPA})$.
By the above discussion, for a partition $\lambda \vdash N$,
\[
\dE^{p,q}_{1}(\tldU_{\lambda}) =
\JetBundle{M}(U_{\lambda}) \otimes \C[\jet{\infty}{G}] \otimes
\Lambda^{p}(\Psi_{ij (-n)} \,|\, \substack{i, j = 1, \dots, N \\ n = 1, 2, \dots})
\]
if $q=0$, and $\dE^{p,q}_1 = 0$ otherwise. It is a Lie algebra cohomology complex associated with the action of $\frg[t]$ on $\JetBundle{M}(U_{\lambda}) \otimes \C[\jet{\infty}{G}]$, as discussed above. We consider the Hochschild-Serre spectral sequence associated with the Lie subalgebra $\frg \subset \frg[t]$. We refer the reader to \cite{HS53} for details of the construction of the spectral sequence and its fundamental properties. Define a filtration $\HSF_{\bullet} \dE_1^{n,0}$ by
\begin{align*}
&\HSF_p \dE_1^{n,0}\\
&= \Bigl\{\sum_{\underline{i}, \underline{j}, \underline{m}} a_{\underline{i}, \underline{j}, \underline{m}} \Psi_{i_1 j_1 (-m_1)} \dots \Psi_{i_n j_n (-m_n)} \,\Bigm|\, a_{\underline{i}, \underline{j}, \underline{m}} = 0 \text{ if } \#\{k \,|\, m_k = 1\} > n - p \Bigr\}
\end{align*}
for $p\in\Z$. Let $\HSE_r^{p,q}$ be the spectral sequence associated with this filtration. The zeroth term $\HSE_0^{p,q} = \HSGr_p \dE_1^{p+q,0}$ is the Lie algebra cohomology complex associated with the action of the Lie subalgebra $\frg$. Indeed, since $G$ acts on $\jet{\infty} G$ freely, the complex $E_0^{p,q}$ factorises as
\begin{align*}
\HSE_0^{p,q} &= \bigl(\JetBundle{M}(U_{\lambda}) \otimes \C[\jet{\infty}{G}]^G \otimes \Lambda^{p}(\Psi_{ij (-n)} \,|\, \substack{i, j = 1, \dots, N \\ n = 2, \dots}) \bigr) \\
&\quad \otimes \bigl(\C[G] \otimes \Lambda^{q}(\Psi_{ij (-1)} \,|\, i, j = 1, \dots, N) \bigr),
\end{align*}
where the second factor is the Lie algebra cohomology complex of the Lie subalgebra $\frg$ with coefficients in $\C[G]$. Then, the first term $\HSE_1^{p,q}$ is obtained by de Rham cohomology $H^q_{\DR}(G, \C)$ as
\begin{equation}\label{eq:19}
\HSE_1^{p,q} \simeq \left(\JetBundle{M}(U_{\lambda}) \otimes \C[\jet{\infty}{G}]^G \otimes \Lambda^{p}(\Psi_{ij (-n)} \,|\, \substack{i, j = 1, \dots, N \\ n = 2, \dots}) \right) \otimes H_{\DR}^{q}(G, \C).
\end{equation}
On the other hand, it follows from \eqref{eq:18} that
\begin{align*}
C^n_{\VPA}(\tldU_{\lambda}) &= \bigoplus_{p+q=n} \JetBundle{M}(U_{\lambda}) \otimes \C[\jet{\infty}{G}]^G \otimes \C[\partial^m \tldmu_{\infty}(E_{ij}) \,|\, \substack{i,j = 1, \dots, N \\ m = 1, 2, \dots}] \\
&\quad \otimes \Lambda^{p}(\Psi_{ij (-m)} \,|\, \substack{i, j = 1, \dots, N \\ m = 2, 3 \dots}) \otimes \Lambda^{q}(\Phi_{ij (-m)} \,|\, \substack{i, j = 1, \dots, N \\ m = 1, 2, \dots}).
\end{align*}
Thus, the cohomology $H^{q}(C^{p, \bullet}_{\VPA}(\tldU_{\lambda}), d^{-}_{\VPA})$ with respect to the coboundary operator $d^{-}_{\VPA}$ coincides with the first factor of \eqref{eq:19} if $q = 0$ and vanishes otherwise. As a consequence,
\[
\HSE_1^{p,q} \simeq H^0(C^{p,\bullet}_{\VPA}(\tldU_{\lambda}), d^{-}_{\VPA}) \otimes H^{q}_{\DR}(G, \C).
\]
Considering the coboundary operator $d^{+}_1 \colon \HSE_1^{p,q} \longrightarrow \HSE_1^{p+1, q}$ induced from $d^{+}_{\VPA}$, we obtain
\[
\HSE_2^{p,q} \simeq H^p(H^0(C_{\VPA}(\tldU_{\lambda}), d^{-}_{\VPA}), d^{+}_{\VPA}) \otimes H_{\DR}^{q}(G, \C).
\]
By the same arguments as for \autoref{prop:VPA-BRST-cohom}, we obtain the isomorphism
\[
H^p(C_{\VPA}(\tldU_{\lambda}), d_{\VPA}) \simeq H^p(H^0(C_{\VPA}(\tldU_{\lambda}), d^{-}_{\VPA}).
\]
In particular, $H^0(C_{\VPA}(\tldU_{\lambda}), d_{\VPA}) \simeq \JetBundle{M}(U_{\lambda})$ by \eqref{eq:19}. The above spectral sequence collapses at the second term, and it converges. Therefore,
\[
H^n(\tldC_{\VPA}(\tldU_{\lambda}), d_{\VPA}) \simeq \bigoplus_{p+q=n} H^p(C_{\VPA}(\tldU_{\lambda}), d_{\VPA}) \otimes H^q_{\DR}(G, \C).
\]
By the isomorphism $H^0(C_{\VPA}(\tldU_{\lambda}), d_{\VPA}) \simeq \JetBundle{M}(U_{\lambda})$ and \autoref{prop:VPA-BRST-cohom}, the above isomorphism implies that $H^p(C_{\VPA}(\tldU_{\lambda}), d_{\VPA}) = 0$ unless $p=0$.
\begin{proposition}\label{prop:relative-VPA-BRST}
For a partition $\lambda \vdash N$ and $n\in\Z$, the vertex Poisson algebra cohomology satisfies
\[
H^n(C_{\VPA}(\tldU_{\lambda}), d_{\VPA}) \simeq
\begin{cases}
\JetBundle{M}(U_{\lambda}), & n=0, \\
0, & n \ne 0.
\end{cases}
\]
\end{proposition}


\subsection{Vanishing Theorem}
\label{sec:vanish-brst}

Finally, we use the results of the preceding sections in order to prove a vanishing (or no-ghost) theorem for the BRST cohomology sheaf $\calH_{\VA}^{\infty/2 + \bullet}(\frg, \tldcalD^\ch_{\frX, \hbar})$, which we then define as the sheaf $\tldcalD^\ch_{M,\hbar}$.

\medskip

Recall the $\hbar$-adic filtration $\hF_{\bullet} C_{\VA}$. Let $\hE_{r}^{p,q}$ be the spectral sequence associated with this filtration. The zeroth term $\hE_0^{p,q}$ is the associated graded complex $\hGr_p C_{\VA}^{p+q} \simeq \hbar^p C_{\VPA}^{p+q}$ with the coboundary operator $d_{\VPA}$.

\begin{lemma}\label{lemma:conv-hE}
For an open subset $\tldU \subset \frX$, the spectral sequence $\hE^{p,q}_{r}(\tldU)$ converges to $\hGr_{p} H^{p+q}(C_{\VA}(\tldU), d_{\VA})$.
\end{lemma}
\begin{proof}
Because the filtration $\hF_{\bullet} C_{\VA}(\tldU)$ is bounded from above and complete, the spectral sequence $\hE^{p,q}_{r}(\tldU)$ converges by the complete convergence theorem (see Theorem~5.5.10 in \cite{Wei94}).
\end{proof}

First, consider the case where $\tldU \cap \mu^{-1}(0) = \emptyset$. By \autoref{lemma:vanish-neg-VPA}~(1) it follows that $\hE^{p,q}_{1}(\tldU) \simeq \hbar^p H^{p+q}(C_{\VPA}(\tldU), d_{\VPA}) = 0$. The spectral sequence $\hE^{p,q}_{r}$ collapses at the first term, and hence $\hE^{p,q}_{r} = 0$ for any $p,q \in \Z$ and $r \in \Z_{>0}$. Then \autoref{lemma:conv-hE} implies:
\begin{lemma}\label{lemma:support-BRST-cohom}
For an open subset $\tldU \subset \frX$ such that $\tldU \cap \mu^{-1}(0) = \emptyset$, the cohomology $H^n(C_{\VA}(\tldU), d_{\VA}) = 0$ for all $n\in\Z$.
\end{lemma}
As remarked above, this lemma completes the definition of the BRST cohomology sheaf $\calH^{\infty/2 + \bullet}_{\VA}(\frg, \tldcalD^\ch_{\frX, \hbar})$ (see \autoref{def:BRST-sheaf} and the preceding discussion).

\medskip

For a partition $\lambda \vdash N$, \autoref{prop:relative-VPA-BRST} yields the isomorphism
\begin{equation}\label{eq:hE1-local-isom}
\hE_1^{p,q}(\tldU_{\lambda}) \simeq
\begin{cases}
\hbar^p \JetBundle{M}(U_{\lambda}), & p + q = 0, \\
0, & p + q \ne 0.
\end{cases}
\end{equation}
This implies that the coboundary operator $d_1 \colon \hE^{p,q}_{1}(\tldU_{\lambda}) \longrightarrow \hE^{p+1,q}_{1}(\tldU_{\lambda})$ on the first term is identically zero. Thus, the spectral sequence $\hE^{p,q}_{r}$ collapses at the first term. By \autoref{lemma:conv-hE},
\[
H^{\infty/2+n}_{\VA}(\frg, \tldcalD^\ch_{\frX, \hbar}(\tldU_{\lambda})) = H^n(C_{\VA}(\tldU_{\lambda}), d_{\VA}) = 0
\]
for $n \ne 0$. Therefore, we obtain the following vanishing theorem for BRST cohomology sheaf.
\begin{theorem}\label{thm:vanish-BRST}
The BRST cohomology sheaf $\calH^{\infty/2+n}_{\VA}(\frg, \tldcalD^\ch_{\frX, \hbar})$ vanishes for $n \neq 0$.
\end{theorem}

We come to the central definition of this text:
\begin{definition}\label{def:chiral-sheaf-Hilb}
We set $\tldcalD^\ch_{M, \hbar} \coloneqq \calH^{\infty/2+0}_{\VA}(\frg, \tldcalD^\ch_{\frX, \hbar}) = \calH^{\infty/2+\bullet}_{\VA}(\frg, \tldcalD^\ch_{\frX, \hbar})$. This is a sheaf of $\hbar$-adic vertex superalgebras on the Hilbert scheme $M$.
\end{definition}

The isomorphism in \eqref{eq:hE1-local-isom} together with \autoref{lemma:conv-hE} induces an isomorphism $\tldcalD^\ch_{M, \hbar}(U_{\lambda}) \simeq \JetBundle{M}(U_{\lambda})[[\hbar]]$ for any $\lambda \vdash N$. However, this isomorphism depends on the local trivialisation \eqref{eq:18} and does not induce an isomorphism of sheaves. The isomorphism only implies:
\begin{proposition}\label{prop:quantization}
There is an isomorphism of sheaves of vertex Poisson superalgebras over $M$,
\[
\tldcalD^\ch_{M, \hbar} / \hbar \tldcalD^\ch_{M, \hbar} \simeq \JetBundle{M}.
\]
We say that the sheaf of $\hbar$-adic vertex superalgebras $\tldcalD^\ch_{M, \hbar}$ quantises the sheaf of vertex Poisson superalgebras $\JetBundle{M}$.
\end{proposition}


\section{Vertex Superalgebra of Global Sections}
\label{sec:F-action}

In the previous section, we introduced the sheaf of $\hbar$-adic vertex superalgebras $\tldcalD^\ch_{M, \hbar}$ on the Hilbert scheme $M=\Hilb^N(\C^2)$. Now, we construct a vertex operator superalgebra $\sfV_{S_N}$ of central charge $c=-3N^2$ from the $\hbar$-adic vertex superalgebra $\tldcalD^\ch_{M, \hbar}(M)$ of global sections and study its conformal structure and associated variety.


\subsection{Equivariant Torus Action and Global Sections}
\label{sec:global-VA}

The resolution of singularities $M = \Hilb^N(\C^2) \longrightarrow M_0 = \C^{2N}/S_N$ is an example of a conical symplectic resolution in the sense of \cite{BLPW16a, BLPW16b}. This entails that $M$ is equipped with an action of the one-dimensional torus $\C^\times$ such that
\begin{enumerate}
\item the symplectic form $\omega$ satisfies $t^* \omega = t^n \omega$ for all $t \in \C^{\times}$ for some positive integer $n$ (here, $n=2$),
\item $\C^{\times}$ acts on the coordinate ring $\C[M]$ with only non-negative weights, and the trivial weight space $\C[M]^{\C^\times} = \C \cdot 1$ is one-dimensional.
\end{enumerate}
Geometrically, the torus action makes $M_0 = \C^{2N} / S_N$ into a cone and contracts $M_0$ to the cone point $o \in M_0$, the image of the origin of $\C^{2N}$. In \cite{BLPW16a, BLPW16b}, $\C^\times$-equivariant quantisations of conical symplectic resolutions are studied.

In the following, we consider the $\C^\times$-action on $M$ and show that it naturally lifts to an equivariant action on the BRST cohomology sheaf $\tldcalD^\ch_{M, \hbar}$. This allows us to define a vertex superalgebra $\sfV_{S_N}=[\tldcalD^\ch_{M, \hbar}(M)]^{\C^\times}$ from the global sections by taking the invariants under this torus action in the manner of \cite{BLPW16a, BLPW16b}.

\medskip

The conical structure of $M = \Hilb^N(\C^2)$ is given by the following action of the one-dimensional torus $\C^\times$: consider the action of $\C^\times$ on $\frX$ that induces an equivariant action on the structure sheaf $\calO_{\frX}$ such that the weights of the generators with respect to it are given by $\Swt(x_{ij}) = \Swt(y_{ij}) = \Swt(\gamma_i) = \Swt(\beta_i)= 1/2$ for $i,j=1,\dots,N$. Here, an element $x$ has (torus) weight $\Swt(x)=m$ if it is semi-invariant for the character $\C^\times\longrightarrow\C^\times$, $t\longmapsto t^m$, i.e.\ if $t^* x=t^m x$ for $t\in\C^\times$. Note that, with respect to this action, the Poisson bracket on $\calO_{\frX}$ is homogeneous of weight $-1$. Since the $\C^\times$-action commutes with the $G$-action, we obtain an induced $\C^\times$-action on $M$.

Moreover, there is the equivariant $\C^\times$-action on the sheaf $\tldcalD^\ch_{\frX, \hbar}$ over $\C$ with the torus weights of the generators given by
\begin{align*}
\Swt(x_{ij (-n)}) &= \Swt(y_{ij (-n)}) = \Swt(\gamma_{i (-n)}) = \Swt(\beta_{i (-n)}) = 1/2,\\
\Swt(\psi_{i (-n)}) &= \Swt(\phi_{i (-n)}) = 1/2,\qquad\Swt(\hbar) = 1
\end{align*}
for $i,j=1,\dots,N$ and $n \in \Z_{>0}$. Note that the operator product expansions of $\tldcalD^\ch_{\frX, \hbar}$ are homogeneous of weight $0$ with respect to the $\C^\times$-action. We extend this torus action to the BRST complexes $\tldC_{\VA}$ and $C_{\VA}$ by
\[
\Swt(\Psi_{ij (-n)}) = 0,\quad\Swt(\Phi_{ij (-n)}) = 1
\]
for $i,j=1,\dots,N$ and $n \in \Z_{>0}$. Then, the element $Q \in \tldC_{\VA}$ is homogeneous of weight $\Swt(Q) = 1$, and hence the coboundary operator $d_{\VA} = \hbar^{-1} Q_{(0)}$ is a homogeneous operator of weight $0$ on the complexes $\tldC_{\VA}$ and $C_{\VA}$. This implies that the BRST cohomology sheaf $\tldcalD^\ch_{M, \hbar} = \calH^{\infty/2+\bullet}_{\VA}(\frg, \tldcalD^\ch_{\frX, \hbar})$ is also equipped with the induced equivariant $\C^\times$-action over $M$. In particular, the space of global sections $\tldcalD^\ch_{M, \hbar}(M)$ is a $\C[[\hbar]]$-module with a $\C^\times$-action over $\C$.

Recall the affine open covering $\frX = \bigcup_{\lambda \vdash N} \tldU_{\lambda}$. For a partition $\lambda$, the open subset $\tldU_{\lambda}$ is closed under the $\C^\times$-action. Each column of the matrix $B_{\lambda} = (X^i Y^j \gamma)_{(i,j) \in \lambda}$ is homogeneous of weight $(i+j+1)/2$ and thus its determinant and minors are again homogeneous. This implies that the $(i,j)$-th row of its inverse $B_{\lambda}^{-1}$ is homogeneous of weight $-(i+j+1)/2$. Therefore, the complex $\tldC_{\VA}(\tldU_{\lambda})$ is generated by homogeneous sections and it can be decomposed into a direct product of weight spaces. The coboundary operator $d_{\VA}$ commutes with the $\C^\times$-action, and hence the BRST cohomology $H^{\bullet}(\tldC_{\VA}(\tldU_{\lambda}), d_{\VA})$ is also a direct product of weight spaces. This implies, in particular, the weight-space decomposition $\tldcalD^\ch_{M, \hbar}(U_{\lambda}) = \prod_{m} \tldcalD^\ch_{M, \hbar}(U_{\lambda})^{\C^\times\!, m}$.

Since the space of global sections $\tldcalD^\ch_{M, \hbar}(M)$ is the intersection of the $\tldcalD^\ch_{M, \hbar}(U_{\lambda})$ for all $\lambda \vdash N$, it can also be decomposed into the direct product of weight spaces
\[
\tldcalD^\ch_{M, \hbar}(M) = \prod_{m \ge 0} \tldcalD^\ch_{M, \hbar}(M)^{\C^\times\!, m}.
\]
We note that the weights $m \in \frac{1}{2} \Z_{\ge 0}$ of the global sections are non-negative and $\tldcalD^\ch_{M, \hbar}(M)^{\C^\times\!, 0} = \C \bfone$. Consider the subspace
\[
\tldcalD^\ch_{M, \hbar}(M)_\fin \coloneqq \bigoplus_{m \in \frac{1}{2} \Z_{\ge 0}} \tldcalD^\ch_{M, \hbar}(M)^{\C^\times\!, m}
\]
of the direct sum of weight spaces. This subspace is a $\C[\hbar]$-module since the weights of the global sections are non-negative and $\Swt(\hbar)=1$. Moreover, because the operator product expansions preserve the $\C^\times$-weight, they also preserve this subspace. Now we set
\begin{equation}\label{eq:sfW}
\sfV_{S_N} = \tldcalD^\ch_{M, \hbar}(M)_\fin \bigm/ (\hbar - 1),
\end{equation}
the quotient space by the ideal generated by $\hbar - 1$. It is a $\C$-vector space equipped with operator product expansions induced from the ones on $\tldcalD^\ch_{M, \hbar}(M)$. Since all the identities involving the vertex operators of $\tldcalD^\ch_{M, \hbar}(M)$ are also satisfied by the ones of $\sfV_{S_N}$, the $\C$-vector space $\sfV_{S_N}$ is a vertex superalgebra.
\begin{remark}
Alternatively, we can think of $\sfV_{S_N}$ as
\[
 [\tldcalD^\ch_{M, \hbar}(M)]^{\C^\times} \coloneqq \big(\tldcalD^\ch_{M, \hbar}(M)\otimes_{\C[[\hbar]]}\C((\sqrt{\hbar}))\big)^{\C^\times}=\bigoplus_{m \in \frac{1}{2} \Z_{\ge 0}} \sfV_{S_N}^{\C^\times\!, m}\hbar^{-m}
\]
(and then forget about $\hbar$), where taking the invariants under the natural extension of the $\C^\times$-action precludes the appearance of infinite sums (cf.\ \cite{KR08,AKM15}).
\end{remark}


\subsection{Associated Variety}
\label{sec:associated-variety}

In the following, we determine the associated variety of the vertex superalgebra $\sfV_{S_N}$ and, as a consequence, show that $\sfV_{S_N}$ is quasi-lisse.

\medskip

For a vertex superalgebra $V$ (or an $\hbar$-adic vertex superalgebra), consider the quotient vector space
\[
\Abar(V) \coloneqq V / C_2(V)\quad\text{with}\quad C_2(V) = \{ a_{(-2)} b \,|\, a, b \in V\}.
\]
The $(-1)$-product of $V$ induces a commutative and associative product on $\Abar(V)$ \cite{Z96}. Moreover, the $(0)$-product of $V$ defines a Poisson bracket on $\Abar(V)$ by $\{a, b\} = a_{(0)} b$ (or $\{a, b\} = \hbar^{-1} a_{(0)} b$, respectively) modulo $V_{(-2)} V$ for $a,b \in V$. Thus, $\Abar(V)$ is a Poisson superalgebra over $\C$ (or $\C[[\hbar]]$, respectively), and is called the $C_2$-Poisson algebra of $V$. The affine algebraic variety $\Spec(\Abar(V)_{\red})$ associated with $\Abar(V)_{\red}$ is called the associated variety of $V$ \cite{Arakawa12}, where $\Abar(V)_{\red} = \Abar(V) / N$ is the quotient algebra by the nilradical $N$ of $\Abar(V)$.

For any open subset $U$ of the Hilbert scheme $M$, the exact sequence
\[
0 \longrightarrow \hbar \tldcalD^\ch_{M, \hbar} \longrightarrow \tldcalD^\ch_{M, \hbar} \longrightarrow \JetBundle{M} \longrightarrow 0
\]
induces an injective homomorphism $\tldcalD^\ch_{M, \hbar}(U) / \hbar \tldcalD^\ch_{M, \hbar}(U) \hookrightarrow \JetBundle{M}(U)$. Then, the composition of the above homomorphism $\tldcalD^\ch_{M, \hbar}(U) \longrightarrow \JetBundle{M}(U)$ and the canonical projection $\JetBundle{M}(U) \longrightarrow \Abar(\JetBundle{M}(U))$ induces an injective homomorphism
\begin{equation}\label{eq:3}
\tldcalD^\ch_{M, \hbar}(U) \bigm/ \bigl(C_2(\tldcalD^\ch_{M, \hbar}(U)) + \hbar \tldcalD^\ch_{M, \hbar}(U)\bigr)
\hookrightarrow \Abar(\JetBundle{M}(U)).
\end{equation}

\begin{lemma}\label{lemma:1}
For $U\subset M$ open, the homomorphism \eqref{eq:3} induces an injective homomorphism of Poisson superalgebras $\Abar(\tldcalD^\ch_{M, \hbar}(U)) / \hbar \Abar(\tldcalD^\ch_{M, \hbar}(U)) \hookrightarrow \tldcalO_{M}(U)$.
\end{lemma}
\begin{proof}
By the isomorphism theorem,
\begin{align*}
\Abar(\tldcalD^\ch_{M, \hbar}(U)) / \hbar \Abar(\tldcalD^\ch_{M, \hbar}(U))
&= \frac{\tldcalD^\ch_{M, \hbar}(U) / C_2(\tldcalD^\ch_{M, \hbar}(U))}{ \hbar \tldcalD^\ch_{M, \hbar}(U) / (C_2(\tldcalD^\ch_{M, \hbar}(U)) \cap \hbar \tldcalD^\ch_{M, \hbar}(U))} \\
&\simeq \tldcalD^\ch_{M, \hbar}(U) / (C_2(\tldcalD^\ch_{M, \hbar}(U)) + \hbar \tldcalD^\ch_{M, \hbar}(U)).
\end{align*}
The asserted homomorphism is then given by the composition of this isomorphism and the one in \eqref{eq:3} mapping to $\Abar(\JetBundle{M}(U)) \simeq \tldcalO_{M}(U)$.
\end{proof}

\autoref{lemma:1} implies the following proposition about the $C_2$-Poisson algebras of $\tldcalD^\ch_{M, \hbar}(M)$ and $\sfV_{S_N}$.
\begin{proposition}\label{prop:C2-subalg}
The $C_2$-Poisson algebra $\Abar(\tldcalD^\ch_{M, \hbar}(M))$ of $\tldcalD^\ch_{M, \hbar}(M)$ is a subalgebra of $\tldcalO_M(M)[[\hbar]]$, and $\Abar(\sfV_{S_N})$ is a subalgebra of $\tldcalO_{M}(M)$.
\end{proposition}

\begin{lemma}
For any $m \in \Z_{>0}$, $\Tr(X^m)$ and $\Tr(Y^m)$ are cocycles in $C_{\VA}(\frX)$.
\end{lemma}
\begin{proof}
We shall see in the proof of \autoref{prop:gen-N4-cocyc} that the operator product expansion $\mu_\ch(E_{ij})(z) \Tr(X^m)(w) \sim \mu_\ch(E_{ij})(z) \Tr(Y^m)(w) \sim 0$ holds for all $i,j=1,\dots,N$. Thus, we obtain $d_{\VA}(\Tr(X^m)) = d_{\VA}(\Tr(Y^m)) = 0$ and moreover $\tldE_{ij(0)} \Tr(X^m) = \tldE_{ij(0)} \Tr(Y^m) = 0$ for any $i,j=1,\dots,N$.
\end{proof}

By the above lemma, there are the elements $\Tr(X^m),\Tr(Y^m) \in \sfV_{S_N}$ and their images in the Poisson algebra $\Abar(\sfV_{S_N})_{\red} \subset \calO_{M}(M)$. Under the isomorphism $\calO_{M}(M) \simeq \C[\C^{2N}]^{S_N}$, the elements $\Tr(X^m)$ and $\Tr(Y^m)$ lie in $\C[x_1, \dots, x_N]^{S_N}$ and $\C[y_1, \dots, y_N]^{S_N}$, respectively. In fact, the sets $\{\Tr(X^m) \,|\, m = 1, \dots, N \}$ and $\{\Tr(Y^m) \,|\, m = 1, \dots, N \}$ generate $\C[x_1, \dots, x_N]^{S_N}$ and $\C[y_1, \dots, y_N]^{S_N}$, respectively, by fundamental properties of symmetric polynomials. The following lemma is also a well-known fact for diagonal invariant algebras.
\begin{lemma}[\cite{Wallach93}, Theorem 2.1]
As a Poisson algebra, $\C[\C^{2N}]^{S_N}$ is generated by the subalgebras
$\C[x_1, \dots, x_N]^{S_N}$ and $\C[y_1, \dots, y_N]^{S_N}$.
\end{lemma}

With this result, we conclude that the $C_2$-Poisson algebra $\Abar(\sfV_{S_N})_{\red}$ includes $\C[\C^{2N}]^{S_N} = \calO_M(M)$. Then, we obtain the following theorem as a consequence of the above arguments. Recall that the Hilbert scheme $M$ is a resolution of singularities of $M_0 \simeq \C^{2N} / S_N$.
\begin{theorem}\label{thm:assoc-var}
The associated variety of the vertex superalgebra $\sfV_{S_N}$ coincides with the symplectic quotient variety $\C^{2N} / S_N$. In particular, $\sfV_{S_N}$ is quasi-lisse.
\end{theorem}
We also recall that $\C^{2N}/S_N\cong\calM_{S_N}\times T^*\C$ as symplectic varieties, where $\calM_{S_N}$ is the canonical symplectic singularity associated with the complex reflection group~$S_N$ \cite{Bea00} (cf.\ \autoref{thm:W-prop}).


\subsection{Conformal Structure}
\label{sec:conformal}

We now equip the sheaf $\tldcalD^\ch_{M,\hbar}$ with a natural conformal structure of central charge $c=-3N^2$ that is inherited via the BRST construction. This also makes $\sfV_{S_N}$ into a vertex operator superalgebra of CFT-type of that central charge.

\medskip

A comment on notation: for an open subset $\tldU \subset \frX$, a section $f$ in $\tldcalO_{\frX}(\tldU) = \calO_{\frX}(\tldU) \otimes \Lambda(\psi_i, \phi_i \,|\, i=1, \dots, N)$ determines a section $\iota(f) \in \tldcalD^\ch_{\frX, \hbar}(\tldU)$ using the natural map $\iota$ defined in \eqref{eq:lift-map}. For the remainder of the text, we shall regard vectors and matrices with entries in $\tldcalO_{\frX}(\tldU)$ as ones with entries belonging to the $\hbar$-adic vertex superalgebra $\tldcalD^\ch_{\frX, \hbar}(\tldU)$. Moreover, we identify sections in $\tldcalD^\ch_{\frX, \hbar}(\tldU)$ with ones in the $\hbar$-adic vertex subalgebra $\tldcalD^\ch_{\frX, \hbar}(\tldU) \otimes \bfone \subset \tldC_{\VA}(\tldU) = \tldcalD^\ch_{\frX, \hbar}(\tldU) \hatotimes \Cl_{\hbar}(T^* \frg)$ of the BRST cochain complex $\tldC_{\VA}(\tldU)$. For example, the matrix product $\phi X \gamma$ means $\sum_{i,j=1}^{N} \iota(\phi_i x_{ij} \gamma_{j}) = \sum_{i,j=1}^{N} x_{ij (-1)} \gamma_{j (-1)} \phi_{i (-1)} \otimes \bfone \in \tldC_{\VA}(\frX)$.

\medskip

First, we endow the sheaf $\tldC_{\VA} = \tldcalD^\ch_{\frX, \hbar} \hatotimes \Cl_{\hbar}(T^* \frg)$ of $\hbar$-adic vertex superalgebras with a conformal structure (cf.\ \cite{DM06,MN99}) of central charge $c=-3N^2$ for which the free-field generators have weights
\begin{align*}
\wt(x_{ij})&=\wt(y_{ij})=\wt(\gamma_i)=\wt(\beta_i)=\wt(\psi_i)=\wt(\phi_i)=1/2,\\
\wt(\Psi_{ij})&=0,\qquad\wt(\Phi_{ij})=1
\end{align*}
for $i,j=1,\dots,N$. That is, we consider the global section
\begin{equation}\label{eq:free-field-conformal}
T\coloneqq \frac{\hbar}{2}\bigl(\Tr(\partial X Y) - \Tr(X \partial Y) + \beta \partial \gamma - \partial \beta \gamma + \partial \phi \psi - \phi \partial \psi\bigr) + \hbar\Tr(\partial \Psi {}^{t} \Phi)
\end{equation}
in $\tldC_{\VA}(\frX)$ satisfying the operator product expansions
\begin{align*}
T(z) T(w) &\sim \frac{-3N^2 \hbar^4 / 2}{(z-w)^4} + \frac{2 \hbar^2}{(z-w)^2} T(w) + \frac{\hbar^2}{z-w} \partial T(w),\\
T(z) A(w) &\sim \frac{\wt(A)\hbar^2}{(z-w)^2} A(w) + \frac{\hbar^2}{z-w} \partial A(w)
\end{align*}
for $A=x_{ij},y_{ij},\gamma_i,\beta_i,\psi_i,\phi_i,\Psi_{ij},\Phi_{ij}$.

We observe that $\tldE_{ij (0)} T = \Phi_{ij (0)} T = 0$ for all $i, j = 1, \dots, N$ so that, by definition, $T\in C_{\VA}(\frX)$. This means that $T$ also defines a conformal structure on the subalgebra $C_{\VA}(\frX)$ of $\tldC_{\VA}(\frX)$.

\begin{remark}
We recall the homomorphisms $\tldmu_\ch\colon V^{-2N}(\frg)_{\hbar}\longrightarrow\tldcalD^\ch_{\frX, \hbar}(\frX)$ and $J\colon V^{2N}(\frg)_{\hbar}\longrightarrow\Cl_{\hbar}(T^* \frg)$ of $\hbar$-adic vertex superalgebras from \autoref{sec:brst-construction}. Both affine vertex algebras can be equipped with the standard Sugawara conformal vector, which we call $T_{\frg}$ and $\widetilde{T}_{\frg}$, respectively. Then a straightforward computation shows that the conformal structure on $\tldC_{\VA} = \tldcalD^\ch_{\frX, \hbar} \hatotimes \Cl_{\hbar}(T^* \frg)$ defined by $T$ is compatible with the conformal structures defined by $\tldmu_\ch(T_{\frg})$ and $J(\widetilde{T}_{\frg})$ on the images of $\tldmu_\ch$ and $J$, respectively, i.e.\ $\tldmu_\ch(T_{\frg})$ and $J(\widetilde{T}_{\frg})$ both have $T_{(1)}$-weights $2$ and satisfy $T_{(2)}\tldmu_\ch(T_{\frg})=T_{(2)}J(\widetilde{T}_{\frg})=0$ (cf.\ \cite{FZ92}).
\end{remark}

The following property is crucial in order to endow also the BRST cohomology sheaf $\tldcalD^\ch_{M, \hbar}$ with a conformal structure inherited from $T$.
\begin{proposition}
The global section $T$ satisfies $T\in\Ker d_{\VA}$, i.e.\ it is a cocycle in the cochain complex $(\tldC_{\VA}(\frX),d_{\VA})$, and in the subcomplex $(C_{\VA}(\frX),d_{\VA})$.
\end{proposition}
This will follow from \autoref{prop:total-conf-vec} below, where we show that $T$ coincides with a certain vector $\TN + \TTr + \TSF \in \Ker d_{\VA}$ modulo $\Im d_{\VA}$.

The proposition implies that $T$ induces a conformal structure of central charge $c=-3N^2$ on the cohomology associated with the complex $(\tldC_{\VA}(\frX),d_{\VA})$. The same is true for the (relative) cohomology $\tldcalD^\ch_{M,\hbar}(M)$ associated with the subcomplex $(C_{\VA}(\frX), d_{\VA})$.

\begin{remark}
A direct calculation reveals that any conformal vector of $C_{\VA}(\frX)$ lying in the kernel of $d_{\VA}$ must be of the form
\begin{align*}
T_{k_1, k_2} &= k_1 \mathrm{Tr}(\partial X Y) + (k_1-1) \Tr(X \partial Y)\\
&\quad+ (k_2-1) (\partial \beta \gamma - \partial \phi \psi) + k_2 (\beta \partial \gamma - \phi \partial \psi) + \Tr(\partial \Psi {}^t \Phi)
\end{align*}
for some $k_1,k_2\in\C$. Moreover, $T_{k_1, k_2} - T_{k_1, 1/2} \in \Im d_{\VA}$ for all $k_1$, $k_2$, and hence we obtain a one-parameter family $T_{k_1, 1/2}$, $k_1 \in \C$, of conformal vectors of $\tldcalD^\ch_{M, \hbar}(M)$. The conformal vector $T = T_{1/2, 1/2}$ is the unique choice that gives the same conformal structure as one discussed in \cite{BMR19} (see \autoref{sec:N4-SCA} below).
\end{remark}

By definition of $C_{\VA}$, the $\hbar$-adic vertex superalgebra $C_{\VA}(\frX)$ is generated by fields whose weights are in $\frac{1}{2}\Z_{>0}$. This implies that $C_{\VA}(\frX)$ is an $\hbar$-adic vertex operator superalgebra of CFT-type. This also holds for $\tldcalD^\ch_{M,\hbar}(M)$:
\begin{proposition}
The vector $T\bmod\Im d_{\VA}$ is a conformal vector of the $\hbar$-adic vertex operator superalgebra $\tldcalD^\ch_{M,\hbar}(M)$ of CFT-type of central charge $c=-3N^2$.
\end{proposition}
In particular, the global sections satisfy a $\frac{1}{2} \Z_{\ge 0}$-graded decomposition into weight spaces for $T_{(1)}$,
\[
\tldcalD^{\ch}_{M, \hbar}(M) = \prod_{n \in \frac{1}{2}\Z_{\ge 0}} \tldcalD^{\ch}_{M, \hbar}(M)_n, \quad
\tldcalD^{\ch}_{M, \hbar}(M)_n = \{ a \in \tldcalD^{\ch}_{M, \hbar}(M) \,|\, T_{(1)} a = n a \},
\]
with $\tldcalD^{\ch}_{M, \hbar}(M)_0 = \C[[\hbar]] \bfone$.

Finally, the reduction of $\tldcalD^\ch_{M,\hbar}(M)$ to $\sfV_{S_N}$ endows the latter with the structure of a vertex operator superalgebra:
\begin{corollary}\label{prop:CFT-type}
$\sfV_{S_N}=[\tldcalD^\ch_{M,\hbar}(M)]^{\C^\times}$ is a vertex operator superalgebra of CFT-type of central charge $c=-3N^2$.
\end{corollary}
Indeed, the following $\frac{1}{2} \Z_{\ge 0}$-graded decomposition into finite-dimensional weight spaces for $T_{(1)}$ holds:
\[
\sfV_{S_N} = \bigoplus_{n \in \frac{1}{2}\Z_{\ge 0}} \sfV_{S_N, n}, \qquad
\sfV_{S_N, n} = \{ a \in \sfV_{S_N} \,|\, T_{(1)} a = n a \},
\]
with $\sfV_{S_N, 0} = \C \bfone$.

We point out that the vertex operator superalgebra $\sfV_{S_N}$ does not have ``correct statistics'', i.e.\ it is not the case that states of integral $T_{(1)}$-weight are exactly those with even parity.


\subsection{Small \texorpdfstring{$\calN=4$}{N=4} Superconformal Algebra}
\label{sec:N4-SCA}

In the following, we assume that $N\ge2$ and show that the global sections $\tldcalD^\ch_{M, \hbar}(M)$ contain a quotient of the small $\calN=4$ superconformal algebra $\on{Vir}_{\calN=4,\hbar}^{c_{S_N}}$ of central charge $c_{S_N}=-3(N^2-1)$. We also show that the conformal structure of the superconformal algebra is essentially the one described in the previous section, inherited from the BRST reduction.

\medskip

We define the global sections
{\allowdisplaybreaks
\begin{align}\label{eq:gen-N4-SCA}
J^{+} &= \frac{1}{2}\Bigl(\Tr(X^2) - \frac{1}{N} \Tr(X)^2\Bigr), &
J^{-} &= \frac{1}{2}\Bigl(\Tr(Y^2) - \frac{1}{N} \Tr(Y)^2\Bigr), \\
J^{0} &= \Tr(X Y) - \frac{1}{N} \Tr(X) \Tr(Y), \nonumber\\
G^{+} &= \phi X \gamma - \frac{1}{N} \Tr(X) \phi \gamma, &
G^{-} &= \phi Y \gamma - \frac{1}{N} \Tr(Y) \phi \gamma, \nonumber\\
\tldG^{+} &= - \beta X \psi + \frac{1}{N} \Tr(X) \beta \psi, &
\tldG^{-} &= - \beta Y \psi + \frac{1}{N} \Tr(Y) \beta \psi, \nonumber \\
\TN &= \frac{1}{\hbar} G^{+}_{(0)}\tldG^{-} - \frac{\hbar}{2} \partial J^0 \nonumber
\end{align}
}%
in $\tldC_{\VA}(\frX)$.

\begin{remark}
By direct computation of the operator product expansions, one can show that $\TN$ takes the form
\begin{align*}
\TN &= \phi X Y \psi - \beta Y X \gamma + \beta \gamma \phi \psi - \frac{1}{N} \phi \gamma \beta \psi \\
&\quad + \frac{\hbar}{2}\Bigl(\Tr(\partial X Y) - \Tr(X \partial Y) - \frac{1}{N} \Tr(\partial X) \Tr(Y)
+ \frac{1}{N} \Tr(X) \Tr(\partial Y) \Bigr) \\
&\quad + \hbar \Bigl(N - \frac{1}{N}\Bigr) (\beta \partial \gamma - \partial \phi \psi).
\end{align*}
\end{remark}

Once again, we show that the above global sections are closed, i.e.\ in $\Ker d_{\VA}$.
\begin{proposition}\label{prop:gen-N4-cocyc}
The global sections \eqref{eq:gen-N4-SCA} are cocycles in $\tldC_{\VA}(\frX)$, and in the subcomplex $C_{\VA}(\frX)$.
\end{proposition}
\begin{proof}
It suffices to show that $\mu_\ch(E_{ij})_{(n)}A = 0$ for $i,j=1,\dots,N$ and $n \in \Z_{\ge 0}$ and for any section $A$ defined in \eqref{eq:gen-N4-SCA}. Then, since these sections do not contain any ghost fields $\Psi_{ij}$ or $\Phi_{ij}$, this proves both assertions.

Recall that the Wick formula is realised by bidifferential operators in the variables $x_{ij (-n)},y_{ij (-n)},\dots$, as explained in \autoref{sec:h-adic-betagamma}. For $m \in \Z_{\ge 0}$, we obtain by direct calculation
{\allowdisplaybreaks
\begin{align*}
&\mu_\ch(E_{ij})(z) \Tr(X^m)(w) \\
&\sim \sum_{p=1}^{N} (x_{ip} y_{pj} - x_{pj} y_{ip})(z) \Tr(X^m)(w) \nonumber\\
&\sim \frac{\hbar}{z-w} \sum_{p=1}^{N} \Bigl(\frac{\partial}{\partial y_{pj}} (x_{ip} y_{pj})\Bigr)(z) \Bigl(\frac{\partial}{\partial x_{jp}} \Tr(X^m)\Bigr)(w) \nonumber\\
&\quad - \frac{\hbar}{z-w} \sum_{p=1}^{N} \Bigl(\frac{\partial}{\partial y_{ip}} (x_{pj} y_{ip})\Bigr)(z) \Bigl(\frac{\partial}{\partial x_{pi}} \Tr(X^m)\Bigr)(w) \nonumber\\
& \sim \frac{\hbar}{z-w} \sum_{p=1}^{N} \sum_{k=0}^{m-1} \left\{ \left(x_{ip} \Tr(X^{k} E_{jp} X^{m-k-1})\right)(w) - \left(x_{pj} \Tr(X^{k} E_{pi} X^{m-k-1})\right)(w) \right\} \nonumber\\
&= \frac{\hbar}{z-w} \sum_{p, q=1}^{N} \sum_{k=0}^{m-1} \left\{ \left(x_{ip} (X^{k})_{qj} (X^{m-k-1})_{pq} \right)(w) - \left(x_{pj} (X^{k})_{qp} (X^{m-k-1})_{iq} \right)(w) \right\} \nonumber\\
&= \frac{m \hbar}{z-w} \bigl\{ \left(X^m\right)_{ij}(w) - \left(X^m \right)_{ij}(w) \bigr\} = 0. \nonumber
\end{align*}
}%
This implies that $\mu_\ch(E_{ij})_{(n)} J^{+} = 0$ for any $i,j=1,\dots,N$ and $n \in \Z_{\ge 0}$. Similarly to the above calculation, we also obtain $\mu_\ch(E_{ij})_{(n)} \Tr(Y^m) = 0$, and thus $\mu_\ch(E_{ij})_{(n)} J^{-} = 0$ for any $i$, $j$ and $n$. Moreover, for $m \in \Z_{\ge 0}$,
{\allowdisplaybreaks
\begin{align*}
&\mu_\ch(E_{ij})(z) (\phi Y^m \gamma)(w) \\
& \sim \frac{\hbar}{z-w} \biggl\{\sum_{p=1}^{N} \Bigl(- y_{pj} \frac{\partial \phi Y^m \gamma}{\partial y_{pi}} + y_{ip} \frac{\partial \phi Y^m \gamma}{\partial y_{jp}} \Bigr)(w) \\
&\quad + \Bigl(\gamma_i \frac{\partial \phi Y^m \gamma}{\partial \gamma_{j}} \Bigr)(w) - \Bigl(\phi_j \frac{\partial \phi Y^m \gamma}{\partial \phi_{i}} \Bigr)(w) \biggr\} \\
& = \frac{\hbar}{z-w} \biggl\{ \sum_{p=1}^{N} \sum_{k=0}^{m-1} \left(- y_{pj} \phi Y^{k} E_{pi} Y^{m-k-1} \gamma + y_{ip} \phi Y^{k} E_{jp} Y^{m-k-1} \gamma \right)(w) \\
&\quad + \left(\gamma_i \phi Y^m \bfe_j \right)(w) - \left(\phi_j {}^t \bfe_i Y^m \gamma \right)(w) \biggr\} \\
& = \frac{\hbar}{z-w} \biggl\{ \sum_{k=0}^{m-1} \left(- (\phi Y^{k+1})_{j} (Y^{m-k-1} \gamma)_i + (\phi Y^{k})_{j} (Y^{m-k} \gamma)_{i} \right)(w) \\
&\quad + \left(\gamma_i (\phi Y^m)_j \right)(w) - \left(\phi_j (Y^m \gamma)_i \right)(w) \biggr\} = 0.
\end{align*}
}%
Therefore, $\mu_\ch(E_{ij})_{(n)} G^{-} = 0$ for any $i,j=1,\dots,N$ and $n \in \Z_{\ge 0}$. By similar calculations, one can check that $\mu_\ch(E_{ij})_{(n)} A = 0$ for $A = J^0, G^{+}, \tldG^{+}, \tldG^{-}$, as desired.
\end{proof}

By \autoref{prop:gen-N4-cocyc}, the global sections $J^{+}, J^{0}, J^{-}, \dots \in C_{\VA}(\frX)$ define elements in the cohomology $H^{0}(C_{\VA}(\frX), d_{\VA}) = \tldcalD^\ch_{M, \hbar}(M)$ denoted by the same symbols. The operator product expansions between these elements can also be obtained as a direct consequence of the Wick formula. First, note that $\Tr(X)$ and $\Tr(Y) \in \tldC_{\VA}(\frX)$ are also cocycles and define elements in the cohomology. It is easy to see that $\Tr(X)(z) A(w) \sim \Tr(Y)(z) A(w) \sim 0$ for any element $A = J^{+}, J^{0}, \dots$ defined in \eqref{eq:gen-N4-SCA}. Indeed, for example,
\begin{align}\label{eq:9}
&\Tr(Y)(z) \Tr(X^m)(w) \sim \frac{\hbar}{z-w} \sum_{i=1}^{N} \Tr(E_{ii})(w) \Bigl(\frac{\partial \Tr(X^m)}{\partial x_{ii}}\Bigr)(w) \\
&= \frac{\hbar}{z-w} \sum_{i=1}^{N} \sum_{k=0}^{m-1} \Tr(X^k E_{ii} X^{m-k-1})(w) = \frac{m \hbar}{z-w} \Tr(X^{m-1})(w) \nonumber
\end{align}
for $m \in \Z_{\ge 0}$, which implies $\Tr(Y)(z) J^{+}(w) \sim 0$. We shall later identify $\Tr(X)$ and $\Tr(Y)$ as generators of an $\hbar$-adic $\beta\gamma$-system in $\tldcalD^\ch_{M, \hbar}(M)$.

\medskip

The following six lemmata aim at studying the operator product expansions among the global sections \eqref{eq:gen-N4-SCA}. We shall see that they define a small $\calN=4$ superconformal algebra (see, e.g., \cite{Kac}).
\begin{lemma}\label{lemma:sl2-OPE}
The elements $J^{+}$, $J^{0}$ and $J^{-}$ generate an $\hbar$-adic affine vertex algebra associated with $\frsl_2$ of level $k=-(N^2-1)/2$, i.e.\ the operator product expansions
\begin{align*}
J^0(z) J^{\pm}(w) &\sim \frac{\pm 2 \hbar}{z-w} J^{\pm}(w), \quad
J^{0}(z) J^{0}(w) \sim \frac{-(N^2-1) \hbar^2}{(z-w)^2}, \\
J^{+}(z) J^{-}(w) &\sim \frac{(N^2-1)\hbar^2 /2}{(z-w)^2} + \frac{- \hbar}{z-w} J^{0}(w)
\end{align*}
hold, and other operator product expansions between the generators are trivial.
\end{lemma}
More precisely, what we mean is that $J^{+}$, $J^{0}$ and $J^{-}$ form some quotient of the universal $\hbar$-adic affine vertex algebra associated with $\frsl_2$ of level $k=-(N^2-1)/2$.
\begin{proof}
It follows from the Wick formula that
{\allowdisplaybreaks
\begin{align*}
&\Tr(XY)(z) \Tr(X^m)(w) \sim \frac{\hbar}{z-w} \sum_{i,j=1}^{N} \Bigl(\frac{\partial \Tr(XY)}{\partial y_{ij}}\Bigr)(w) \Bigl(\frac{\partial \Tr(X^m)}{\partial x_{ji}}\Bigr)(w) \\
&\sim \frac{\hbar}{z-w} \sum_{i,j=1}^{N} \sum_{k=0}^{m-1} \Tr(X E_{ij})(w) \Tr(X^k E_{ji} X^{m-k-1})(w) \\
&\sim \frac{\hbar}{z-w} \sum_{i,j,p=1}^{N} \sum_{k=0}^{m-1} x_{ji}(w) \left( (X^k)_{pj} (X^{m-k-1})_{ip} \right)(w) = \frac{m \hbar}{z-w} \Tr(X^m)(w)
\end{align*}
}%
for $m \in \Z_{\ge 0}$. Thus, $J^{0}(z) J^{+}(w) \sim 2 J^{+}(w) \hbar/(z-w)$. Moreover,
{\allowdisplaybreaks
\begin{align*}
&\Tr(X^2)(z) \Tr(Y^2)(w) \\
&\sim \frac{\hbar^2/2}{(z-w)^2} \sum_{i,j,k,l=1}^{N} \Bigl(\frac{\partial^2 \Tr(X^2)}{\partial x_{ij} \partial x_{kl}} \Bigr)(w) \Bigl(\frac{\partial^2 \Tr(Y^2)}{\partial y_{ji} \partial y_{lk}} \Bigr)(w) \\
&\quad - \frac{\hbar}{z-w} \sum_{i,j=1}^{N} \Bigl(\frac{\partial \Tr(X^2)}{\partial x_{ij}} \Bigr)(w) \Bigl(\frac{\partial \Tr(Y^2)}{\partial y_{ji}} \Bigr)(w) \\
&= \frac{\hbar^2/2}{(z-w)^2} \sum_{i,j,k,l=1}^{N} \Tr(E_{ij}E_{kl} + E_{kl} E_{ij})(w) \Tr(E_{ji}E_{lk} + E_{lk} E_{ji})(w) \\
&\quad - \frac{\hbar}{z-w} \sum_{i,j=1}^{N} \Tr(E_{ij} X + X E_{ij})(w) \Tr(E_{ji} Y + Y E_{ji})(w) \\
&= \frac{2N^2 \hbar^2}{(z-w)^2} - \frac{4 \hbar}{z-w} \Tr(XY)(w),
\end{align*}
}%
and $\Tr(X^2)(z) \Tr(Y)^2(w) \sim 2 N \hbar^2/(z-w)^2 - 4 \hbar (\Tr(X) \Tr(Y))(w)/(z-w)$ by \eqref{eq:9}. It follows that $J^{+}(z) J^{-}(w) \sim ((N^2-1) \hbar^2 / 2) / (z-w)^2 - \hbar J^0(w) / (z-w)$. The other operator product expansions are obtained analogously.
\end{proof}

By \autoref{lemma:sl2-OPE}, there is an action of the $\hbar$-adic affine Lie algebra $\hatsl_2$ of level $k=-(N^2-1)/2$ on $\tldcalD^\ch_{M, \hbar}(M)$. A natural question is how the affine Lie algebra $\hatsl_2$ acts on the other elements $G^{+}$, $G^{-}$, $\tldG^{+}$ and $\tldG^{-}$. By direct verification, similar to the above lemmata, we obtain:
\begin{lemma}\label{lemma:sl2-doublets}
For $A = G,\tldG$, the elements $\{A^{+}, A^{-}\}$ generate some quotient of the Weyl module of the affine Lie algebra $\hatsl_2$ of level $k=-(N^2-1)/2$ associated with the two-dimensional irreducible representation of $\frsl_2$, i.e.\ the following operator product expansions hold:
\[
J^{0}(z) A^{\pm}(w) \sim \frac{\pm \hbar}{z-w} A^{\pm}(w), \;\; J^{\pm}(z) A^{\mp}(w) \sim \frac{\mp \hbar}{z-w} A^{\pm}(w), \;\; J^{\pm}(z) A^{\pm}(w) \sim 0.
\]
\end{lemma}

\begin{lemma}\label{lemma:doublet-OPE-triv}
For $A = G,\tldG$, we have $A^{+}(z) A^{-}(w) \sim A^{\pm}(z) A^{\pm}(w) \sim 0$.
\end{lemma}
\begin{proof}
It is clear that $A^{\pm}(z) A^{\pm}(w) \sim 0$. It remains to show $G^{+}(z) G^{-}(w) \sim 0$ and $\tldG^{+}(z) \tldG^{-}(w) \sim 0$. To conclude the former, we observe that
\begin{align*}
(\phi X \gamma)(z) (\phi Y \gamma)(w) &\sim \frac{-\hbar}{z-w} \sum_{i,j=1}^{N} \Bigl(\frac{\partial \phi X \gamma}{\partial x_{ij}} \Bigr)(w) \Bigl(\frac{\partial \phi Y \gamma}{\partial y_{ji}} \Bigr)(w) \\
&= \frac{-\hbar}{z-w} \sum_{i,j=1}^{N} \left(\phi E_{ij} \gamma \right)(w) \left(\phi E_{ji} \gamma \right)(w) \\
&= \frac{-\hbar}{z-w} \sum_{i,j=1}^{N} (\phi_i \gamma_j \phi_j \gamma_i)(w) = \frac{-\hbar}{z-w} \sum_{i,j=1}^{N} (\phi \gamma)^2(w) = 0
\end{align*}
since $\phi \gamma$ is an odd element. Also, $\Tr(X)(z) (\phi Y \gamma)(w) \sim - (\phi \gamma)(w)/(z-w)$, and thus we obtain $G^{+}(z) G^{-}(w) \sim 0$. The operator product expansion $\tldG^{+}(z) \tldG^{-}(w) \sim 0$ can be obtained in a similar way.
\end{proof}

The remaining operator product expansions between the elements defined in \eqref{eq:gen-N4-SCA} are of the form $G^{a}(z) \tldG^{b}(w)$ for $a,b = +,-$. To compute them, we first need the following lemma.
\begin{lemma}\label{lemma:7}
For $m \in \Z_{\ge 0}$, the global sections $\beta X^m \gamma - \phi X^m \psi$ and $\beta Y^m \gamma - \phi Y^m \psi$ in $\tldC_{\VA}(\frX)$ are coboundaries, i.e.\ in $\Im d_{\VA}$.
\end{lemma}
\begin{proof}
Consider the element $\Tr(X^m \Phi)$, where $\Phi = (\Phi_{ij})_{i,j=1}^{N}$. Then, using the Wick formula, we see that $d_{\VA} \Tr(X^m \Phi) = (1/\hbar) Q_{(0)} \Tr(X^m \Phi) = \beta X^m \gamma - \phi X^m \psi$.
\end{proof}

\begin{lemma}\label{lemma:Virasoro}
The element $\TN$ defines a conformal vector of central charge $c_{S_N}=-3(N^2-1)$ of the subalgebra of $\tldcalD^\ch_{M, \hbar}(M)$ generated by the elements $J^{\pm}$, $J^{0}$, $G^{\pm}$, $\tldG^{\pm}$, $\TN$ in \eqref{eq:gen-N4-SCA}, i.e.\ the following operator product expansions hold:
\begin{align*}
\TN(z) \TN(w) &\sim \frac{-3 \hbar^4 (N^2-1) / 2}{(z-w)^4} + \frac{2 \hbar^2}{(z-w)^2} \TN(w) + \frac{\hbar^2}{z-w} \partial \TN(w), \\
\TN(z) A(w) &\sim \frac{\hbar^2}{(z-w)^2} A(w) + \frac{\hbar^2}{z-w} \partial A(w),\\
\TN(z) B(w) &\sim \frac{3 \hbar^2 / 2}{(z-w)^2} B(w) + \frac{\hbar^2}{z-w} \partial B(w)
\end{align*}
for $A = J^{\pm},J^{0}$ and $B = G^{\pm},\tldG^{\pm}$.
\end{lemma}
\begin{proof}
By direct calculation, we verify $G^{-}_{(0)} \tldG^{+} - (\hbar^2/2) \partial J^0 = - \hbar \TN$. The operator product expansions for $\TN$ then follow from \autoref{lemma:sl2-doublets}, \autoref{lemma:doublet-OPE} below and Borcherds' identity $(a_{(0)} b)_{(n)} = a_{(0)} b_{(n)} - (-1)^{p(a) p(b)} b_{(n)} a_{(0)}$.
\end{proof}

\begin{lemma}\label{lemma:doublet-OPE}
For $a,b = +,-$, the following operator product expansions hold:
\[
G^{a}(z) \tldG^{b}(w) \sim \frac{-(N^2-1) \epsilon^{ab} \hbar^3}{(z-w)^3} + \frac{2 J^{ab}(w) \hbar^2}{(z-w)^2} + \frac{\epsilon^{ab} \TN(w) \hbar + \partial J^{ab}(w) \hbar^2}{z-w},
\]
where $\epsilon^{+-} = 1$, $\epsilon^{-+} = -1$, $\epsilon^{\pm \pm} = 0$ and $J^{+-} = J^{-+} = J^0/2$, $J^{\pm \pm} = J^{\pm}$.
\end{lemma}
\begin{proof}
Again by the Wick formula, we compute
{\allowdisplaybreaks
\begin{align*}
&(\phi X \gamma)(z) (\beta X \psi)(w) \sim \frac{- \hbar^2}{(z-w)^2} \sum_{i,j=1}^{N} \Bigl(\frac{\partial^2 \phi X \gamma}{\partial \phi_i \partial \gamma_j}\Bigr)(w) \Bigl(\frac{\partial^2 \beta X \psi}{\partial \psi_i \partial \beta_j}\Bigr)(w) \\
&\quad + \frac{- \hbar^2}{z-w} \sum_{i,j=1}^{N} \Bigl(\partial_w \frac{\partial^2 \phi X \gamma}{\partial \phi_i \partial \gamma_j}\Bigr)(w) \Bigl(\frac{\partial^2 \beta X \psi}{\partial \psi_i \partial \beta_j}\Bigr)(w) \\
&\quad + \frac{\hbar}{z-w} \sum_{i=1}^{N} \Bigl(\frac{\partial \phi X \gamma}{\partial \phi_i}\Bigr)(w) \Bigl(\frac{\partial \beta X \psi}{\partial \psi_i}\Bigr)(w) + \frac{-\hbar}{z-w} \sum_{i=1}^{N} \Bigl(\frac{\partial \phi X \gamma}{\partial \gamma_i}\Bigr)(w) \Bigl(\frac{\partial \beta X \psi}{\partial \beta_i}\Bigr)(w) \\
&= \frac{- \hbar^2}{(z-w)^2} \sum_{i,j=1}^{N} (x_{ij} x_{ji})(w) + \frac{- \hbar^2}{z-w} \sum_{i,j=1}^{N} (\partial x_{ij} x_{ji})(w) \\
&\quad + \frac{\hbar}{z-w} \sum_{i=1}^{N} \left\{ (X \gamma)_i(w) (\beta X)_i(w) - (\phi X)_{i}(w) (X \psi)_i(w) \right\} \\
&= \frac{- \hbar^2}{(z-w)^2} \Tr(X^2)(w) + \frac{- \hbar^2}{z-w} \frac{1}{2} \partial_w \Tr(X^2)(w) + \frac{\hbar}{z-w}(\beta X^2 \gamma - \phi X^2 \psi)(w) \\
&\equiv \frac{- \hbar^2}{(z-w)^2} \Tr(X^2)(w) + \frac{- \hbar^2}{z-w} \frac{1}{2} \partial_w \Tr(X^2)(w),
\end{align*}
}%
where we use \autoref{lemma:7} for the equivalence modulo $\Im d_{\VA}$. Similarly,
\[
(\phi X \gamma)(z) (\Tr(X) \beta \psi)(w) \sim \frac{- \hbar^2}{(z-w)^2} \Tr(X)^2(w) + \frac{- \hbar^2}{z-w} \frac{1}{2} \partial_w \Tr(X)^2(w)
\]
modulo $\Im d_{\VA}$, and hence we obtain the desired operator product expansion $G^{+}(z) \tldG^{+}(w)$ of the lemma. The other operator product expansions can be verified in the same way.
\end{proof}

As a conclusion of Lemmata~\ref{lemma:sl2-OPE}--\ref{lemma:doublet-OPE}, we obtain:
\begin{proposition}\label{prop:small-N4-SCA}
For $N\ge2$, the elements \eqref{eq:gen-N4-SCA} define a homomorphism of $\hbar$-adic vertex superalgebras
\[
\on{Vir}_{\calN=4,\hbar}^{c_{S_N}}\longrightarrow \tldcalD^\ch_{M, \hbar}(M).
\]
\end{proposition}
We denote the image of the above homomorphism by $V_{\calN=4,\hbar}$. It is some quotient of the universal $\hbar$-adic small $\calN=4$ superconformal algebra $\on{Vir}_{\calN=4,\hbar}^{c_{S_N}}$ of central charge $c_{S_N}=-3(N^2-1)$.

The analogous statement holds for the vertex operator superalgebra $\sfV_{S_N}$ with a vertex algebra homomorphism $\on{Vir}_{\calN=4}^{c_{S_N}}\longrightarrow \sfV_{S_N}$, whose image we denote $V_{\calN=4}$.

\medskip

Finally, we study further cocycles in $C_{\VA}(\frX)$, corresponding to non-zero elements in $\tldcalD^\ch_{M, \hbar}(M)$, and show that together with $V_{\calN=4,\hbar}$ they generate a conformal subalgebra of the $\hbar$-adic vertex operator superalgebra $\tldcalD^\ch_{M, \hbar}(M)$.
\begin{lemma}
The elements $\Tr(X)/\sqrt{N}$ and $\Tr(Y)/\sqrt{N} \in \tldcalD^\ch_{M, \hbar}(M)$ form an $\hbar$-adic $\beta\gamma$-system $\calD^\ch(T^*\C^1)_{\hbar}$.
\end{lemma}
\begin{proof}
This follows directly from
\begin{align*}
\Tr(X)(z)\Tr(Y)(w) &\sim - N \hbar/ (z-w),\\
\Tr(X)(z)\Tr(X)(w) &\sim \Tr(Y)(z)\Tr(Y)(w) \sim 0.\qedhere
\end{align*}
\end{proof}
\begin{lemma}
The elements $\Lambda_1 \coloneqq \phi \gamma / \sqrt{N}$ and $\Lambda_2 \coloneqq \beta \psi / \sqrt{N} \in \tldcalD^\ch_{M, \hbar}(M)$ generate an $\hbar$-adic symplectic fermion vertex superalgebra $\SF_{\hbar}$.
\end{lemma}
\begin{proof}
The assertion follows from the operator product expansions
\begin{align*}
\Lambda_1(z)\Lambda_2(w) &\sim -\hbar^2/ (z-w)^2,\\
\Lambda_1(z)\Lambda_1(w) &\sim \Lambda_2(z)\Lambda_2(w) \sim 0.\qedhere
\end{align*}
\end{proof}

Moreover, one can easily verify that the operator product expansions among these three vertex subalgebras are trivial. Thus,
\[
V_{\calN=4,\hbar} \hatotimes \calD^\ch(T^*\C^1)_{\hbar} \hatotimes \SF_{\hbar}\subset \tldcalD^\ch_{M, \hbar}(M)
\]
as $\hbar$-adic vertex superalgebras.

Now, consider the conformal vectors
\begin{align*}
\TTr &= \frac{\hbar}{2N}\bigl(\Tr(\partial X) \Tr(Y) - \Tr(X) \Tr(\partial Y) \bigr) \in \calD^\ch(T^*\C^1)_{\hbar},\\
\TSF &= \Lambda_1 \Lambda_2 = \frac{1}{N}( \phi \gamma \beta \psi + \hbar \beta \partial \gamma - \hbar \partial \phi \psi) \in \SF_{\hbar}
\end{align*}
of the $\hbar$-adic vertex subalgebras $\calD^{\ch}(T^*\C^1)_{\hbar}$ and $\SF_{\hbar}$. They have central charge $c=-1$ and $-2$, respectively.

Overall, the tensor product $V_{\calN=4,\hbar} \hatotimes \calD^\ch(T^*\C^1)_{\hbar} \hatotimes \SF_{\hbar}$ inside $\tldcalD^\ch_{M, \hbar}(M)$ has a conformal structure of central charge $c=-3(N^2-1) + (-1) + (-2) = -3 N^2$ defined by $\TN + \TTr + \TSF$. We show that it coincides with the natural conformal structure of $\tldcalD^\ch_{M, \hbar}(M)$ defined in \autoref{sec:conformal}:
\begin{proposition}\label{prop:total-conf-vec}
For $N\ge2$, the $\hbar$-adic vertex operator superalgebra $\tldcalD^\ch_{M, \hbar}(M)$ of central charge $c=-3N^2$ is a conformal extension of $V_{\calN=4,\hbar} \hatotimes \calD^\ch(T^*\C^1)_{\hbar} \hatotimes \SF_{\hbar}$.
\end{proposition}
\begin{proof}
We need to show that for $T$ from \eqref{eq:free-field-conformal},
\begin{equation}\label{eq:4}
T \equiv \TN + \TTr + \TSF
\end{equation}
modulo $\Im d_{\VA}$. Indeed, one can easily verify the identity
\begin{align*}
&T - (\TN + \TTr + \TSF) \\
&= \beta Y X \gamma - \phi X Y \psi - \beta \gamma \phi \psi -
\hbar\Bigl(N - \frac{1}{2}\Bigr)\beta \partial \gamma - \frac{\hbar}{2} \partial \beta \gamma \\
&\quad - \frac{\hbar}{2} \phi \partial \psi + \hbar\Bigl(N + \frac{1}{2}\Bigr) \partial \phi \psi + \hbar\Tr(\partial \Psi {}^t \Phi) \\
&= \frac{1}{2} d_{\VA} \bigl(\Tr(X Y \Phi) + \Tr(Y X \Phi) - \beta \Phi \gamma + \phi \Phi \psi + (\beta\gamma) \Tr(\Phi) \\
& \phantom{=\frac{1}{2} d_{\VA} \bigl(} - (\phi \psi) \Tr(\Phi) - N \hbar \Tr(\partial \Phi)
\bigr) \in \Im d_{\VA}.\qedhere
\end{align*}
\end{proof}
In the next section, we shall actually see that $\tldcalD^\ch_{M, \hbar}(M)$ is isomorphic to a tensor product of $\calD^\ch(T^*\C^1)_{\hbar} \hatotimes \SF_{\hbar}$ and some conformal extension of $V_{\calN=4,\hbar}$. That is, we can split off the (unimportant) $\beta\gamma$-system and the symplectic fermion completely.

As an immediate consequence of the proposition we obtain:
\begin{corollary}
For $N\ge2$, the vertex operator superalgebra $\sfV_{S_N}$ of central charge $c=-3N^2$ is a conformal extension of $V_{\calN=4} \otimes \calD^\ch(T^*\C^1) \otimes \SF$.
\end{corollary}


\section{Free-Field Realisation}
\label{sec:Wakimoto}

In this section, we consider the local sections of the sheaf $\tldcalD^\ch_{M,\hbar}$ over the affine open subset $U_{(N)}$ defined in \autoref{sec:big-cell}, and thus obtain a free-field realisation of the global sections, and of $\sfV_{S_N}$, given by the sheaf restriction morphism.

This also shows a factorisation of $\sfV_{S_N}$ that allows us to split off the vertex operator superalgebra $\W_{S_N}$ of central charge $c_{S_N}=-3(N^2-1)$ from $\sfV_{S_N}$. The latter is the vertex operator superalgebra for the reflection group $S_N$ conjectured by Bonetti, Meneghelli and Rastelli \cite{BMR19}. We show that $\W_{S_N}$ is a conformal extension of the small $\calN=4$ superconformal algebra of central charge $c_{S_N}$. Moreover, $\W_{S_N}$ has the associated variety $\calM_{S_N}$, is quasi-lisse and has a free-field realisation in terms of a $\beta\gamma bc$-system of rank $N-1$ that coincides with the one proposed in \cite{BMR19} (generalising \cite{Adamovic16} for $N=2$).

\medskip

We recall from \autoref{sec:big-cell} the affine open subset $\tldU_{(N)} \subset \frX$ and the sections $[X^N : X^i]_{(N)}$, $[Y : X^i]_{(N)}$ and $\{\psi : X^i \gamma\}_{(N)} \in \tldcalO_{\frX}(\tldU_{(N)})$ for $i=0,\dots,N-1$. In this section, we omit the subscript ${(N)}$ from these sections and simply write $[X^N : X^i]$, $[Y : X^i]$ and $\{\psi : X^i \gamma\}$, respectively. Via the embedding $\tldcalO_M \hookrightarrow \tldcalD^\ch_{\frX, \hbar} \subset \tldC_{\VA}$, we regard
these sections as elements of $\tldC_{\VA}(\tldU_{(N)})$.

\begin{lemma}\label{lemma:deriv-matB}
The matrix $B_{(N)} = (\gamma, X \gamma, \dots, X^{N-1} \gamma)$ satisfies the following identities for $p,q=1,\dots,N$:
\begin{enumerate}[leftmargin=*]
\item $(\partial / \partial x_{pq}) B_{(N)} = \sum_{k=0}^{N-2} (X^k \gamma)_q (0, \dots, 0, \bfe_p, X \bfe_p, \dots, X^{N-k-1} \bfe_p)$,
\item $(\partial / \partial x_{pq}) B^{-1}_{(N)} = - \sum_{k=0}^{N-2} (X^k \gamma)_q B^{-1}_{(N)} (0, \dots, 0, \bfe_p, X \bfe_p, \dots, X^{N-k-1} \bfe_p) B^{-1}_{(N)}$.
\end{enumerate}
\end{lemma}
\begin{proof}
By the Leibniz rule, we obtain $(\partial / \partial x_{pq}) X^m \gamma = \sum_{k=0}^{m-1} X^{m-k-1} E_{pq} X^{k} \gamma = \sum_{k=0}^{m-1} (X^k \gamma)_q X^{m-k-1} \bfe_p$, which implies (1). The identity (2) follows from (1) and $B_{(N)} (\partial B_{(N)}^{-1} / \partial x_{pq}) + (\partial B_{(N)} / \partial x_{pq}) B^{-1}_{(N)} = 0$.
\end{proof}

\begin{lemma}\label{lemma:prim-gen-cocycle}
The elements $[X^N : X^i]$, $[Y : X^i]$ and $\{\psi : X^i \gamma\} \in \tldC_{\VA}(\tldU_{(N)})$ for $i=0,\dots,N-1$ are cocycles with respect to the coboundary operator $d_{\VA}$.
\end{lemma}
\begin{proof}
Again, it suffices to show that $\mu_\ch(E_{ij})_{(n)} A = 0$ for $A = [X^N : X^k], {[Y : X^k]},\allowbreak \{\psi : X^k \gamma\}$ and $k=0,\dots,N-1$. We consider the operator product expansion of $\mu_\ch(A_{ij})(z) (B_{(N)}^{-1} Y \gamma)(w)$ for the vector $B_{(N)}^{-1} Y \gamma = {}^t([Y : X^{k-1}])_{k=1, \dots, N}$. By the Wick formula,
\begin{align*}
&\mu_\ch(E_{ij})(z) (B_{(N)}^{-1} Y \gamma)(w) \\
&\sim \frac{- \hbar^2}{(z-w)^2} \sum_{i,j,p=1}^{N} \biggl(\frac{\partial^2 B_{(N)}^{-1} Y \gamma}{\partial x_{jp} \partial y_{pi}} - \frac{\partial^2 B_{(N)}^{-1} Y \gamma}{\partial x_{pi} \partial y_{jp}}\biggr)(w) + \frac{\hbar}{z-w} ( \dots ).
\end{align*}
The second term $\hbar ( \dots ) / (z-w)$ comes from the single contraction of the operator product expansion. Note that $[Y : X^k]$ is $G$-invariant for any $k=0,\dots,N-1$, which implies that the second term vanishes. We then note that
\begin{align*}
\sum_{p=1}^{N} \frac{\partial^2 B_{(N)}^{-1} Y \gamma}{\partial x_{jp} \partial y_{pi}} &= - \sum_{p=1}^{N} \sum_{k=0}^{N-2} (X^k \gamma)_p B^{-1}_{(N)} (0, \dots, 0, \bfe_j, \dots, X^{N-k-1} \bfe_j) B^{-1}_{(N)} \gamma_i \bfe_p \\
&= - \sum_{k=0}^{N-2} \gamma_i B^{-1}_{(N)} (0, \dots, 0, \bfe_j, \dots, X^{N-k-1} \bfe_j) B^{-1}_{(N)} X^k \gamma \\
&= - \sum_{k=0}^{N-2} \gamma_i B^{-1} (0, \dots, 0, \bfe_j, \dots, X^{N-k-1} \bfe_j) \bfe_{k+1}.
\end{align*}
The first $k+1$ columns of the matrix $(0, \dots, 0, \bfe_j, \dots, X^{N-k-1} \bfe_j)$ vanish so that $(0, \dots, 0, \bfe_j, \dots, X^{N-k-1} \bfe_j) \bfe_{k+1} = 0$. Hence, $\mu_\ch(E_{ij})(z) (B^{-1}_{(N)} Y \gamma)(w) \sim 0$. On the other hand, the elements $[X^N : X^i]$ and $\{\psi : X^i \gamma\}$ are $G$-invariant and the operator product expansions of $\mu_\ch(E_{ij})(z) [X^N : X^i](w)$ and $\mu_\ch(E_{ij})(z) \{\psi : X^i \gamma\}(w)$ do not cause multiple contractions. Thus, these elements are cocycles.
\end{proof}

Since the elements $[X^N : X^i]$, $[Y : X^i]$, $\{\psi : X^i \gamma\}$ and
$\phi X^i \gamma$ for $i=0,\dots,N-1$ are $G$-invariant, they lie in the subcomplex $C_{\VA}(\tldU_{(N)})$. By the above lemma, they define elements in $H^0(C_{\VA}(\tldU_{(N)}), d_{\VA}) = \tldcalD^\ch_{M, \hbar}(U_{(N)})$. In fact, the $\hbar$-adic vertex superalgebra $\tldcalD^\ch_{M, \hbar}(U_{(N)})$ of sections over $U_{(N)}$ is strongly generated by these sections, i.e.
\begin{align}\label{eq:1}
\tldcalD^\ch_{M, \hbar}(U_{(N)}) &= \C[[\hbar]][[X^N : X^i]_{(-n)}, [Y : X^i]_{(-n)} \,|\, \substack{ i = 0, \dots, N-1 \\ n = 1, 2, \dots}] \\
&\quad \hatotimes \Lambda_{\C[[\hbar]]}(\{\psi : X^i \gamma\}_{(-n)}, (\phi X^i \gamma)_{(-n)} \,|\, \substack{ i = 0, \dots, N-1 \\ n = 1, 2, \dots}) \nonumber
\end{align}
as $\C[[\hbar]]$-module. We shall show that the $\hbar$-adic vertex superalgebra $\tldcalD^\ch_{M, \hbar}(U_{(N)})$ is isomorphic to an $\hbar$-adic $\beta\gamma bc$-system of rank $N-1$. First, we determine the operator product expansions between the above generators.

\begin{lemma}\label{lemma:OPE-Tr-local}
For $i=0,\dots,N-1$, the operator product expansions
{\allowdisplaybreaks
\begin{align}\label{eq:8}
\Tr(X)(z) [Y : X^i](w) &\sim
\begin{cases}
\displaystyle \frac{\hbar}{z-w}, & i = 0, \\
0, & i \ne 0,
\end{cases} \\
\label{eq:11}
\Tr(Y)(z) [Y : X^i](w) &\sim
\begin{cases}
0, & i = N-1, \\
\displaystyle - \frac{(i+1) \hbar}{z-w} [Y : X^{i+1}], & i \ne N-1,
\end{cases} \\
\nonumber
\Tr(Y)(z) [X^N : X^i](w) &\sim
\begin{cases}
\displaystyle \frac{N \hbar}{z-w}, & i = N-1, \\
\displaystyle - \frac{(i+1) \hbar}{z-w} [X^N : X^{i+1}](w), & i \ne N-1,
\end{cases} \\
\nonumber
\Tr(Y)(z) \{\psi : X^i \gamma\}(w) &\sim
\begin{cases}
0, & i = N-1, \\
\displaystyle - \frac{(i+1) \hbar}{z-w} \{ \psi : X^{i+1} \gamma \}(w), & i \ne N-1,
\end{cases} \\
\nonumber
\Tr(Y)(z) (\phi X^i \gamma)(w) &\sim \frac{i \hbar}{z-w} (\phi X^{i-1} \gamma)(w)
\end{align}
}%
hold, and other operator product expansions between $\Tr(X)$, $\Tr(Y)$ and the elements in \eqref{eq:1} are trivial.
\end{lemma}
\begin{proof}
Since the elements $[X^N : X^i]$, $\{\psi : X^i \gamma\}$ and $\phi X^i \gamma$ do not contain the variables $y_{pq}$, $p,q=1,\dots,N$, they have trivial operator product expansions with $\Tr(X)$. We note that
\begin{align*}
&\Tr(X)(z) (B^{-1}_{(N)} Y \gamma)(w) \sim \frac{- \hbar}{z-w} \sum_{p=1}^{N}\biggl(\frac{\partial B^{-1}_{(N)} Y \gamma}{\partial y_{pp}} \biggr)(w) \\
&= \frac{- \hbar}{z-w} \sum_{p=1}^{N}\left(B^{-1}_{(N)} E_{pp} \gamma \right)(w)= \frac{- \hbar}{z-w}\left(B^{-1}_{(N)} \gamma\right)(w) = \frac{- \hbar}{z-w} \bfe_1,
\end{align*}
which implies \eqref{eq:8}. To see \eqref{eq:11}, we use \autoref{lemma:deriv-matB}~(2) and
\begin{align*}
&\Tr(Y)(z) (B^{-1}_{(N)} Y \gamma)(w) \sim \frac{\hbar}{z-w} \sum_{p=1}^{N} \biggl(\frac{\partial B^{-1}_{(N)}}{\partial x_{pp}} Y \gamma \biggr)(w) \\
&\sim \frac{- \hbar}{z-w} B^{-1}_{(N)} (0, \gamma, 2 X \gamma, \dots, (N-1) X^{N-2} \gamma) B^{-1}_{(N)} Y \gamma \\
&= \frac{- \hbar}{z-w} (0, \bfe_1, 2 \bfe_2, \dots, (N-1) \bfe_{N-1})\, {}^t \!\left([Y : X^{i-1}]\right)_{i=1, \dots, N} \\
&= \frac{- \hbar}{z-w} {}^t\!\left(j [Y : X^j]\right)_{j=1, \dots, N}.
\end{align*}
The other operator product expansions can be obtained in a similar way.
\end{proof}

\begin{lemma}
For $i,j=0,\dots,N-1$, there are the operator product expansions
{\allowdisplaybreaks
\begin{align*}
[Y : X^i](z) [X^N : X^j](w) &\sim
\begin{cases}
\displaystyle \frac{\hbar}{z-w}, & i=N-j-1, \\
\displaystyle \frac{-\hbar}{z-w} [X^N : X^{i+j+1}](w), & i=0, \dots, N-j-2, \\
0, & i = N-j, \dots, N-1,
\end{cases} \\
[Y : X^i](z) [Y : X^j](w) &\sim 0, \\
[Y : X^i](z) \{\psi : X^j \gamma\}(w) &\sim
\begin{cases}
\displaystyle \frac{-\hbar}{z-w} \{\psi : X^{i+j+1} \gamma\}(w), & i=0, \dots, N-j-2, \\
0, & i=N-j-1, \dots, N-1,
\end{cases} \\
[Y : X^i](z) (\phi X^{j} \gamma)(w) &\sim
\begin{cases}
\displaystyle \frac{\hbar}{z-w} (\phi X^{j-i-1} \gamma)(w), & i=0, \dots, j-1, \\
0, & i=j, \dots, N-1.
\end{cases}
\end{align*}
}%
\end{lemma}
\begin{proof}
By \autoref{lemma:deriv-matB}, for $p=1,\dots,N$,
\begin{align*}
\sum_{q=1}^{N} \gamma_{q} \frac{\partial}{\partial x_{qp}} B^{-1}_{(N)} &= \sum_{k=0}^{N-2} (X^k \gamma)_{p} B^{-1}_{(N)} (0, \dots, 0, \gamma, X \gamma, \dots, X^{N-k-1} \gamma) B^{-1}_{(N)}, \\
\sum_{q=1}^{N} \gamma_{q} \frac{\partial}{\partial x_{qp}} X^N \gamma
&= \sum_{k=0}^{N-1} (X^k \gamma)_p X^{N-k-1} \gamma.
\end{align*}
This implies that
{\allowdisplaybreaks
\begin{align*}
&(Y \gamma)_{p}(z) (B^{-1}_{(N)} X^N \gamma)(w) \sim \frac{\hbar}{z-w} \sum_{q=1}^{N} \Bigl(\gamma_q \frac{\partial}{\partial x_{qp}} B^{-1}_{(N)} X^N \gamma\Bigr)(w) \\
&= \frac{\hbar}{z-w} \biggl\{\sum_{k=0}^{N-1} \bigl((X^k \gamma)_p B^{-1}_{(N)} X^{N-k-1} \gamma\bigr)(w) \\
&\quad - \sum_{k=0}^{N-k-2} \bigl((X^k \gamma)_{p} B^{-1}_{(N)} (0, \dots, 0, \gamma, X \gamma, \dots, X^{N-k-1} \gamma) B^{-1}_{(N)} X^{N} \gamma\bigr)(w)\biggr\} \\
&= \frac{\hbar}{z-w} {\Big.}^t\!\Bigl( (X^{N-j} \gamma)_{p}(w) - \sum_{k=0}^{N-2} (X^k \gamma)_p [X^N : X^{k+j}] \Bigr)_{j=1, \dots, N}
\end{align*}
}%
since $B^{-1}_{(N)} X^{j-1} \gamma = \bfe_{j}$ for $j=1,\dots,N$. Multiplying by $\sum_{p=1}^{N} (B^{-1}_{(N)})_{ip}$ on both sides, we obtain the first identity of the lemma. The other operator product expansions can be obtained in a similar fashion.
\end{proof}

\begin{lemma}\label{lemma:10}
For $i,j=0,\dots,N-1$, there is the operator product expansion
$(\phi X^i \gamma)(z) \{\psi : X^j \gamma\}(w) \sim \delta_{ij} \hbar / (z-w)$.
\end{lemma}
\begin{proof}
By the Wick formula,
\begin{align*}
(\phi X^i \gamma)(z) \{\psi : X^j \gamma\}(w) &\sim \frac{\hbar}{z-w} \sum_{p,q=1}^{N} \left((X^i \gamma)_p \frac{\partial}{\partial \psi_p} (B^{-1}_{(N)})_{j+1,q} \psi_q \right)(w) \\
&= \frac{\hbar}{z-w} \sum_{p=1}^{N} \left((B^{-1}_{(N)})_{j+1,p} (X^i \gamma)_p \right)(w) = \frac{\hbar}{z-w} \delta_{ij}.\qedhere
\end{align*}
\end{proof}

Obviously, the operator product expansions between $[X^N : X^i]$ and $\{\psi : X^j \gamma\}$, $\phi X^j \gamma$ are trivial for any $i,j=0,\dots,N-1$. Note that all operator product expansions in Lemmata \ref{lemma:OPE-Tr-local}--\ref{lemma:10} have simple poles, and they have a certain triangular property, e.g., $[Y : X^{i}](z) [X^N : X^j](w) \sim 0$ and $[Y : X^{i}](z) \{\psi : X^j \gamma\}(w) \sim 0$ unless $i < N-j-1$. This enables us to diagonalise the operator product expansions, as stated in the following proposition. For $m=1,\dots,N-1$, we define
{\allowdisplaybreaks
\begin{align*}
\Wb_m &= \sum_{k=1}^{m} \left(\frac{-1}{N}\right)^{m-k} \binom{m}{k} \Tr(X)^{m-k} \Bigl( \phi X^k \gamma - \frac{1}{N} \Tr(X^k) \phi \gamma \Bigr), \\
\Wc_m &= \sum_{k=m}^{N-1} \left(\frac{1}{N}\right)^{m-k} \binom{k}{m} \Tr(X)^{k-m} \{\psi : X^k \gamma\}, \\
\Wbeta_{m} &= \frac{1}{m+1} \biggl(\, \sum_{k=2}^{m+1} \left(\frac{-1}{N}\right)^{m-k+1} \binom{m+1}{k} \Tr(X)^{m-k+1} \Tr(X^k) \\
&\quad + m \left(\frac{-1}{N}\right)^{m} \Tr(X)^{m+1} \biggr), \\
\Wgamma_{m} &= \sum_{k=m}^{N-1} \left(\frac{1}{N}\right)^{k-m} \binom{k}{m} \Tr(X)^{k-m} [Y : X^k]
- \sum_{j=1}^{N-m-2} \Wb_j \Wc_{m+j+1}\\
&\quad - \sum_{j=1}^{N-m-3} \frac{j+1}{N} \Wbeta_j \phi\gamma \Wc_{m+j+2} - \frac{N-m-1}{N} \phi\gamma \Wc_{m+1}.
\end{align*}
}%

\begin{proposition}\label{prop:betagambc}
The $\Wbeta_i$, $\Wgamma_i$ are bosonic and the $\Wb_i$, $\Wc_i$ fermionic elements in $\tldcalD^\ch_{M, \hbar}(U_{(N)})$ for $i=1,\dots,N-1$. They satisfy the operator product expansions
\[
\Wbeta_i(z) \Wgamma_j(w) \sim - \delta_{ij} \hbar / (z-w)\quad\text{and}\quad\Wb_i(z) \Wc_j(w) \sim \delta_{ij} \hbar / (z-w),
\]
and all other operator product expansions between them vanish, i.e.\ they form an $\hbar$-adic $\beta\gamma bc$-system of rank $N-1$. Moreover, they have trivial operator product expansions with $\Tr(X)$, $\Tr(Y)$ and with $\phi \gamma$, $\beta \psi$.
\end{proposition}

By the above proposition, there is the following isomorphism of $\hbar$-adic vertex superalgebras
\begin{equation}\label{eq:factor-h-Wakimoto}
\tldcalD^\ch_{M, \hbar}(U_{(N)}) \simeq
\calD^\ch(T^* \C^{N-1})_{\hbar} \hatotimes \Cl_{\hbar}(T^* \C^{N-1})
\hatotimes \calD^\ch(T^* \C^1)_{\hbar} \hatotimes \SF_{\hbar},
\end{equation}
where $\calD^\ch(T^* \C^{N-1})_{\hbar}$ is the $\hbar$-adic $\beta\gamma$-system of rank $N-1$ generated by $\Wbeta_{i}$ and $\Wgamma_i$ for $i=1,\dots,N-1$, $\Cl_{\hbar}(T^* \C^{N-1})$ is the $\hbar$-adic $bc$-system of rank $N-1$ generated by $\Wb_{i}$ and $\Wc_i$ for $i=1,\dots,N-1$, $\calD^\ch(T^* \C^1)_{\hbar}$ is another $\hbar$-adic $\beta\gamma$-system generated by $\Tr(X)/\sqrt{N}$ and $\Tr(Y)/\sqrt{N}$ and $\SF_{\hbar}$ is the $\hbar$-adic symplectic fermion vertex superalgebra generated by $\phi \gamma/\sqrt{N}$ and $\beta \psi/\sqrt{N}$. That is, the sheaf restriction morphism (see the proof of \autoref{thm:Wakimoto} for the injectivity)
\[
\tldcalD^\ch_{M, \hbar}(M) \hookrightarrow \tldcalD^\ch_{M, \hbar}(U_{(N)})
\]
defines a free-field realisation.

Recall the conformal structure on $\tldcalD^\ch_{M, \hbar}(M)$ of central charge $c=-3N^2$. We extend it to the free-field vertex superalgebra $\tldcalD^\ch_{M, \hbar}(U_{(N)})$ via the above restriction morphism. Then, the generators have weights
\[
\wt(\Wbeta_m)=\frac{m+1}{2},\;\wt(\Wgamma_m)=\frac{-m+1}{2},\;
\wt(\Wb_m)=\frac{m+2}{2},\;\wt(\Wc_m)=-\frac{m}{2}
\]
for $m=1,\dots,N-1$. This shows that the $\beta\gamma$-system generated by $\Wbeta_m$, $\Wgamma_m$, which generally affords a one-parameter family of conformal structures, has a conformal structure of central charge $c=-1+3m^2$, while similarly the $bc$-system generated by $\Wb_m$, $\Wc_m$ has a conformal structure of central charge $c=-2-6m-3m^2$ for $m=1,\dots,N-1$. Moreover, we already saw that the $\beta\gamma$-system generated by $\Tr(X)/\sqrt{N}$, $\Tr(Y)/\sqrt{N}$ and the symplectic fermion vertex superalgebra have central charge $c=-1$ and $c=-2$, respectively, extending the above formulae for the central charges to $m=0$.

We also obtain an analogous free-field realisation of the vertex operator superalgebra $\sfV_{S_N}=[\tldcalD^\ch_{M,\hbar}(M)]^{\C^\times}$:
\begin{theorem}\label{thm:Wakimoto}
The restriction morphism $\tldcalD^\ch_{M, \hbar}(M) \hookrightarrow \tldcalD^\ch_{M, \hbar}(U_{(N)})$ induces a free-field realisation of $\sfV_{S_N}$, i.e.\ an embedding
\[
\sfV_{S_N}\hookrightarrow\calD^\ch(T^* \C^{N-1}) \otimes \Cl(T^* \C^{N-1}) \otimes \calD^\ch(T^* \C^1) \otimes \SF
\]
of vertex operator superalgebras.
\end{theorem}
The conformal structure on the right-hand side is the one with central charge $c=\sum_{m=0}^{N-1}((-1+3m^2)+(-2-6m-3m^2))=-3N^2$ described in the preceding paragraph.
\begin{proof}
First, we show that the restriction $\Res^{M}_{U_{(N)}}\colon \tldcalD^\ch_{M, \hbar}(M) \longrightarrow \tldcalD^\ch_{M, \hbar}(U_{(N)})$ is injective. Let $f \in \tldcalD^\ch_{M, \hbar}(M)$ be a global section such that $\Res^{M}_{U_{(N)}}(f) = 0$. We assume for the sake of contradiction that $f \not\in \hbar\tldcalD^\ch_{M, \hbar}(M)$. By the isomorphism $\tldcalD^\ch_{M, \hbar} / \hbar \tldcalD^\ch_{M, \hbar} \simeq \JetBundle{M}$ of \autoref{prop:quantization}, its $0$-th symbol $f_0 \coloneqq \sigma_0(f)$ is a non-zero global section of $\JetBundle{M}$.

Consider the grading \mbox{$\deg x_{ij (-n)} = \deg y_{ij (-n)} = \deg \beta_{i (-n)} = \deg \gamma_{i (-n)} = 0$} and $\deg \psi_{i (-n)} = \deg \phi_{i (-n)} = 1$ for all $i$, $j$ and $n$. This grading induces a filtration $\{F_r \JetBundle{M} \}_{r=0}^{\infty}$ of $\JetBundle{M}$ defined by $F_r \JetBundle{M} = \{ f \in \JetBundle{M} \;|\; \deg f \le r \}$. Let $\Gr_r \JetBundle{M}$ be the quotient sheaf $F_{r} \JetBundle{M} / F_{r+1} \JetBundle{M}$ for $r\in\Z_{\ge0}$. Then $\Gr_r \JetBundle{M}$ is a locally free $\calO_{M}$-module of infinite rank. The short exact sequence $0 \rightarrow F_{r+1} \JetBundle{M} \rightarrow F_{r} \JetBundle{M} \rightarrow \Gr_r \JetBundle{M} \rightarrow 0$
induces the commutative diagram
\[
\begin{tikzcd}
0 \arrow[r] & F_{r+1} \JetBundle{M}(M) \arrow[r] \arrow[d, "\Res^{M}_{U_{(N)}}" swap] & F_{r} \JetBundle{M}(M) \arrow[r, "\tilde{\sigma}_0"] \arrow[d, "\Res^{M}_{U_{(N)}}"] & \Gr_r \JetBundle{M}(M) \arrow[d, "\Res^{M}_{U_{(N)}}"]\\
0 \arrow[r] & F_{r+1} \JetBundle{M}(U_{(N)}) \arrow[r] & F_{r} \JetBundle{M}(U_{(N)}) \arrow[r, "\tilde{\sigma}_0"] & \Gr_r \JetBundle{M}(U_{(N)})
\end{tikzcd}
\]
where each row is exact. Let $g \in \Gr_r \JetBundle{M}$ be a global section and assume $\Res^M_{U_{(N)}}(g) = 0$. Since $\Gr_r \JetBundle{M}$ is a locally free $\calO_{M}$-module, the zero locus $\{p \in M \;|\; g(p) = 0 \}$ is a Zariski closed subset of $M$. By assumption, $g(p) = 0$ for any point $p$ of the Zariski open subset $U_{(N)}$, and hence $g(p) = 0$ for any point $p$ of the closure $\overline{U}_{(N)} = M$. Then, $g = 0$ and thus the restriction $\Res^M_{U_{(N)}}\colon\Gr_r \JetBundle{M}(M) \longrightarrow \Gr_r \JetBundle{M}(U_{(N)})$ is injective. Since $f_0 \in \JetBundle{M}(M)$ is a non-zero global section, there exists a unique $r\in\Z_{\ge0}$ such that $f_0 \in F_r \JetBundle{M}(M) \setminus F_{r+1} \JetBundle{M}(M)$. By assumption, $\Res^{M}_{U_{(N)}}(\tilde{\sigma}_0(f_0)) = \tilde{\sigma}_0(\Res^{M}_{U_{(N)}}(f_0)) = 0$,
and thus we obtain $\tilde{\sigma}_0(f_0) = 0$ by the above injectivity. This implies that $f_0 \in F_{r+1} \JetBundle{M}(M)$, contradicting the definition of $r$. By contradiction, we conclude that $\Res^{M}_{U_{(N)}}(f) = 0$ implies $f = 0$.

The image of the $\C[\hbar]$-submodule $\tldcalD^\ch_{M, \hbar}(M)_\fin$ with respect to the restriction homomorphism is clearly included in the $\C[\hbar]$-submodule of all superpolynomials in $\Wbeta_{i (-n)}$, $\Wgamma_{i (-n)}$, $\Wb_{i (-n)}$, $\Wc_{i (-n)}$, $\Tr(X)_{(-n)}$, $\Tr(Y)_{(-n)}$, $(\phi \gamma)_{(-n)}$, $(\beta \psi)_{(-n)}$ for $i=1,\dots,N-1$ and $n \in \Z_{>0}$. Taking the quotient by the ideal generated by $(\hbar - 1)$, these elements strongly generate the vertex superalgebra $\calD^\ch(T^* \C^{N-1}) \otimes \Cl(T^* \C^{N-1}) \otimes \calD^\ch(T^* \C^1) \otimes \SF$, and hence there is a homomorphism of vertex superalgebras from $\sfV_{S_N}$ to it.
\end{proof}

\begin{remark}
One can verify that the free-field realisation of the small $\calN=4$ superconformal algebra of level $k=-(N^2-1)/2$ obtained in \autoref{thm:Wakimoto} coincides for $N=2$ with the one given by Adamović in \cite{Adamovic16} and for $N=3$ with the one obtained by Bonetti, Meneghelli and Rastelli in \cite{BMR19}, Section~4.2.2, with parameter $\Lambda = 1/3$.
\end{remark}

\begin{remark}
The free-field realisation of \autoref{thm:Wakimoto} is an analogue of the Wakimoto realisation for affine vertex algebras studied in \cite{Wakimoto86, FF90, Frenkel05}.
\end{remark}

Finally, as an application of this free-field realisation, we discuss a certain factorisation of the $\hbar$-adic vertex superalgebra $\tldcalD^{\ch}_{M, \hbar}(M)$. Recall the $\hbar$-adic vertex operator subalgebra $\calD^{\ch}(T^*\C^1)_{\hbar} \hatotimes \SF_{\hbar}$ of $\tldcalD^{\ch}_{M, \hbar}(M)$ generated by $\Tr(X)$, $\Tr(Y)$, $\Lambda_1 = \phi \gamma/\sqrt{N}$ and $\Lambda_2 = \beta\psi/\sqrt{N}$. Let
\[
\tldcalD^{\ch}_{M, \hbar}(M)^{\perp}\coloneqq\Com_{\tldcalD^{\ch}_{M, \hbar}(M)}(\calD^{\ch}(T^* \C^1)_{\hbar} \hatotimes \SF_{\hbar})
\]
be the commutant vertex operator superalgebra of $\calD^{\ch}(T^* \C^1)_{\hbar} \hatotimes \SF_{\hbar}$ in $\tldcalD^{\ch}_{M, \hbar}(M)$. We show that the double commutant is again $\calD^{\ch}(T^* \C^1)_{\hbar} \hatotimes \SF_{\hbar}$, and not some extension of it. In other words:
\begin{proposition}\label{prop:h-factorization}
The following factorisation of the $\hbar$-adic vertex operator superalgebra $\tldcalD^{\ch}_{M, \hbar}(M)$ of global sections holds:
\[
\tldcalD^{\ch}_{M, \hbar}(M) \cong \tldcalD^{\ch}_{M, \hbar}(M)^{\perp} \hatotimes \calD^{\ch}(T^* \C^1)_{\hbar} \hatotimes \SF_{\hbar}.
\]
\end{proposition}
\begin{proof}
Recall the isomorphism \eqref{eq:factor-h-Wakimoto}. Consider any element of the form $\sum_{i} a_i \otimes b_i$ of $\tldcalD^{\ch}_{M, \hbar}(M) \subset \tldcalD^{\ch}_{M, \hbar}(U_{(N)})$, where $a_i \in \calD^{\ch}(T^* \C^{N-1})_{\hbar} \hatotimes \Cl_{\hbar}(T^* \C^{N-1})$ and $b_i \in \calD^{\ch}(T^* \C^1)_{\hbar} \hatotimes \SF_{\hbar}$. By induction on the conformal weight with respect to the conformal vector $\TTr + \TSF \in \calD^{\ch}(T^* \C^1)_{\hbar} \hatotimes \SF_{\hbar}$, the simplicity of $\calD^{\ch}(T^* \C^1)_{\hbar} \hatotimes \SF_{\hbar}$ implies that $a_i \otimes \bfone$ lies in $\tldcalD^{\ch}_{M, \hbar}(M) / (\hbar^m)$ for any $i$ and any positive integer $m$. Since $\tldcalD^{\ch}_{M, \hbar}(M)^{\perp} = \tldcalD^{\ch}_{M, \hbar}(M) \cap (\calD^{\ch}(T^* \C^{N-1}))_{\hbar} \hatotimes \Cl_{\hbar}(T^* \C^{N-1}))$, this implies $a_i \in \tldcalD^{\ch}_{M, \hbar}(M)^{\perp}$ for any $i$.
\end{proof}

Analogously, let
\[
\W_{S_N}\coloneqq\Com_{\sfV_{S_N}}(\calD^{\ch}(T^*\C^1) \otimes \SF)
\]
be the commutant of $\calD^{\ch}(T^*\C^1) \otimes \SF$ inside $\sfV_{S_N}$. Then \autoref{prop:h-factorization} implies the following factorisation of the vertex operator superalgebra
\[
\sfV_{S_N} = \W_{S_N} \otimes \calD^\ch(T^* \C^1) \otimes \SF,
\]
which allows us to split off the uninteresting tensor factor $\calD^{\ch}(T^*\C^1) \otimes \SF$. As a consequence of \autoref{thm:assoc-var}, \autoref{prop:small-N4-SCA}, \autoref{prop:CFT-type} and \autoref{thm:Wakimoto}, we conclude the following theorem, which summarises the properties of $\W_{S_N}$.
\begin{theorem}\label{thm:W-prop}
For $N\ge2$, the vertex subalgebra $\W_{S_N}\subset \sfV_{S_N}$ is a vertex operator superalgebra of CFT-type of central charge $c_{S_N}=-3(N^2-1)$ satisfying:
\begin{enumerate}
\item $\W_{S_N}$ is a conformal extension of some quotient of the universal small $\calN=4$ superconformal vertex superalgebra $\on{Vir}_{\calN=4}^{c_{S_N}}$.
\item The associated variety of $\W_{S_N}$ is $\calM_{S_N}=(V_{S_N} \oplus V_{S_N}^*) / S_N$, where $V_{S_N}=\C^{N-1}$ is the reflection representation of the symmetric group $S_N$. In particular, $\W_{S_N}$ is quasi-lisse.
\item There is a free-field realisation $\W_{S_N}\hookrightarrow\calD^\ch(T^* \C^{N-1}) \otimes \Cl(T^* \C^{N-1})$ in terms of $\rk(S_N)=N-1$ many copies of the $\beta\gamma b c$-system with central charges $c=-3(2p_i-1)$, where $p_i=2,\dots,N$ are the degrees of the fundamental invariants of $S_N$.
\end{enumerate}
\end{theorem}

It follows that $\W_{S_N}$ is the vertex operator superalgebra for the reflection group $S_N$ conjectured by Bonetti, Meneghelli and Rastelli \cite{BMR19}, i.e.\ the vertex superalgebra corresponding to the four-dimensional $\calN=4$ supersymmetric Yang-Mills theory $\on{SYM}_{\frsl_N}$ with gauge group $\SL_N$ via the 4D/2D duality \cite{BeeLemLie15}.

We remark that in the case of $N=1$, $\W_{S_N}$ reduces to the trivial vertex operator algebra $\C\bfone$.

\begin{remark}
The small $\calN=4$ superconformal vertex superalgebra $\on{Vir}_{\calN=4}^{c_{S_N}}$ has the group of outer automorphisms isomorphic to $\on{SL}_2(\C)$, which acts on the odd generators of conformal weight $3/2$ by $G^+ \mapsto a G^+ - b \tldG^+$, $\tldG^+ \mapsto - c G^+ + d \tldG^+$, $G^- \mapsto a G^- - b \tldG^-$ and $\tldG^- \mapsto - c G^- + d \tldG^-$ for $a,b,c,d \in \C$ with $ad - bc = 1$ and trivially on the other generators $J^{\pm}$, $J^0$ and $T_{\calN=4}$ (see \cite{CLR22}, cf.\ \cite{MI96}).

It can be hypothesised that this action lifts to an action on the vertex operator superalgebras $\sfW_{S_N}$ and $\sfV_{S_N}$ given by $\phi P \gamma \mapsto a (\phi P \gamma) + b (\beta P \psi)$, $\beta P \psi \mapsto c (\phi P \gamma) + d (\beta P \psi)$ and $\Tr P \mapsto \Tr P$ for $a,b,c,d \in \C$ with $ad - bc = 1$ and $P = P(X, Y) \in \C\langle X, Y\rangle$.
\end{remark}


\section{Characters and Schur Indices}
\label{sec:chars}

In the following, we determine the supercharacters of the vertex operator superalgebras $\sfV_{S_N}$ and $\W_{S_N}$, which coincide with the Schur indices of the four-dimensional $\calN=4$ supersymmetric Yang-Mills theories with gauge groups $\GL_N$ and $\SL_N$, respectively. Using the Euler-Poincaré principle and the exact integration formulae in \cite{PP22} we show that these characters are quasimodular forms (see, e.g., \cite{KZ95,Zag08}). Moreover, as was already noticed in \cite{KLS21,HO23}, it follows from the results in \cite{AR13} that they can be written as certain (generalised) multiple $q$-zeta values discovered over a century ago \cite{Mac21}. This also implies formulae for the expansion of these characters in terms of Eisenstein series \cite{Hua22}. Here, we present a different formulation communicated to us by Henrik Bachmann.

We comment that further results on quasimodularity and the appearance of multiple $q$-zeta values for vertex operator algebras in the context of the 4D/2D duality are obtained, e.g., in \cite{Mil22}.


\subsection{Euler-Poincaré Principle}

We recall the BRST cohomology sheaf on $M$, whose only non-vanishing degree component is $\tldcalD^\ch_{M, \hbar}=\calH^{\infty/2+0}_{\VA}(\frg, \tldcalD^\ch_{\frX, \hbar})$ of ghost degree~$0$. The vertex superalgebra $\sfV_{S_N}=[\tldcalD^\ch_{M, \hbar}(M)]^{\C^\times}$ is obtained from the global sections of this sheaf by taking the invariants under the torus action. The conformal vector $T$ from \autoref{sec:conformal} makes $\sfV_{S_N}$ a vertex operator superalgebra of central charge $c=-3N^2$ that is $\frac{1}{2}\Z_{\ge0}$-graded by weights, of CFT-type and whose (super)character is
\[
\on{(s)ch}_{\sfV_{S_N}}(q)=\Tr_{\sfV_{S_N}}(\pm1)^Pq^{T_{(1)}-c/24}=q^{N^2/8}(1+\calO(q^{1/2})),
\]
where $P$ is the usual parity operator with eigenvalues in $\Z/2\Z$. In the following, we determine this supercharacter and study its modular properties.

The global sections $\tldcalD^\ch_{M, \hbar}(M)$ are given by the cohomology of $d_{\VA} = (1/\hbar) Q_{(0)}$ on the $\hbar$-adic vertex superalgebra $C_{\VA}(\frX)\subset\tldC_{\VA}(\frX)$. The latter is the $\hbar$-adic free-field vertex superalgebra $\tldC_{\VA}(\frX) = \calD^\ch(T^* V)_{\hbar} \otimes \Cl_{\hbar}(T^* \C^N) \otimes \Cl_{\hbar}(T^* \frg)$. For the torus invariants this entails that
\[
\sfV_{S_N}=H^\bullet(C, Q_{(0)})=H^0(C, Q_{(0)})
\]
is the cohomology of $Q_{(0)}$ on the vertex superalgebra
\[
C \coloneqq \{ c \in \tldC \,|\, \tldE_{ij (0)} c = \Phi_{ij (0)} c = 0 \text{ for all } i, j = 1, \dots, N \} \subset \tldC
\]
given as the intersection of the kernels of $\tldE_{ij (0)}$ and $\Phi_{ij (0)}$ for $i,j=1,\dots,N$ on the free-field vertex operator superalgebra
\[
\tldC \coloneqq \calD^\ch(T^* V) \otimes \Cl(T^* \C^N) \otimes \Cl(T^* \frg),
\]
where $\calD^\ch(T^* V) = \C[x_{ij (-n)}, y_{ij (-n)}, \gamma_{i (-n)}, \beta_{i (-n)} \,|\, \substack{i,j=1, \dots, N \\ n = 1, 2, \dots}]$ is the $\beta\gamma$-system associated with the symplectic vector space $T^* V$, $V = \gEnd(\C^N) \oplus \C^N$, with generators $x_{ij}$, $y_{ij}$, $\gamma_i$ and $\beta_i$, $\Cl(T^* \C^N) = \Lambda_{\C}(\psi_{i (-n)}, \phi_{i (-n)} \,|\, \substack{i=1, \dots, N \\ n = 1, 2, \dots})$ is the $bc$-system with generators $\psi_i$ and $\phi_i$ and $\Cl(T^* \frg) = \Lambda_\C(\Psi_{ij (-n)}, \Phi_{ij (-n)} \,|\, \substack{i,j=1, \dots, N \\ n = 1, 2, \dots})$ is the (ghost) $bc$-system generated by $\Psi_{ij}$ and $\Phi_{ij}$. The $\Z$-grading by ghost degree on $\tldC$ is defined via $\deg(\Psi_{ij (-n)}) = 1$ and $\deg(\Phi_{ij (-n)}) = -1$.

For convenience we also introduce the intermediate kernel
\[
C \subset \bar{C} \coloneqq \{ c \in \tldC \,|\, \Phi_{ij (0)} c = 0 \text{ for all } i, j = 1, \dots, N \} \subset \tldC.
\]
Taking the kernels of the $\Phi_{ij (0)}$ in $\tldC$ simply amounts to removing all underived $\Psi$-fields from the tensor factor $\Cl(T^* \frg)$, i.e.
\[
\bar{C}=\calD^\ch(T^* V) \otimes \Cl(T^* \C^N)\otimes\Lambda_\C(\Psi_{ij (-n-1)}, \Phi_{ij (-n)} \,|\, \substack{i,j=1, \dots, N \\ n = 1, 2, \dots}).
\]

As an application of the Euler-Poincaré principle, the (super)character of $\sfV_{S_N}$ can be computed as
\begin{align*}
\on{(s)ch}_{\sfV_{S_N}}(q)&=\Tr_{\sfV_{S_N}}(\pm1)^Pq^{T_{(1)}-c/24}=\Tr_{H^0(C, Q_{(0)})}(\pm1)^{P_\text{m}}q^{T_{(1)}-c/24}\\
&=\sum_{i\in\Z}(-1)^i\Tr_{H^i(C, Q_{(0)})}(\pm1)^{P_\text{m}}q^{T_{(1)}-c/24}\\
&=\sum_{i\in\Z}(-1)^i\Tr_{C^i}(\pm1)^{P_\text{m}}q^{T_{(1)}-c/24}\\
&=\Tr_{C}(\pm1)^{P_\text{m}}(-1)^{P_\text{gh}}q^{T_{(1)}-c/24},
\end{align*}
where $P=P_\text{m}+P_\text{gh}$ is the parity operator on $\tldC$ and $P_\text{m}$, $P_\text{gh}$ the parity operators only acting on matter and ghost fields, respectively (so, $P_\text{gh}$ gives the ghost degree modulo 2). For the fourth equality we used the Euler-Poincaré principle and that $Q_{(0)}$ commutes with $T_{(1)}$ and $P_\text{m}$. For the supercharacter of $\sfV_{S_N}$ we then obtain the simple formula
\[
\on{sch}_{\sfV_{S_N}}(q)=\on{sch}_{C}(q).
\]
Hence, we need to compute the supercharacter of the kernel $C$ in the free-field vertex superalgebra $\bar{C}$.

To understand this intersection of the kernels of the $\tldE_{ij (0)} = (Q_{(0)} \Phi_{ij})_{(0)}$, recall from \autoref{lemma:sln-rep} that they form a representation of $\frg = \frgl_N$ on $\tldC$, which restricts to $\bar{C}$. We shall describe this representation in the following.

For any finite-dimensional representation $V$ of $\frg$ we modify the usual symmetric algebra $S(V)=\bigoplus_{k=0}^\infty S^k(V)$ by introducing a formal variable $x$,
\[
S_x(V)\coloneqq S(Vx)=\bigoplus_{k=0}^\infty S^k(V)x^k\subset S(V)[[x]]
\]
and analogously for the exterior algebra $\Lambda(V)$.

Let $\bar{C}=\bigoplus_{p,r,n}\bar{C}^{p,r}_n$ be the decomposition of $\bar{C}$ into simultaneous eigenspaces for $F_\text{m}$, $F_\text{gh}$ and $T_{(1)}$ with eigenvalues $p,r\in\Z$ and $n\in\frac{1}{2}\Z_{\ge0}$ respectively, where $F_\text{gh}$ denotes the ghost degree (or ghost fermion number operator) defined above by $F_\text{gh}\Psi_{ij}=\Psi_{ij}$ and $F_\text{gh}\Phi_{ij}=-\Phi_{ij}$ and $F_\text{m}$ is the analogous fermion number operator for the matter fields defined by $F_\text{m}\psi_i=\psi_i$ and $F_\text{m}\phi_i=-\phi_i$. Note that $(-1)^{F_\text{m}}=(-1)^{P_\text{m}}$ and $(-1)^{F_\text{gh}}=(-1)^{P_\text{gh}}$. Then we define
\[
\bar{C}_{y,z,q}=\bigoplus_{\substack{p,r\in\Z\\n\in\frac{1}{2}\Z_{\ge0}}}\bar{C}^{p,r}_ny^pz^rq^n\subset\bar{C}[[y^{\pm1},z^{\pm1},q^{1/2}]].
\]
It is not difficult to see:
\begin{lemma}\label{lem:tensorglN}
As a module for $\frgl_N$,
\begin{align*}
\bar{C}_{y,z,q}&\cong\bigotimes_{n=0}^\infty\Bigl(S_{q^{n+1/2}}(\frgl_N)\otimes S_{q^{n+1/2}}(\frgl_N)\otimes S_{q^{n+1/2}}(\C^N)\otimes S_{q^{n+1/2}}((\C^N)^*)\\
&\quad\otimes\Lambda_{yq^{n+1/2}}(\C^N)\otimes\Lambda_{y^{-1}q^{n+1/2}}((\C^N)^*)\otimes\Lambda_{zq^{n+1}}(\frgl_N)\otimes\Lambda_{z^{-1}q^{n+1}}(\frgl_N)\Bigr),
\end{align*}
where $\frgl_N$ denotes the (self-dual) adjoint representation and $\C^N$ and $(\C^N)^*$ the vector and dual vector representation, respectively.
\end{lemma}
The lemma also encodes the eigenvalues of the semisimple operator $y^{F_\text{m}}z^{F_\text{gh}}q^{T_{(1)}}$ on $\bar{C}$. For example, for any $n$, the $x_{ij (n)}\bfone$ for $i,j=1,\dots,N$ span an adjoint representation of $\frgl_N$ supported in $T_{(1)}$-weight $n+1/2$, amounting to the first tensor product in the above expression. Then, for the (super)character we shall take the trace over the specialisation $(\pm1)^{F_\text{m}}(-1)^{F_\text{gh}}q^{T_{(1)}}=(\pm1)^{P_\text{m}}(-1)^{P_\text{gh}}q^{T_{(1)}}$ of this operator.


\subsection{Quasimodularity}

By definition, the vertex superalgebra $C$ is nothing but the sum of all copies of the trivial representation of $\frgl_N$ in the tensor-product representation $\bar{C}$ from \autoref{lem:tensorglN}. The trivial representation may be selected by integrating over the Haar measure of the corresponding compact Lie group $\on{U}(N)$ with suitable normalisation.

Such integrals are well-studied in the physics literature in the context of computing the Schur index of four-dimensional $\calN=4$ supersymmetric Yang-Mills theories. We refer to \cite{BDF15,PP22,PWZ22,BSR22} in the case of the gauge groups $\GL_N$ and $\SL_N$. Indeed, the above integral coincides precisely with that given for $\GL_N$ in these references (in particular, see \cite{BDF15}, but take note of a slightly different convention here that leads to a factor of $q^{N^2/8}$ compared to \cite{BDF15}). This is of course to be expected, as the prediction from \cite{BeeLemLie15} is that the Schur index of the four-dimensional theory is recovered by the supercharacter of the corresponding vertex operator superalgebra.

Let $\eta(q)=q^{1/24}\prod_{n=0}^\infty(1-q^n)$ be the well-known Dedekind eta function, a holomorphic modular form of weight $1/2$ for some character. As an immediate corollary we obtain \cite{BDF15}:
\begin{proposition}\label{prop:qseries}
The supercharacters of $\sfV_{S_N}$ and $\W_{S_N}$ are
\begin{align*}
\on{sch}_{\sfV_{S_N}}(q)&=\frac{\eta(q)}{\eta(q^{1/2})^2}{\sum_{n=0}^\infty}(-1)^n\biggl(\binom{N+n}{N}+\binom{N+n-1}{N}\biggr)q^{(N+2n)^2/8},\\
\on{sch}_{\W_{S_N}}(q)&=\frac{1}{\eta(q)^3}{\sum_{n=0}^\infty}(-1)^n\biggl(\binom{N+n}{N}+\binom{N+n-1}{N}\biggr)q^{(N+2n)^2/8},
\end{align*}
and they coincide with the Schur indices of the four-dimensional $\calN=4$ supersymmetric Yang-Mills theories with gauge groups $\GL_N$ and $\SL_N$, respectively.
\end{proposition}
The two characters (or Schur indices) differ precisely by
\[
\frac{\eta(q)^2}{\eta(q^{1/2})^2}\eta(q)^2=\on{sch}_{\calD^\ch(T^* \C^1) \otimes \SF}(q),
\]
i.e.\ the above factorisation $\sfV_{S_N} = \W_{S_N} \otimes \calD^\ch(T^* \C^1) \otimes \SF$ corresponds to considering $\frgl_N$ rather than $\frsl_N$ in the quiver construction (or $\GL_N$ rather than $\SL_N$ on the level of the Schur indices).

Using the exact integration formulae in \cite{PP22}, we can study the modular properties of these characters. As usual, we consider $q=\e^{2\pi\ii\tau}$, where $\tau$ is from the complex upper half-plane $\mathbb{H}=\{z\in\C\,|\,\Im(z)>0\}$ equipped with an action of $\SL_2(\Z)=\langle S,T\rangle$. Let $\Gamma^0(2)=\langle T^2,STS\rangle$ denote the congruence subgroup of matrices whose upper right entry vanishes modulo~2.
\begin{proposition}[\cite{PP22}]\label{prop:quasimodular}
The supercharacter $\on{sch}_{\W_{S_N}}(q)$ converges to a sum of quasimodular forms of weights $N-1,N-3,\dots\in\Z_{\ge0}$ for the full modular group $\SL_2(\Z)$ if $N$ is odd and for $\Gamma^0(2)$ with some character if $N$ is even.
\end{proposition}
\begin{proof}
The supercharacter (or the corresponding Schur index) can be written as $\rk(\frsl_N)=N-1$ many iterated contour integrations. After each integration, one obtains an expression in terms of certain generalised Eisenstein series. Applying the precise integration formulae in \cite{PP22} to the Schur index for $\SL_N$, the statement follows by induction over $N$.
\end{proof}
We shall sketch an alternative proof of this statement in \autoref{sec:qMZV} based on certain multiple $q$-zeta values.

In particular, for odd $N$, the Fourier expansion of $\on{sch}_{\W_{S_N}}(q)=q^{(N^2-1)/8}+\dots$ contains only integral $q$-powers. For even $N$, on the other hand, the $q$-powers are in $\frac{3}{8}+\frac{1}{2}\Z$ or $\frac{7}{8}+\frac{1}{2}\Z$.

By multiplying by $\eta(q)^4/\eta(q^{1/2})^2$ we obtain the corresponding statement for the supercharacter $\on{sch}_{\sfV_{S_N}}(q)=q^{N^2/8}+\dots$ of $\sfV_{S_N}$, which is then a quasimodular form of weights $N$, $N-2$, $\dots$ for $\Gamma^0(2)$ with some character.

\begin{remark}
We comment on the appearance of quasimodular forms of mixed weight as supercharacters of the vertex superalgebras $\W_{S_N}$ (see, e.g., the discussion in the introduction of \cite{BSR22}).

Let $f\colon\mathbb{H}\longrightarrow\C$ be a (holomorphic) quasimodular form of weight~$k$ and depth~$p$, say, for the full modular group $\SL_2(\Z)$, but the following argument works similarly for $\Gamma^0(2)$. Recall that necessarily $p\leq k/2$. Then, by definition, applying the weight-$k$ Petersson slash operator $|_k M$ to $f$ for some $M\in\SL_2(\Z)$ will result in a function $f|_k M$ that is a polynomial of degree $p$ in $c/(c\tau+d)$ with quasimodular coefficients. If, however, we apply the weight-$w$ Petersson slash operator $|_w$ for any weight $w\leq k-p$, then $f|_w M$ is a polynomial of degree $k-w$ in $\tau$ or equivalently $\log(q)$ with quasimodular coefficients.

This argument applies in particular to $w=0$, and hence we can interpret any quasimodular form of mixed weight, say with some highest weight $k$, as sitting inside of a vector-valued function $F\colon\mathbb{H}\longrightarrow\C^d$ (which is not unique, of course) whose components are polynomials in $\tau$ of degree $k$ with quasimodular coefficients that transforms with weight~0 under $\SL_2(\Z)$.

This point of view is sensible when considering the supercharacter of $\W_{S_N}$. Indeed, $\W_{S_N}$ is quasi-lisse and hence satisfies a modular linear differential equation of weight~0 \cite{AraKaw18,Li23}. This is compatible with $\W_{S_N}$ being a component of a vector-valued quasimodular form on $\mathbb{H}$ with $\log(q)$-terms.
\end{remark}

We also point out that the characters $\on{sch}_{\W_{S_N}}(q)$ can be generalised by introducing additional variables so that they probably transform as quasi-Jacobi forms \cite{KM15}. The modular properties of these characters should reveal information about the representation theory of $\W_{S_N}$ and its free-field realisation in \autoref{thm:Wakimoto}. This is studied for the case $N=2$ in \cite{PP22,PWZ22}.


\subsection{Multiple \texorpdfstring{$q$}{q}-Zeta Values}
\label{sec:qMZV}

It is in principle straightforward to write down bases for the relevant (finite-dimensional) spaces of quasimodular forms in \autoref{prop:quasimodular} in terms of the (normalised) Eisenstein series
\[
E_{2k}(q)=1-\frac{4k}{B_{2k}}\sum_{n=0}^\infty\sigma_{2k-1}(n)q^n
\]
for $k\in\Z_{>0}$, with Bernoulli numbers $B_{2k}$. These are holomorphic quasimodular forms (in fact, modular for $k\geq2$) of weight $2k$ for the modular group $\SL_2(\Z)$. For even $N$, one also needs the scaled versions $E_{2k}(q^{1/2})$.

Then one can expand the characters in \autoref{prop:qseries} in these bases for small~$N$. For instance, for $N=2$, $3$ the characters are given by
\begin{align*}
\on{sch}_{\W_{S_2}}(q)&=\frac{\eta(q^{1/2})^2}{\eta(q)^4}(E_2(q)-E_2(q^{1/2}))/24\\
&=q^{3/8}(1+3q-4q^{3/2}+9q^2-12q^{5/2}+22q^3+\dots),\\
\on{sch}_{\W_{S_3}}(q)&=\frac{1}{24}(1-E_2(q))\\
&=q(1 + 3 q + 4 q^2 + 7 q^3 + 6 q^4 + 12 q^5 + 8 q^6+\dots).
\end{align*}
See \cite{PP22,BSR22,HO23} for expressions up to $N\approx 10$.

In the following, we explain how a general formula for arbitrary $N$ can be proved in the context of multiple $q$-zeta values. More precisely, we give a generating function for the coefficients of the expansion of the characters in terms of Eisenstein series. This was done in \cite{Hua22}, confirming a conjecture in \cite{PP22}. Here, we present a version communicated to us by Henrik Bachmann.

\medskip

In \cite{Mac21}, MacMahon introduced the following $q$-analogues of generalised divisor sums (all summation indices being integers):
\begin{align*}
A_k(q)&=\sum_{0<m_1<\dots<m_k}\frac{q^{m_1+\dots+m_k}}{(1-q^{m_1})^2\dots(1-q^{m_k})^2},\\
C_k(q)&=\sum_{0<m_1<\dots<m_k}\frac{q^{2m_1+\dots+2m_k-k}}{(1-q^{2m_1-1})^2\dots(1-q^{2m_k-1})^2}
\end{align*}
for $k\in\Z_{>0}$. $A_k(q)$ is a well-known example of a multiple $q$-zeta value (see, e.g., \cite{Bra05,Zha07}). In general (in one of several versions), these multiple $q$-zeta values take the form
\[
\zeta_q(a_1,\dots,a_k)=\sum_{0<m_1<\dots<m_k}\frac{q^{(a_1-1)m_1+\dots+(a_k-1)m_k}}{(1-q^{m_1})^{a_1}\dots(1-q^{m_k})^{a_k}}
\]
for $a_i\in\Z_{>0}$ with $a_k\geq2$. By design, they tend to usual multiple zeta values in the limit $q\to1$, after multiplying with $(1-q)^{a_1+...+a_k}$. It is apparent that one obtains $A_k(q)=\zeta_q(2,\dots,2)$ as very symmetric special case for $a_1=\dots=a_k=2$. $C_k(q)$ is not a multiple $q$-zeta value, but in some sense an odd version thereof \cite{Hof19}.

The $q$-expansions of $A_k(q)$ and $C_k(q)$ are computed in \cite{AR13}. By comparing these $q$-expansions with those in \autoref{prop:qseries}, it follows immediately that
\[
\on{sch}_{\W_{S_N}}(q)=\begin{cases}
A_{(N-1)/2}(q),& N\text{ odd},\\
\frac{\eta(q^{1/2})^2}{\eta(q)^4}C_{N/2}(q^{1/2}),& N\text{ even},
\end{cases}
\]
as was already noticed in \cite{KLS21,HO23}.

Moreover, the authors of \cite{AR13} prove the recurrence relations
\begin{align*}
A_k(q)&=\frac{1}{(2k+1)2k}\Big(\big(6A_1(q)+k(k-1)\big)A_{k-1}(q)-2q\frac{d}{dq}A_{k-1}(q)\Big),\\
C_k(q)&=\frac{1}{2k(2k-1)}\Big(\big(2C_1(q)+(k-1)^2\big)C_{k-1}(q)-q\frac{d}{dq}C_{k-1}(q)\Big),
\end{align*}
which together with $A_1(q)=(1-E_2(q))/24$ and $C_1(q)=(E_2(q^2)-E_2(q))/24$ can be used to prove the quasimodularity statements in \autoref{prop:quasimodular} in a somewhat independent way.

Moreover, these recursion relations can be used to prove certain formulae for the expansions of $A_k(q)$ and $C_k(q)$ in terms of Eisenstein series. Such expansions for $\on{sch}_{\W_{S_N}}(q)$ were conjectured in \cite{PP22} and proved in \cite{Hua22}. Here, we state a different formulation due to Henrik Bachmann. Indeed, using the formulae for the Fourier expansions of the multiple Eisenstein series $G_{2,\dots,2}$ in \cite{Bac20} one can obtain the following expressions for the generating series of $A_k(q)$ and $C_k(q)$, which can then be used to obtain explicit (non-recursive) expressions for $A_k(q)$ and $C_k(q)$ in terms of Eisenstein series:
\begin{align*}
1+\sum_{k=1}^\infty A_k(q)x^{2k}&=\frac{2}{x}\arcsin\Bigl(\frac{x}{2}\Bigr)\exp\biggl(\sum_{j=1}^\infty\frac{(-1)^{j-1}}{j}\tilde{G}_{2j}(q)\Bigl(2\arcsin\Bigl(\frac{x}{2}\Bigr)\Bigr)^{2j}\biggr),\\
1+\sum_{k=1}^\infty C_k(q)x^{2k}&=\exp\biggl(\sum_{j=1}^\infty\frac{(-1)^{j-1}}{j}\bigl(\tilde{G}_{2j}(q)-\tilde{G}_{2j}(q^2)\bigr)\Bigl(2\arcsin\Bigl(\frac{x}{2}\Bigr)\Bigr)^{2j}\biggr).
\end{align*}
Here $\tilde{G}_{2k}(q)=-\frac{B_{2k}}{2(2k)!}E_{2k}(q)$ for $k\in\Z_{>0}$ denotes the Eisenstein series in a different normalisation.


\section{\texorpdfstring{$N=2$}{N=2} Case: Twisted Chiral de Rham Algebra on \texorpdfstring{$\bbP^1$}{ℙ¹}}
\label{sec:n=2-case}

\newcommand{\Ix}{0}
\newcommand{\Iy}{\infty}
\newcommand{\cX}{\bar{x}}
\newcommand{\cY}{\bar{y}}
\newcommand{\cDx}{\bar{\xi}}
\newcommand{\cDy}{\bar{\eta}}
\newcommand{\qX}{x}
\newcommand{\qY}{y}
\newcommand{\qDx}{\xi}
\newcommand{\qDy}{\eta}
\newcommand{\qBx}{b^{\Ix}}
\newcommand{\qCx}{c^{\Ix}}
\newcommand{\qBy}{b^{\Iy}}
\newcommand{\qCy}{c^{\Iy}}

In this section, we study the easiest non-trivial case of $N=2$, i.e.\ the sheaf of $\hbar$-adic vertex superalgebras $\tldcalD^\ch_{M, \hbar}$ on the Hilbert scheme $M = \Hilb^2(\C^2)$. We see that this sheaf essentially coincides with the twisted chiral de Rham algebra on the projective line $\bbP^1$ with parameter $\alpha = 1/2$, introduced in \cite{GMS05}. Moreover, the vertex operator superalgebra $\W_{S_2}$ of global sections is equal to the simple quotient of $\on{Vir}_{\calN=4}^{-9}$ and the free-field realisation in terms of the $\beta\gamma bc$-system coincides with the one given in \cite{Adamovic16}.

\medskip

First, we discuss the local coordinates of the Hilbert scheme $M = \Hilb^2(\C^2)$. In this section, we write $U_{\Ix} = U_{(2)}$ and $U_{\Iy} = U_{(1^2)}$ so that there is an affine open covering $M = U_{\Ix} \cup U_{\Iy}$. Recall the local coordinates discussed in \autoref{sec:big-cell}. There are the functions $[X^2 : 1]_{\Ix}$, $[X^2 : X]_{\Ix}$, $[Y : 1]_{\Ix}$, $[Y : X]_{\Ix}$ defined over $U_{\Ix}$, and $[Y^2 : 1]_{\Iy}$, $[Y^2 : Y]_{\Iy}$, $[X : 1]_{\Iy}$, $[X : Y]_{\Iy}$ defined over $U_{\Iy}$. The traces $\Tr(X)$, $\Tr(Y)$ are functions defined over $M$ and are written in the local coordinates as
\begin{align*}
\Tr(X) &= [X^2 : X]_{\Ix} = 2 [X : 1]_{\Iy} + [X : Y]_{\Iy} [Y^2 : Y]_{\Iy}, \\
\Tr(Y) &= 2 [Y : 1]_{\Ix} + [Y : X]_{\Ix} [X^2 : X]_{\Ix} = [Y^2 : Y]_{\Iy}.
\end{align*}
We define
\begin{align*}
\cX &= [Y : X]_{\Ix}, &\cDx &= - \Tr(X^2) + (1/2) \Tr(X)^2 = - 2 [X^2 : 1]_{\Ix} - (1/2) \Tr(X)^2, \\
\cY &= [X : Y]_{\Iy}, &\cDy &= \Tr(Y^2) - (1/2) \Tr(Y)^2 = 2 [Y^2 : 1]_{\Iy} + (1/2) \Tr(Y)^2.
\end{align*}
Then, $(\cX, \cDx, \Tr(X), \Tr(Y))$ and $(\cY, \cDy, \Tr(X), \Tr(Y))$ give local coordinates for $M$ over $U_{\Ix}$ and $U_{\Iy}$, respectively. The transition map between these local coordinates can be seen to be $\cY = {1}/{\cX}$, $\cDy = - \cX^2 \cDx$. This implies the well-known fact that the Hilbert scheme $M$ is isomorphic to the direct product
\[
M = \Hilb^2(\C^2) \simeq T^* \bbP^1 \times T^* \C^1.
\]
It is not difficult to see that this isomorphism is an isomorphism of holomorphic symplectic manifolds, not just one of algebraic varieties.

\medskip

Now we discuss the structure of the sheaf $\tldcalD^\ch_{M, \hbar}$ with respect to the above local coordinates. As discussed in \autoref{sec:Wakimoto}, the following local trivialisations hold:
\begin{align*}
\tldcalD^\ch_{M, \hbar}(U_{\Ix}) &= \C[[\hbar]][\qX_{(-n)}, \qDx_{(-n)} \,|\, n = 1, 2, \dots] \hatotimes \Lambda_{\C[[\hbar]]}(\qBx_{(-n)}, \qCx_{(-n)} \,|\, n = 1, 2, \dots) \\
&\quad \hatotimes \calD^\ch(T^* \C^1)_{\hbar} \hatotimes \SF_{\hbar}, \\
\tldcalD^\ch_{M, \hbar}(U_{\Iy}) &= \C[[\hbar]][\qY_{(-n)}, \qDy_{(-n)} \,|\, n = 1, 2, \dots] \hatotimes \Lambda_{\C[[\hbar]]}(\qBy_{(-n)}, \qCy_{(-n)} \,|\, n = 1, 2, \dots) \\
&\quad \hatotimes \calD^\ch(T^* \C^1)_{\hbar} \hatotimes \SF_{\hbar}
\end{align*}
with
\begin{alignat*}{2}
\qX &= [Y : X]_{\Ix},\qquad & \qDx &= - \Tr(X^2) + (1/2) \Tr(X)^2,\\
\qY &= [X : Y]_{\Iy},\qquad & \qDy &= \Tr(Y^2) - (1/2) \Tr(Y)^2
\end{alignat*}
and
\begin{alignat*}{2}
\qBx &= \phi X \gamma - (1/2) \Tr(X) \phi \gamma,\qquad & \qCx &= \{\psi : X \gamma\}_{\Ix},\\
\qBy &= \phi Y \gamma - (1/2) \Tr(Y) \phi \gamma,\qquad & \qCy &= \{\psi : Y \gamma\}_{\Iy}
\end{alignat*}
and where $\calD^\ch(T^* \C^1)_{\hbar}$ and $\SF_{\hbar}$ are generated by $\Tr(X)/\sqrt{2}$, $\Tr(Y)/\sqrt{2}$ and $\phi \gamma/\sqrt{2}$, $\beta \psi/\sqrt{2}$, respectively. It is not difficult to verify that
\begin{gather*}
\qBy = \qX_{(-1)} \qBx, \qquad \qCy = (1/\qX)_{(-1)} \qCx,
\end{gather*}
and thus $\qBy_{(-1)} \qCy = \qBx_{(-1)} \qCx$. The global section $J^0 = \Tr(XY) - (1/2) \Tr(X) \Tr(Y)$ can be written in terms of the above elements as
\begin{equation}\label{eq:2}
J^0 = - 2 \qX_{(-1)} \qDx + \qBx_{(-1)} \qCx = 2 \qY_{(-1)} \qDy - \qBy_{(-1)} \qCy.
\end{equation}
Combining this identity with $\qY = 1/\qX$ and $\qBy_{(-1)} \qCy = \qBx_{(-1)} \qCx$, we obtain that $\qDy = - \qX_{(-1)}^2 \qDy + \qX_{(-1)} \qBx_{(-1)} \qCx - (3/2) \hbar \qX_{(-2)} \bfone$. The above description of the transition homomorphism implies that the sections $(\qX, \qDx, \qBx, \qCx)$ defined over $U_{\Ix}$ and the sections $(\qY, \qDy, \qBy, \qCy)$ over $U_{\Iy}$ form the $\hbar$-adic version of the twisted chiral de Rham algebra $\Omega^\ch_{\bbP^1, \alpha}$ on $\bbP^1$ for $\alpha=1/2$ introduced by Gorbounov, Malikov and Schechtman in \cite{GMS05}. On the other hand, there are the bosonic global section $\Tr(X)$ and $\Tr(Y)$ and the fermionic ones $\phi \gamma$ and $\beta \psi$ in $\tldcalD^\ch_{M, \hbar}(M)$.

\begin{proposition}\label{prop:chiral-de-Rham}
There is an isomorphism of sheaves of $\hbar$-adic vertex superalgebras over
$M = \Hilb^2(\C^2) \simeq T^* \bbP^1 \times T^* \C^1$,
\[
\tldcalD^\ch_{M, \hbar} \simeq \Omega_{T^* \bbP^1, 1/2,\hbar}^\ch \hatotimes \calD^\ch(T^*\C^1)_{\hbar} \hatotimes
\SF_{\hbar},
\]
where $\Omega^\ch_{T^* \bbP^1, 1/2,\hbar}$ is a certain $\hbar$-adic analogue of the twisted chiral de Rham algebra $\Omega^\ch_{\bbP^1, 1/2}$ defined as a sheaf of $\hbar$-adic vertex superalgebras on $T^* \bbP^1$, $\calD^\ch(T^* \C^1)_{\hbar}$ is the $\hbar$-adic $\beta\gamma$-system generated by $\Tr(X)/\sqrt{2}$, $\Tr(Y)/\sqrt{2}$ and $\SF_{\hbar}$ is the $\hbar$-adic symplectic fermion vertex superalgebra generated by $\phi \gamma/\sqrt{2}$, $\beta \psi/\sqrt{2}$.
\end{proposition}

Finally, we describe the global sections of $\tldcalD^\ch_{M, \hbar}$ in the case of $N=2$. The elements $J^+ = \qDx$, $J^0$ defined by \eqref{eq:2}, $J^{-} = \qDy$, $G^+ = \qBx$, $G^{-} = \qBy,\dots$ form the quotient $V_{\calN=4}$ of the universal small $\calN = 4$ superconformal algebra. The description of these global sections in the sections $(\qX, \qDx, \qBx, \qCx)$ defined over $U_{\Ix}$ gives a free-field realisation of the superconformal algebra that coincides with the free-field realisation of the simple quotient introduced by Adamović in \cite{Adamovic16}. The main consequence of this is, that for $N=2$, $V_{\calN=4}$ is the simple quotient of $\on{Vir}_{\calN=4}^{-9}$.

Moreover, comparing the global sections in \cite{GMS05} with the free-field realisation in \cite{Adamovic16} shows that $\W_{S_2}$ is actually equal to $V_{\calN=4}$, rather than a conformal extension of it, as would happen for $N>2$.
\begin{proposition}
The vertex operator superalgebra $\sfV_{S_2}$ of global sections coincides with the tensor product
\[
\sfV_{S_2}\cong V_{\calN=4} \otimes \calD^\ch(T^* \C^1) \otimes \SF,
\]
where $V_{\calN=4}=\W_{S_2}$ is the simple quotient of the small $\calN=4$ superconformal algebra $\on{Vir}_{\calN=4}^{-9}$ of level $k=-3/2$. In particular, $\W_{S_2}$ and $\sfV_{S_2}$ are simple.
\end{proposition}
We are led to believe that the vertex operator algebras of global sections $\sfV_{S_N}$, or equivalently $\W_{S_N}$, and the quotient $V_{\calN=4}$ of the small $\calN=4$ superconformal algebra are simple for all $N\ge2$.


\bibliographystyle{alpha_noseriescomma}
\bibliography{refs_new2}

\end{document}